\documentclass[reqno,12pt]{amsart}
\usepackage{amsmath, amssymb, amsthm, epsfig}

\usepackage{color}
\usepackage[inner=1.0in,outer=1.0in,bottom=1.0in, top=1.0in]{geometry}

\usepackage{hyperref, latexsym}
\usepackage{url}

\usepackage{color}

 \usepackage{ulem}

\def\today{\ifcase\month\or
  January\or February\or March\or April\or May\or June\or
  July\or August\or September\or October\or November\or December\fi
  \space\number\day, \number\year}

 \newtheorem{theorem}{Theorem}[section]

 \newtheorem{lemma}[theorem]{Lemma}
 
 \newtheorem{proposition}[theorem]{Proposition}
 
 \theoremstyle{definition}

 \theoremstyle{remark}
 
 \newcommand{\mc}{\mathcal}
\newcommand{\sumtwo}{\operatorname*{\sum\sum}}
\newcommand{\sumthree}{\operatorname*{\sum\sum\sum}}

\newcommand{\summany}{\operatorname*{\sum...\sum}}

\newcommand{\m}{\frak{m}}
\newcommand{\n}{\frak{n}}

 \newcommand{\Scal}{\mathcal{S}}
 \newcommand{\C}{\mathbb{C}}
 \newcommand{\R}{\mathbb{R}}

 \newcommand{\W}{\mathcal{W}} 
 
\newcommand{\Imp}{\textrm{Im}}
\newcommand{\Rep}{\textrm{Re}}
\newcommand{\hW}{\widetilde{\W}}
\newcommand{\FW}{  \frac{\W (q/Q)}{ \varphi(q)} }

\newcommand{\emts}{e^{-t^2}}

\newcommand{\sumstar}{\sideset{}{^*}\sum}
\newcommand{\sumsh}{\sideset{}{^\sharp}\sum}
\newcommand{\shortmod}{\ensuremath{\negthickspace \negthickspace
\negthickspace \pmod}}
\newcommand{\smod}{\ensuremath{\negthickspace \negthickspace
\negthickspace \pmod}}

\newcommand{\es}[1]{\begin{equation}\begin{split}#1\end{split}\end{equation}}
\newcommand{\est}[1]{\begin{equation*}\begin{split}#1\end{split}\end{equation*}}

 \newcommand{\pr}[1]{\left( #1\right)}

\newcommand{\vecx}{\boldsymbol{x}}

\newcommand{\vecu}{\boldsymbol{u}}

\newcommand{\vecd}{\boldsymbol{d}}

\begin{document}

\title{$n$-level density of the low-lying zeros of primitive Dirichlet $L$-functions} 

\date{\today}

\author[V. Chandee]{Vorrapan Chandee}
\address{Mathematics Department \\ Kansas State University \\ Manhattan, KS 66503}
\email{chandee@ksu.edu}

\author[Y. Lee]{Yoonbok Lee}
\address{Department of Mathematics \\ Research Institute of Basic Sciences \\ Incheon National University \\ Incheon, 22012 \\ Korea}
\email{leeyb@inu.ac.kr, leeyb131@gmail.com}

\begin{abstract}
Katz and Sarnak conjectured that the statistics of low-lying zeros of various family of $L$-functions matched with the scaling limit of eigenvalues from the random matrix theory. In this paper we confirm this statistic for a family of primitive Dirichlet $L$-functions matches up with corresponding statistic in the random unitary ensemble, in a range that includes the off-diagonal contribution. To estimate the $n$-level density of zeros of the $L$-functions, we use the asymptotic large sieve method developed by Conrey, Iwaniec and Soundararajan. For the random matrix side, a formula from Conrey and Snaith allows us to solve the matchup problem.
\end{abstract}

\allowdisplaybreaks
\numberwithin{equation}{section}

\maketitle

\section{Introduction} \label{sec:intro}
Efforts to understand the location of zeros of the Riemann zeta function have played an important role in the development of analytic number theory. Classically, information about the horizontal distribution of these zeros yielded better understanding about the distribution of prime numbers. Moreover, Montgomery \cite{Montgomery} calculated statistics of the spacings of zeros along the vertical line; more specifically, he examined the so called pair-correlation function, 
which is a quantity roughly of the form
$$ \frac{1}{N(T)} \sum_{\substack{0 < \gamma, \gamma' \leq T}} f\left((\gamma - \gamma') \frac{\log T}{2\pi}\right),$$
where under the Riemann hypothesis (RH), $1/2 + i \gamma$ are non-trivial zeros of the Riemann zeta function, $N(T)$ is the number of zeros such that $0 < \gamma \leq T$, and $f$ is a Schwartz function on $\mathbb R$. Then assuming the Fourier transform $\hat f$ is supported in $(-1, 1)$, he showed that as $T \rightarrow \infty$
\es{ \label{Mont:f}  \frac{1}{N(T)}\sum_{\substack{0 < \gamma, \gamma' \leq T \\ \gamma \neq \gamma'}} f\left((\gamma - \gamma') \frac{\log T}{2\pi}\right)  \rightarrow  \int_{-\infty}^{\infty} f(x) W^{(2)}(x) \> dx,}
where $W^{(2)}(x) =  1- \left(\frac{\sin \pi x}{\pi x}\right)^2$. Equation \eqref{Mont:f} is expected to be true for any Schwartz function, and this is the Pair Correlation conjecture.  Dyson later pointed out to Montgomery that the factor $W^{(2)}(x)$ is the same as the distribution of the spacings of eigenvalues of the Gaussian unitary ensemble (GUE) distribution from random matrix theory, which forshadowed a great deal of work later.  Indeed, the link between the Riemann zeta function and random matrix theory has led to a better understanding of both moments and zeros of $L$-functions (see  for example \cite{KatzSarnak}, \cite{KeatingSnaith} and \cite{RudnickSarnak}).

\"{O}zl\"{u}k \cite{Ozluk} studied a $q$-analogue of Montgomery's pair correlation result under the Generalized Riemann hypothesis (GRH) for Dirichlet $L$-functions. In particular, he considered the pair correlation function of a family of Dirichlet $L$-functions averaging over characters $\chi$ modulo $q$, where $q \in [Q, 2Q]$.
The large size of the family ($\sim Q^2$) compared to the conductor ($\sim Q$) allows for an extension of the support of the Fourier transform of the test function beyond what is readily available.  In this undertaking, \"{O}zl\"{u}k dealt with the contribution of certain off diagonal terms, and he was able to succeed with the extra average over the modulus. 
Recently, the authors in joint work with Liu and Radziwi\l\l  \cite{CLLR} revisited \"Ozl\"uk's pair correlation function but averaging over primitive characters instead, using an asymptotic large sieve introduced by Conrey, Iwaniec and Soundararajan \cite{ConreyIwaniecSoundararajan}. As a result, we improved the proportion of simple zeros of primitive Dirichlet $L$-functions.

The pair correlation conjecture has been extended to $n$-level correlation of the zeros of the Riemann zeta function through random matrix theory, which studies statistics involving $n$-tuples of zeros. In support of the conjecture, Rudnick and Sarnak \cite{RudnickSarnak} proved the result for some special test functions $f$. To describe their results more precisely, assuming RH, let $1/2 + i \gamma_j$ be nontrivial zeros of the Riemann zeta function. 
Rudnick and Sarnak studied the sum of the the form
\est{ R(T; f, h)  = \summany_{ \substack{j_1,..., j_n \\ \textrm{all distinct} }} h \pr{\frac{\gamma_{j_1}}{T}} \cdots  h \pr{\frac{\gamma_{j_n}}{T}} f\pr{\gamma_{j_1} L, \dots ,\gamma_{j_n}L}, }
where $L = \frac{\log T}{2\pi}$ and $h$ is a rapidly decaying cut-off function. We define the $n$-level correlation density for the GUE model as
\es{\label{def:Wn} W^{(n)} (\vecx) := W^{(n)} ( x_1 , \dots, x_n ) := \det ( K_0 ( x_j, x_k ))_{j,k} ,}
 where
$$ K_0 ( x,y) = \frac{ \sin( \pi ( x-y))}{ \pi (x-y)} . $$   
Then assuming the condition 
$$ \mathrm{supp} \hat{f} \subset \{   |\xi_1 | + \cdots + |\xi_n| <2 \} $$
and a couple of other technical conditions omitted for now, Rudnick and Sarnak showed that
$$R(T; f, h) \sim N(T) \left(  \int_{-\infty}^{\infty} h(r)^n  \> dr \right) \int_{\mathbb R^n} f(x) W^{(n)}(\vecx) \delta\pr{\frac{1}{n}(x_1 + \cdots + x_n)} \> d\vecx,
$$
where $\delta$ is the Dirac-delta function.
This result essentially reduces to \eqref{Mont:f} when $n=2$. To deal with the sum over non-trivial zeros appearing in $R(T; f, h)$, they applied the explicit formula, which connects this sum over zeros to a sum over prime powers, basically of the form
$$ \summany_{\substack{n_1,..., n_r, m_1,..., m_s \\ \textrm{all prime powers}}} \frac{c(n_1) \cdots c(n_r) \overline{c(m_1)\cdots c(m_s)}}{\sqrt{n_1 \cdots  n_r m_1 \cdots m_s}} A({\bf n, m}, T),$$
where the factor $A({\bf n, m}, T)$ contains terms involving the Fourier transform of $f$. 
The restriction of the support of the Fourier transfrom of $f$ is required so that the contribution from the off diagonal terms $n_1 \cdots n_r \neq m_1 \cdots m_s$ can be ignored. Although it is not hard to evaluate the diagonal terms $n_1\cdots n_r = m_1 \cdots m_s$, it was still a challenge to verify that their answers agree with the conjecture arising from the random matrix theory. Rudnick and Sarnak went through complicated combinatorial arguments involving random walks. Later, Conrey and Snaith presented a new formula for $n$-correlation from the random matrix theory side in \cite{ConreySnaith} and applied it in \cite{ConreySnaith2} to straightforwardly match results from both sides. Although this formula looks more intricate than the determinant form in \eqref{def:Wn}, it expresses the answer in terms of a test function, where the Fourier transform is supported in any range, and this allows one to naturally match answers from the number theory side.
  
In analogy with the Pair Correlation conjecture, we expect Rudnick and Sarnak's result above to hold without any condition on the support of the Fourier transform of $f$, where the off-diagonal terms also contribute.  It is worth noting that this type of conjecture is quite powerful and appears currently intractable.  In particular, Montgomery's original Pair Correlation conjecture easily implies that there are infinitely many pairs of zeros of $\zeta(s)$ which are far less than the average spacing apart, and this has deep consequences towards Siegel zeros.  Typically, even extending the support of the Fourier transform beyond what is currently available is a challenging problem.

Katz and Sarnak \cite[appendix]{KatzSarnak} computed the $n$-level density of eigenvalues of various random matrix ensembles and conjectured that the statistics of low-lying zeros of various families of $L$-functions is the same as the corresponding one from the random matrix theory. Rubinstein \cite{Rubinstein} studied a family of quadratic Dirichlet $L$-functions and proved that the $n$-level density for the family matched with the one for symplectic unitary ensemble in a certain range. Later Gao \cite{Gao} doubled the allowable range of the support of the Fourier transform of the test function, but he was not able to prove that his answer matched the conjecture from random matrix theory. This was then resolved by Entin, Roditty-Gershon and Rudnick through zeta functions over function fields \cite{ERRudnick}. Recently, Mason and Snaith \cite{MasonSnaith} presented an alternative proof of this result using a new formula for  $n$-level densities of the random symplectic ensemble, analogous to the work of Conrey and Snaith in \cite{ConreySnaith} and \cite{ConreySnaith2}.
 
 While only a symplectic family is considered in \cite{ERRudnick}, \cite{Gao} and  \cite{Rubinstein}, we consider a family of primitive Dirichlet $L$-functions, which is a unitary case.  To be more precise,  let $\chi$ be a primitive Dirichlet character modulo $q>1$, and  a Dirichlet $L$-functions associated to it is defined to be 
$$ L(s, \chi) = \sum_{n = 1}^\infty \frac{\chi(n)}{n^s}$$
for Re$(s) > 1.$ Throughout this paper, we assume GRH for the Dirichlet $L$-function $L(s, \chi)$ and write its nontrivial zeros as $ \tfrac12 + i \gamma^\chi_j$, $ j = \pm 1, \pm 2 , \dots$, where
$$   \cdots \leq \gamma_{-3}^\chi \leq \gamma_{-2}^\chi \leq \gamma_{-1}^\chi < 0 \leq \gamma_{1}^\chi \leq \gamma_{2}^\chi \leq \cdots. $$
We say that a function $ f : \R^n \to \R$ has the C4-Property provided that  
\begin{description}
	\item[P1] There exist even functions $f_i : \R \to \R$ for $ i \leq n $ such that each function $f_i$ has a Fourier transform $\hat{f}_i(u) := \int_\R f_i(x) e^{ 2 \pi i x u} dx $ with a support contained in an interval $[-\eta_i, \eta_i ]$ and  
	$$ f(\vecx) = f(x_1 , \dots, x_n ) = \prod_{i=1}^n f_i (x_i) . $$
	\item[P2]  $ \eta := \sum_{i =1}^n \eta_i < 4$ and $\varepsilon := 4 - \eta>0$.  
\end{description}
We define the $n$-level density function  by
$$\mathcal{L}_0 (f, \W, Q) =     \sum_{q} \frac{\W (q/Q)}{ \varphi(q)} \sumstar_{\chi \smod q} \ \sumsh_{j_1, ..., j_n}  f(\mathcal U \gamma_{j_1}^\chi  ,  \dots, \mathcal U  \gamma_{j_n}^\chi   )   ,$$
where $\W$ is a smooth function with a compact support in $[1,2]$, the $*$-sum is over primitive Dirichlet characters modulo $q$, the $\sharp$-sum is over distinct indices $j_k$  and throughout this paper
\es{ \label{def:U} \mathcal U = \frac{\log Q}{2\pi}.}
 
  If $\eta < 2$, the off-diagonal terms in $\mathcal{L}_0 $ do not contribute to the main term, and the same method as for proving $n$-correlation of the Riemann zeta function can be applied here, and we do not even need the extra average over $q$. For example, previously, Hughes and Rudnick \cite{HR} derived the same result as in Theorem \ref{main thrm} when $n = 1$ and averaging only over primitive characters of a fixed prime modulus.    Otherwise, the off-diagonal terms also contribute to the main term in $\mathcal{L}_0$.  In this paper, we use the asymptotic large sieve technique to deal with the off-diagonal terms and evaluate
$$ \mathcal{L}_1 (f, \W, Q)  := \int_\R   \sum_{q} \frac{\W (q/Q)}{ \varphi(q)} \sumstar_{\chi \smod q} \ \sumsh_{j_1, ..., j_n}  f( \mathcal U( \gamma_{j_1}^\chi  -t ),  \dots, \mathcal U ( \gamma_{j_n}^\chi - t  ) )  \emts  dt .$$
The $t$-average is fairly short due to the rapid decay of $e^{-t^2}$ along the vertical line, and its appearance is to deal with certain unbalanced sums of the prime powers. Thus this average involves points very close to the real axis, and it is expected to have the same asymptotic formula as $\mathcal{L}_0$ up to a constant factor. It would be very interesting to develop techniques to evaluate $\mathcal L_0$ without the additional short average over $t$.  The computation of the sixth \cite{CISsixth} and eighth moment \cite{CL} of Dirichlet $L$-functions,  averaging over the same family of primitive characters, also contains a similar $t$-average for the same reason.

 Our goal is to prove the following theorem.
\begin{theorem}\label{main thrm}
	Assume GRH for all primitive Dirichlet $L$-functions. Let $f$ have the C4-Property as described above. Then
	\begin{equation}\label{eq of main thm}
	\lim_{ Q \to \infty} \frac{ \mathcal{L}_1 ( f, \W, Q)}{D(\W, Q)} =  \int_{ \mathbb{R}^n} f(\vecx ) W^{(n)} (\vecx) d\vecx , 
	\end{equation}
	where   
	\es{\label{def:DWQ} D(\W,Q) &:= \sum_{q}  \frac{\W (q/Q)}{ \varphi(q)} \varphi^*(q)  \int_{-\infty}^{\infty} \emts \> dt ,}
	  $\varphi^*(q)$ is the number of primitive characters mod $q$ and $W^{(n)}  (\vecx)$ is defined in (\ref{def:Wn}).
\end{theorem}

Note that instead of the C4-Property we may assume in Theorem \ref{main thrm} that $f : \mathbb{R}^n  \rightarrow \mathbb{R}$ is even in all variables and $\hat{f}(\vecu)$ is supported in the region $ \sum_{j=1}^n |u_j | <  4$  by an approximation argument as in Corollary 1.3 of \cite{ERRudnick}.

This is consistent with the $n$-correlation conjecture arising from the GUE model in random matrix theory where we are able to use a test function whose Fourier transform has double the support of the ones appearing in Rudnick and Sarnak's work. This is the first time for unitary ensemble that the conjecture is verified for a wider range for all $n$. 

 We note that stronger estimations for $n=1$ without $t$-average were studied and conjectured. For details, see \cite{FM} and \cite{HR}. 
  
The proof contains a number of technical details, so we outline it here. In Section \ref{sec:setup}, we will apply a combinatorial sieving, which transforms the sum over distinct ordered zeros in $\mathcal L_1$ to the unrestricted sums. By the explicit formula for a primitive Dirichlet $L$-function, we can express the sum over zeros as a sum over primes. Then, essentially we need to understand the sum $\Scal$ in Proposition \ref{prop:als}. 
The diagonal term is easy to  evaluate, but in our case there is an off-diagonal contribution. To deal with this, we apply the asymptotic large sieve technique developed in \cite{ConreyIwaniecSoundararajan}. Certain delicate combinatorial arrangements appear in these terms along this process. This phenomena does not occur in the pair correlation work of \cite{CLLR} because it can be easily reduced to cases when $\m$ and $\n$ are prime numbers. The details will be covered in Section \ref{sec:ALS}. As a result, the asymptotic formula for \eqref{eq of main thm} is given in \eqref{eqn:lhs estimation complete}.

 Finally, we verify that the result agrees with the random matrix conjecture through the new $n$-correlation formula of Conrey and Snaith \cite{ConreySnaith}, \cite{ConreySnaith2}.  The detailed proof will appear in Section \ref{sec:RMT}.

\section{Initial Setup for the proof of Theorem \ref{main thrm}}\label{sec:setup}
In this section, we will explain how the sum over distinct ordered zeros in $\mathcal L_1(f, \W, Q)$ can be deduced from the unrestricted sum by the combinatorial sieving. This sieving is also appeared in  \cite{RudnickSarnak}, but we describe it here for the sake of completeness.

 A set partition $\underline{G} = \{ G_1 , \dots, G_\nu \}  $ of $ [n] = \{ 1, 2, \dots, n \}$ is a decomposition of $[n]$ into disjoint nonempty subsets $ G_1 , \dots, G_\nu $, where $ \nu = \nu(\underline{G})$.   The collection $\Pi_n$ of all set partitions of $[n]$ forms a lattice with the partial ordering given by $\underline{H}  \preceq \underline{G} $ if every set $G_i $ in $\underline{G}$ is a union of sets in $\underline{H}$. For example, $ \{ \{1, 4\}, \{2\}, \{3\} \} \preceq \{  \{1, 4 \}, \{2, 3\} \} $ in $\Pi_4$. Hence the minimal element of $\Pi_n $ is $\underline{O}=  \{  \{ 1 \} , \{ 2  \} , \dots, \{ n \} \} $ and the maximal element is $ \{ [n] \} $.

\begin{lemma}\label{lemma:cs}     
     The M\"{o}bius function of the poset $\Pi_n$ is the unique function $\mu_n ( \underline{H},\underline{G})$ such that for any functions $C, R : \Pi_n \to \R$, satisfying
   $$ C_{\underline{H}} = \sum_{ \underline{H} \preceq \underline{G} } R_{\underline{G}}, $$ 
   we have
   $$ R_{\underline{H}} = \sum_{ \underline{H} \preceq \underline{ G} } \mu_n ( \underline{H}, \underline{G} ) C_{\underline{G}}  . $$
   In particular, 
   $$ \mu_n ( \underline{O}, \underline{G}) = \prod_{j =1}^\nu (-1)^{ |G_j |-1} ( |G_j |-1)! . $$
      \end{lemma}    
      
Given a set partition $\underline{G} =  \{  G_1 , \dots, G_\nu \}  \in \Pi_n$, define an embedding $\iota_{\underline{G}} : \R^\nu \to \R^n$ by $ \iota_{\underline{G}}(x_1 , \dots, x_\nu ) = ( y_1 , \dots, y_n ) $, where $ y_\ell = x_j $ if $ \ell \in G_j $.  For example, when $\underline{G} = \{ \{1, 4 \}, \{2\}, \{3\} \}$, $ \iota_{\underline{G}}(x_1, x_2, x_3) = (x_1, x_2, x_3, x_1 ).$ 
We also define
 $$ R_{1,\underline{G}}   :=    \sum_q \frac{ \W(q/Q)}{\varphi(q)}  \sumstar_{\chi \smod q} \ \sumsh_{  \gamma_{j_1}^\chi , \dots  , \gamma_{j_\nu}^\chi  } g ( \iota_{\underline{G}} ( \gamma_{j_1}^\chi , \dots ,  \gamma_{j_\nu}^\chi ))  $$
   and
   $$ C_{1,\underline{G}}  :=    \sum_q \frac{ \W(q/Q)}{\varphi(q)}  \sumstar_{\chi \smod q} \ \sum_{  \gamma_{j_1}^\chi , \dots  , \gamma_{j_\nu}^\chi  } g ( \iota_{\underline{G}} ( \gamma_{j_1}^\chi , \dots ,  \gamma_{j_\nu}^\chi )) , $$
   where the $*$-sum is over primitive Dirichlet characters modulo $q$, the $\sharp$-sum is over distinct indices $j_k$, and
$$ g(u_1 , \dots, u_n  ) = \int_\R  f( \mathcal U( u_1  -t ),  \dots, \mathcal U ( u_n  - t  ) )  \emts  dt  . $$   
   Then
   $$ C_{1,\underline{H}} = \sum_{ \underline{H} \preceq \underline{G} } R_{1,\underline{G}}, $$
  and by Lemma \ref{lemma:cs} 
 \begin{equation}\label{eqn:rc rel}
 \mathcal{L}_1 (f, \W, Q)  = R_{1,\underline{O}}   = \sum_{   \underline{ G} \in \Pi_n  } \mu_n ( \underline{O}, \underline{G}) C_{1,\underline{G}}  .
  \end{equation}      
    
We focus on computing $C_{1,\underline{G}} .$   Let
\es{\label{def:Fell} F_\ell (x) = \prod_{i \in G_\ell} f_i (x)   }
for $ \underline{G} = \{G_1 , \dots, G_{\nu } \} \in \Pi_n $. Then by Claim 1 of \cite{Rubinstein} the Fourier transform $\widehat F_\ell (u)$ is supported in $[-\kappa_\ell, \kappa_\ell]$ with 
\es{ \label{kappa ell def}
\kappa_\ell :=\sum_{i \in G_\ell} \eta_i }
 and the function $\prod_{\ell \leq \nu} F_\ell (x_\ell) $ has the C4-Property defined in Section \ref{sec:intro}.  Thus, we see that
\es{ \label{eqn:C1Fginitial}
 C_{1,\underline{G}} & = \int_\R    \sum_q \frac{ \W(q/Q)}{\varphi(q)}  \sumstar_{\chi \smod q} \ \sum_{  \gamma_{j_1}^\chi , \dots  , \gamma_{j_\nu }^\chi  } \prod_{\ell \leq \nu} F_\ell   \big( \, \mathcal U(\gamma_{j_\ell }^\chi -t) \big)   \emts dt \\
 & = \int_\R    \sum_q \frac{ \W(q/Q)}{\varphi(q)}  \sumstar_{\chi \smod q} \  \prod_{\ell \leq \nu} \bigg(  \sum_{  \gamma^\chi   } F_\ell   \big( \, \mathcal U(\gamma^\chi -t) \big) \bigg)  \emts dt .
}
Applying the explicit formula in Lemma \ref{lemma:explicit formula}, we find that
\begin{align*}
 C_{1,\underline{G}} 
 & = \sum_{  S_1 + \cdots + S_4=  [ \nu ]  }   \int_\R    \sum_q \frac{ \W(q/Q)}{\varphi(q)}  \sumstar_{\chi \smod q} \  \prod_{\ell \in S_1 } D_\ell(t)  \prod_{\ell \in S_2 } \overline{D_\ell(t)} \prod_{\ell \in S_3 } \widehat{F}_\ell (0) \prod_{\ell \in S_4 } E_\ell (t)    \emts dt ,
\end{align*}
where  
$$ D_\ell(t) = - \frac{1}{ \log Q}   \sum_{m=1}^\infty \frac{ \Lambda(m) \chi(m)}{m^{ 1/2 + it}} \widehat{F}_\ell  \left( - \frac{\log m}{\log Q} \right) ,  $$
 \begin{align*}
E_\ell ( t) & : = E_{F_\ell} (t)  =  O \bigg(  \frac{ \log (2+|t|)}{ \log Q}  \bigg) 
\end{align*}
and $E_{F_\ell}(t)$ is defined in \eqref{EFt}. Let $A_1$,..., $A_k$ be a collection of disjoint sets of integers, and let $B$ be a set of integers. Here and throughout this paper, $A_1 + \cdots + A_k = B$ means $B$ is a disjoint union of $A_1,..., A_k$.
Next we write
 $$ \prod_{ \ell \in S_1}  D_\ell (t) = \frac{ (-1)^{|S_1|}}{( \log Q)^{ |S_1|}}   \sum_{\m=1}^\infty \frac{a_\m(S_1) \chi(\m) }{\m^{1/2+it}} $$
 and
 $$   \prod_{ \ell \in S_2} \overline{ D_\ell (t) }  = \frac{  (-1)^{|S_2|}}{(\log Q)^{ |S_2|}}  \sum_{\n=1}^\infty \frac{b_\n(S_2) \bar{\chi}(\n) }{\n^{1/2-it}}, $$
 where $|S_i|$ is the number of elements in $S_i,$
$$   a_\m (S_1) = \sum_{ \prod_{\ell \in S_1} m_\ell = \m } \bigg(  \prod_{ \ell \in S_1 }   \Lambda( m_\ell )   \widehat{F}_\ell \bigg( - \frac{ \log m_\ell }{ \log Q} \bigg)  \bigg)  ,$$
and
$$   b_\n (S_2) =  \sum_{ \prod_{\ell \in S_2} n_\ell = \n }  \bigg( \prod_{ \ell \in S_2 }   \Lambda( n_\ell )  \widehat{F}_\ell \bigg(  \frac{ \log n_\ell }{ \log Q} \bigg)  \bigg)  . $$
Then
\es{ \label{def:C1Fg}
 C_{1,\underline{G}} 
 & = \sum_{  S_1 + \cdots + S_4=  [ \nu ]  }   \bigg( \prod_{\ell \in S_3 } \widehat{F}_\ell (0) \bigg) \frac{  (-1)^{|S_1|+ |S_2|}}{(\log Q)^{|S_1|+ |S_2|}}  \\
 &  \times \int_\R    \sum_q \frac{ \W(q/Q)}{\varphi(q)}  \sumstar_{\chi \smod q} \ \bigg(   \sum_{\m=1}^\infty \frac{a_\m(S_1) \chi(\m) }{\m^{1/2+it}} \bigg) \bigg(   \sum_{\n=1}^\infty \frac{b_\n(S_2) \bar{\chi}(\n) }{\n^{1/2-it}}  \bigg)  \prod_{\ell \in S_4 } E_\ell (t)    \emts dt .
}

We estimate $C_{1, \underline{G}}$ in Sections \ref{sec:refineterms}--\ref{sec:ALS}. In Section \ref{sec:refineterms} we first prove that 
the main contribution to $C_{1, \underline{G}}$ comes from the cases $S_4 = \emptyset$ and squarefree $\m, \n $. As mentioned in the introduction, the main contribution is categorized into two types -- diagonal terms ($\m = \n$), calculated in Section \ref{sec:refineterms}, and off-diagonal terms ($\m \neq \n$), estimated in Section \ref{sec:ALS}.


\section{Preliminary lemmas} \label{sec:lemma}
In this section, we present lemmas required in the proof of Theorem \ref{main thrm}.  Let $F   $ be a smooth and rapidly decreasing function on $\mathbb{R}$ with 
$$ \mathrm{supp} ~\hat{F} \in [-\kappa_0 , \kappa_0 ]. $$ 
Then $F$ has an extension to the complex plane that is entire with 
$$ |F(z)| \leq A e^{2 \pi \kappa_0  |z|}$$
for some $ A>0$. See Theorem 3.3 in \cite[p.122]{SS} for a proof. Moreover, for any integer $A_1 \geq 0$ and any real number $A_2 >0$, one can show by partial integrations that
\begin{equation}\label{eqn F bound}
  F( v+ iy  ) = \int_{\mathbb{R}} \hat{F}(w) e^{2 \pi wy}  e^{-2 \pi i w v }     dw \ll_{A_1 ,A_2  ,\kappa_0}   \frac{1}{ 1+|v|^{A_1}} 
\end{equation}
for $v \in \R  $ and $y \in [-A_2 , A_2 ] $.
\begin{lemma}[Explicit Formula]\label{lemma:explicit formula}
Let $ \chi$ be a primitive Dirichlet character modulo $ q > 1$,  $Q \leq q \leq 2Q$  and $F : \R \to \R $ be a smooth and rapidly decreasing function with   its Fourier transform $\hat{F}$ compactly supported. Assume GRH for $L(s,\chi)$. Define
$$ \kappa = \kappa_ \chi  = \begin{cases} 
0 & \mathrm{if} \  \chi(-1) = 1, \\
1 & \mathrm{if} \ \chi(-1) = -1.
  \end{cases}
  $$
 Then we have
\begin{align*}
 \sum_\gamma F\big(\, \mathcal U ( \gamma - t ) \big)    = &   - \frac{1}{ \log Q}  \sum_{m=1}^\infty   \frac{ \Lambda(m) \chi(m)}{ m^{1/2+it } }   \widehat{  F} \bigg( - \frac{ \log m }{ \log Q} \bigg)     -  \frac{1}{ \log Q}  \sum_{m=1}^\infty   \frac{ \Lambda(m) \bar{\chi}(m)}{ m^{1/2- it } }   \widehat{  F} \bigg(   \frac{ \log m }{ \log Q} \bigg)    \\
 & +  \widehat{F}(0)  + E_F(  t),
\end{align*}
where $\mathcal U = (\log Q)/(2\pi) $ and
\begin{align} \label{EFt}
E_F  (  t) &  :=  \frac{1}{ 2 \pi    }  \int_{-\infty}^\infty  F\big(\, \mathcal U ( u -t)\big)  \Rep \bigg[ \frac{\Gamma'}{\Gamma} \bigg( \frac12 \bigg( \frac12+iu +\kappa \bigg) \bigg)   \bigg] du  =  O \bigg(  \frac{ \log (2+|t|)}{ \log Q}  \bigg) .
\end{align}

\end{lemma}
\begin{proof}
Define
$$ \xi(s, \chi) = L(s , \chi) \Gamma \bigg( \frac{ s+\kappa}{2} \bigg) \bigg( \frac{q}{\pi} \bigg)^{(s+\kappa)/2} . $$
Then $\xi(s, \chi)$ is an entire function and its zeros are exactly the nontrivial zeros of $L(s, \chi)$. By Cauchy's integral formula
\begin{align*}
 \sum_\gamma F\big(\, \mathcal U( \gamma - t ) \big)   = &  \frac{1}{ 2 \pi i }  \int_{(1)}  F\big(\, \mathcal U( -iw-t)\big) \frac{ \xi'}{\xi} \bigg( \frac12 + w , \chi\bigg) dw \\
  & -   \frac{1}{ 2 \pi i }  \int_{(-1)}  F\big(\, \mathcal U( -iw-t)\big) \frac{ \xi'}{\xi} \bigg( \frac12 + w , \chi\bigg) dw \\
  := & I_1 + I_2 .
\end{align*}
Note that the contribution from the horizontal lines vanishes by \eqref{eqn F bound}. We shall estimate $I_1$ first.
\begin{align*}
I_1 &=   \frac{1}{ 2 \pi i }  \int_{(1)}  F\big(\, \mathcal U( -iw-t)\big) \bigg(  \frac{ L'}{L} \bigg( \frac12 + w , \chi\bigg)  + \frac12 \frac{\Gamma'}{\Gamma} \bigg( \frac12 \bigg( \frac12+w+\kappa \bigg) \bigg) +\frac12 \log \frac{q}{\pi}  \bigg)dw  \\
   &=: I_{11} + I_{12} + I_{13} .
\end{align*}
Writing out $ L'/L (s)$ in term of its Dirichlet series and shifting the contour integration to Re$(w) = 0$, we have 
\begin{align*}
 I_{11}  & = 
    -\sum_{m=1}^\infty   \frac{ \Lambda(m) \chi(m)}{ m^{1/2+it } }   \frac{1}{ 2 \pi \mathcal U    }  \int_{-\infty}^\infty  F(u)   m^{- i u/ \mathcal U     }   du  = - \frac{1}{ \log Q}  \sum_{m=1}^\infty   \frac{ \Lambda(m) \chi(m)}{ m^{1/2+it } }   \widehat{  F} \bigg( - \frac{ \log m }{ \log Q} \bigg);
 \end{align*}

 \begin{align*}
I_{12}     =    \frac{1}{ 4 \pi    }  \int_{-\infty}^\infty  F\big(\, \mathcal U( u -t)\big)   \frac{\Gamma'}{\Gamma} \bigg( \frac12 \bigg( \frac12+iu +\kappa \bigg) \bigg)   du;
  \end{align*}
and
$$ I_{13} =    \frac{\log  {q}/{\pi} }{ 2\log Q} \widehat{F}(0).  $$
Next we consider $I_2 $. By the functional equation of $\xi(s,\chi)$, 
$$ \frac{ \xi'}{\xi} (s, \chi) = - \frac{\xi'}{\xi} ( 1-s , \bar{\chi}).$$
(See Section 10.1 of \cite{MV} for the detail.)
Thus,
\begin{align*}
I_2 &=          \frac{1}{ 2 \pi i }  \int_{(1)}  F\big(\, \mathcal U( iw-t)\big) \frac{ \xi'}{\xi} \bigg( \frac12 +  w , \bar{\chi}\bigg) dw \\
   &=   \frac{1}{ 2 \pi i }  \int_{(1)}  F\big(\, \mathcal U( iw-t)\big) \bigg(  \frac{ L'}{L} \bigg( \frac12 + w , \bar{\chi} \bigg)  + \frac12 \frac{\Gamma'}{\Gamma} \bigg( \frac12 \bigg( \frac12+w+\kappa \bigg) \bigg) +\frac12 \log \frac{q}{\pi}  \bigg)dw  \\
    & =: I_{21} + I_{22} + I_{23} .
\end{align*}
By the same argument as $I_1$, we obtain that
\begin{align*}
 I_{21}  & =  
  - \frac{1}{ \log Q}  \sum_{m=1}^\infty   \frac{ \Lambda(m) \bar{\chi}(m)}{ m^{1/2- it } }   \widehat{  F} \bigg(   \frac{ \log m }{ \log Q} \bigg)    ,
 \end{align*}
 \begin{align*}
I_{22} 
  & =    \frac{1}{ 4 \pi    }  \int_{-\infty}^\infty  F\big(\, \mathcal U( u -t)\big)   \frac{\Gamma'}{\Gamma} \bigg( \frac12 \bigg( \frac12 - iu +\kappa \bigg) \bigg)   du, 
  \end{align*}
 and
$$ I_{23} =      \frac{\log  {q}/{\pi} }{ 2\log Q} \widehat{F}(0) =  I_{13}.  $$
Hence,
\begin{align*}
 \sum_\gamma F\big(\, \mathcal U ( \gamma - t ) \big)   = &   
   - \frac{1}{ \log Q}  \sum_{m=1}^\infty   \frac{ \Lambda(m) \chi(m)}{ m^{1/2+it } }   \widehat{  F} \bigg( - \frac{ \log m }{ \log Q} \bigg)     -  \frac{1}{ \log Q}  \sum_{m=1}^\infty   \frac{ \Lambda(m) \bar{\chi}(m)}{ m^{1/2- it } }   \widehat{  F} \bigg(   \frac{ \log m }{ \log Q} \bigg)    \\
 & +\frac{\log  {q}/{\pi} }{ \log Q} \widehat{F}(0)  +  \frac{1}{ 2 \pi    }  \int_{-\infty}^\infty  F\big(\, \mathcal U( u -t)\big)  \Rep \bigg[ \frac{\Gamma'}{\Gamma} \bigg( \frac12 \bigg( \frac12+iu +\kappa \bigg) \bigg)   \bigg] du .
\end{align*}
 Since $Q \leq q \leq 2Q$, $\frac{\log q/\pi}{\log Q} = 1 + O\left( \frac{1}{\log Q} \right)$. Moreover, by Stirling's formula, the integration above is bounded by
\begin{align*}
& \ll  \int_{-10}^{10} |F \big(\, \mathcal U(u-t)\big)| du +  \int_{|u|>10}  | F  \big(\, \mathcal U( u -t)\big) | \log |u|  du   
 \ll  \frac{ \log (2+|t|)}{ \log Q},
\end{align*}
and we then obtain (\ref{EFt}).

 \end{proof}

\begin{lemma}[Large sieve inequality] \label{lemma:lsi}
For any complex numbers $a_m$ with $ M< m  \leq M+N$, where $N$ is a positive integer, we have
$$ \sum_{ Q < q \leq 2Q}  \frac{1}{ \varphi(q)} \sumstar_{\chi \shortmod q} \bigg|  \sum_{ M < m \leq M+N }  a_m \chi(m)  \bigg|^2  \ll   \bigg(  Q + \frac{N}{Q} \bigg) \sum_{  M < m \leq M+N } |a_m|^2 . $$ 
\end{lemma}
This is a consequence of Theorem 7.13 in \cite{IK}.

\begin{lemma}\label{lemma:D asymp} Let $D(\W, Q)$ be defined as in (\ref{def:DWQ}). Then
$$ D(\W,Q)  =  \sqrt \pi  \hW(1) Q \prod_p \left(1 - \frac 1{p^2} - \frac{1}{p^3} \right) + O\big( \sqrt{Q}  \big),
  $$
  where $\widetilde{\W}(s) = \int_0^\infty  \W(x) x^{s-1} dx  $ is the Mellin transform of $\W$ and the product is over the prime numbers. 
\end{lemma}
\begin{proof}
By the Mellin inversion formula and 
\es{\label{int:et2} \int_{-\infty}^{\infty} e^{-t^2} \> dt = \sqrt{\pi},} 
we have
\est{D(\W,Q) & = \sqrt \pi \sum_q \frac{1}{2\pi i} \int_{(2)} \frac{Q^s}{ \varphi(q) q^s} \hW(s) \varphi^*(q) \> ds.}
 Since $ \varphi^*(q)  = \sum_{cd=q} \varphi(c) \mu(d)$ and $ \varphi( q) = \sum_{ cd=q} c\mu(d)$, we obtain that
$$ \sum_{ q=1}^\infty \frac{ \varphi^*(q)}{ \varphi(q) q^s} = \frac{ \zeta(s)}{\zeta(s+1)} \mathcal G(s),$$
where $$\mathcal G(s) =  \prod_p  \left( 1 - \frac{1}{p^{s+1}}\right)^{-1}\left(  1- \frac{1}{(p -1)p^s} + \frac{1}{(p-1)p^{2s} } - \frac{1}{p^{2s + 1}}\right),$$ and it is absolutely convergent when $\Rep (s) > 0.$ Therefore 
\est{D(\W,Q) & = \sqrt \pi   \frac{1}{2\pi i} \int_{(2)}  Q^s  \hW(s)  \frac{ \zeta(s)}{\zeta(s+1)} \mathcal G(s) \> ds. }
 Moving the contour integration to the line $\Rep (s) = 1/2$, we pick up a simple pole at $s = 1$.  Since the Mellin transform of the compactly supported $\W$ decays faster than any power of $\Imp(s)$, we derive the lemma.

\end{proof}

\begin{lemma}\label{lemma:B sum} Let $m$ be a positive integer. Then  for $ \Rep (s) > 0 $ 
$$ \sum_{ \substack{d \\ (d, m)=1}} \frac{1}{ \varphi(cd) d^s } = \frac{1}{\varphi(c)} \zeta(1+s) B(s) B_1( s,m) B_2 (s , c)   ,$$
where
\begin{equation*}\begin{split}
B(s) & = \prod_p \bigg( 1+ \frac{1}{ (p-1)p^{s+1}}\bigg) \\
B_1(s , m) & = \prod_{p |m} \bigg( 1- \frac{1}{p^{s+1}}\bigg)\bigg( 1+ \frac{1}{ (p-1)p^{s+1}} \bigg)^{-1}   \\
B_2(s, c) & = \prod_{p|c} \bigg( 1+ \frac{1}{ (p-1)p^{s+1}} \bigg)^{-1} .
\end{split}\end{equation*}
 For $ \Rep(s)=-1+\varepsilon > -1 $ we have
$ |B(s) | \ll_\varepsilon 1 $,  $|B_1 (s,m)  |   \ll_\varepsilon  2^{\omega(m)}$ and $ |B_2 ( s, c) | \ll_\varepsilon 1 $, where $ \omega (m)$ is the number of distinct prime factors of $m$.
\end{lemma}
 
\begin{proof}
One can prove the first identity in the lemma by changing the sum to its Euler product. The proof is quite standard and we omit it. To prove the inequalities we see that for $ \Rep(s)=-1+\varepsilon >-1$
\es{ \label{Bsbounds} |B(s) |  & \leq \prod_p \bigg( 1 + \frac{1}{ (p-1)p^\varepsilon} \bigg),  \\
 |B_1 ( s,m) | & \leq  \prod_{p |m} \bigg( 1 +  \frac{1}{p^{\varepsilon }}\bigg) 
 \prod_p \bigg( 1- \frac{1}{ (p-1)p^{\varepsilon }} \bigg)^{-1} \ll_\varepsilon  \prod_{p|m} 2  = 2^{\omega(m)}\\
 |B_2 ( s,c) |  &   \leq     \prod_p  \bigg( 1- \frac{1}{ (p-1)p^{\varepsilon}} \bigg)^{-1} .  }
 Note that the infinite products in \eqref{Bsbounds} are convergent.
\end{proof}

    \begin{lemma}\label{lemma:primecharsum1}
 Let  $\Psi $ be a nonprincipal Dirichlet character modulo $  d > 1$. Suppose that $ c,d  \leq Q^4$. Assume GRH for $L(s, \Psi)$. Define
 \est{a_\m =\mu^2 (\m)  \sum_{  p_1 \cdots p_k = \m  }\bigg(  \prod_{j=1}^k \log p_j   \widehat{F}_j \left( -\frac{\log p_j }{\log Q} \right)  \bigg)  }
and 
\est{b_\n = \mu^2(\n)  \sum_{  p_{k+1} \cdots p_{k+r} = \n  }    \bigg( \prod_{j=k+1}^{k+r}   \log p_j   \widehat{F}_j \left( \frac{\log p_j }{\log Q} \right) \bigg) ,}
where $F_j$ is defined in \eqref{def:Fell}.  Let $ \alpha , \beta \in \C$ with $\Rep (\alpha), \Rep (\beta) \in \left(\frac 12 -\frac{10}{\log Q} , \frac 12 +\frac{10}{\log Q} \right) $.   Then
   $$ \sumtwo_{\substack{\m,\n   \\ (\m,\n)=1 \\ (\m\n, c)=1  }}     \frac{a_\m \Psi (\m) b_n\overline{\Psi }(\n) }{\m^\alpha \n^\beta }  \ll \bigg(\log \big(Q( 2+|\Imp (\alpha) |)\big)\bigg)^{2k}  \bigg(\log \big( Q(2 + |\Imp (\beta) |\big) \bigg)^{2r},$$
   where the implied constant depends on $k$ and $ r$.

 \end{lemma}
  \begin{proof}
  Define 
  \begin{align*}
   g( p_1, \dots, p_{k+r} )  := &   \prod_{j=1}^k \bigg( \frac{\Psi(p_j)  \log p_j}{p_j^\alpha}\widehat{F}_j \bigg( - \frac{ \log p_j}{\log Q} \bigg) \bigg)  \prod_{j=k+1}^{k+r} \bigg( \frac{\overline{\Psi} (p_j)  \log p_j}{p_j^\beta}\widehat{F}_j \bigg(   \frac{ \log p_j}{\log Q} \bigg) \bigg)
   \end{align*}
  for $ ( p_1 \cdots p_{k+r}, c ) = 1 $, and otherwise,  $ g( p_1, \dots, p_{k+r} )=0 $. Furthermore, we define
 $$ R_{0,\underline{G}}   :=  \sumsh_{ p_1 , \dots, p_{\nu} } g ( \iota_{\underline{G}} ( p_1 , \dots, p_{\nu} ))  $$
   and
   $$ C_{0,\underline{G}} :=   \sum_{ p_1 , \dots, p_{\nu} } g ( \iota_{\underline{G}} ( p_1 , \dots, p_{\nu} )), $$
where $\underline{G} = \{ G_1,...,G_\nu \}  \in \Pi_{ k+r}$, $\iota_{\underline{G}}$, and $\sumsh$ are defined in Section \ref{sec:setup}.  It is clear that
   $$ C_{0,\underline{H}}  = \sum_{ \underline{H} \preceq \underline{G} } R_{0,\underline{G}},$$
  and by Lemma \ref{lemma:cs}, we have
   \es{\label{eqn:usecsambn} \sum_{\substack{\m,\n   \\ (\m,\n)=1 \\ (\m\n, c)=1  }}     \frac{a_\m \Psi (\m) b_\n\overline{\Psi }(\n) }{\m^\alpha \n^\beta } =  R_{0,\underline{O}}  = \sum_{   \underline{G} \in \Pi_{k+r}  } \mu_{k+r} ( \underline{O}, \underline{G}) C_{0,\underline{G}}  .}
For each $\underline{G} =  \{ G_1,...,G_\nu \}  \in \Pi_{ k+r}$, we have
\begin{align*}
C_{0, \underline{G}}  = \prod_{j=1}^\nu \bigg[   \sum_{ (p,c)=1}    \bigg(  \prod_{\substack{ \ell \leq k \\ \ell \in G_j }}     \frac{\Psi(p )  \log p }{p^\alpha}\widehat{F}_\ell  \bigg( - \frac{ \log p }{\log Q} \bigg) \bigg) \bigg(  \prod_{\substack{ k < \ell \leq k+r \\  \ell \in G_j }}   \frac{\overline{\Psi} (p )  \log p }{p^\beta}\widehat{F}_\ell \bigg(   \frac{ \log p }{\log Q} \bigg) \bigg)  \bigg] .
 \end{align*} 
If $|G_j | \geq 3 $, then
$$  \sum_{ (p,c)=1}  \bigg(   \prod_{\substack{ \ell \leq k \\ \ell \in G_j }}    \frac{\Psi(p )  \log p }{p^\alpha}\widehat{F}_\ell  \bigg( - \frac{ \log p }{\log Q} \bigg) \bigg)  \bigg(  \prod_{\substack{ k < \ell \leq k+r \\  \ell \in G_j }}  \frac{\overline{\Psi} (p )  \log p }{p^\beta}\widehat{F}_\ell \bigg(   \frac{ \log p }{\log Q} \bigg) \bigg) =O(1) .$$
When $|G_j| =2$, $\widehat F_\ell$ is compactly supported in $[-\kappa_\ell, \kappa_\ell ],$ where $\sum_\ell \kappa_\ell < 4,$  and it follows that
$$  \sum_{ (p,c)=1}   \bigg(   \prod_{\substack{ \ell \leq k \\ \ell \in G_j }}     \frac{\Psi(p )  \log p }{p^\alpha}\widehat{F}_\ell  \bigg( - \frac{ \log p }{\log Q} \bigg) \bigg) \bigg(   \prod_{\substack{ k < \ell \leq k+r \\  \ell \in G_j }}   \frac{\overline{\Psi} (p )  \log p }{p^\beta}\widehat{F}_\ell \bigg(   \frac{ \log p }{\log Q} \bigg) \bigg) = O\left( (\log Q)^2 \right) .$$
Hence
\begin{align*}
C_{0, \underline{G}}  \ll   \prod_{\substack{  G_j= \{\ell\}  \\ \ell \leq k   } }  \bigg| \sum_{ (p,c)=1}         \frac{\Psi(p )  \log p }{p^\alpha}\widehat{F}_\ell  \bigg( - \frac{ \log p }{\log Q} \bigg) \bigg|
 \prod_{\substack{  G_j= \{\ell\}  \\ k< \ell \leq k+r   } }  \bigg| \sum_{ (p,c)=1}         \frac{\overline{\Psi}(p )  \log p }{p^\beta}\widehat{F}_\ell  \bigg(   \frac{ \log p }{\log Q} \bigg) \bigg| (\log Q)^{2 g} ,
 \end{align*} 
 where $g$ is the number of $j$ such that $|G_j| = 2.$
 When Re$(s) \geq \frac{1}{2} + \frac{1}{\log Q}$, it is known that under GRH,
 $$ \frac{L'}{L} ( s, \Psi) = O\bigg( \log^2 \big(Q ( 2+ |\Imp (s)| )\big)\bigg)$$
(e.g. Chapter 19 in \cite{Da}). By Fourier inversion formula, the fact that $ \widehat{F}_\ell $ is supported in $[- \kappa_\ell , \kappa_\ell ] $, and the integration by parts, we have for $|y| \leq 1,000$,
 $$ F_\ell ( v+iy)   = \int_\R \widehat{F}_\ell (w) e^{- 2 \pi i wv} e^{  2 \pi  w y} dw \ll  \frac{ 1 }{ 1 + |v|^A } $$
 for any nonnegative integer $A$. 
 Because $\Psi$ is a non-principal character and using the bound above, we have
 \begin{equation}\label{eqn:pr char sum}\begin{split}
\sum_{ (p,c)=1}   &       \frac{\Psi(p )  \log p }{p^\alpha}\widehat{F}_\ell  \bigg( - \frac{ \log p }{\log Q} \bigg) \\
& = \sum_n \frac{\Psi(n )  \Lambda(n) }{n^\alpha}\widehat{F}_\ell  \bigg( - \frac{ \log n }{\log Q} \bigg) + O\left( \log Q\right) \\
&=-\mathcal U \int_\R F_\ell \left( \, \mathcal U \bigg( v-\tfrac{20i}{\log Q} \bigg)\right)  \frac{L'}{L} \left( \tfrac{20}{\log Q} + \alpha + iv  , \Psi \right)  dv  + O( \log Q)\\
& \ll 
\mathcal U \int_\R  \frac{ 1 }{1 + \mathcal U^A |v|^A}  \log^2\big( Q ( 2 + |\Imp (\alpha) | + |v|  )\big) dv \\
&\ll \log^2\big(Q(2+ |\Imp \alpha |) \big) . 
\end{split}\end{equation}
 Therefore,
$$C_{0, \underline{G}}  \ll \bigg(\log \big(Q( 2+|\Imp (\alpha) |)\big)\bigg)^{2k}  \bigg(\log \big( Q(2 + |\Imp (\beta) |\big) \bigg)^{2r}, $$
and the lemma follows from the above and Equation (\ref{eqn:usecsambn}).

  \end{proof}

   \begin{lemma}\label{lemma:prime sum}
    Assume RH and that $ F: \R \to \R $ is smooth and rapidly decreasing with $\hat{F} $ compactly supported.   
    Define 
$$ R_\pm (\alpha, F ) =    \sum_p  \frac{ \log p }{ p^{\alpha }} \widehat{F} \bigg( \pm   \frac{ \log p }{ \log Q} \bigg) - F( \pm  i \mathcal{U} ( 1-\alpha  )) \log Q.$$
Then
\begin{align*}
R_\pm ( \alpha, F) = 
 -\log Q \int_{-\infty}^0   \hat{F}( \pm w) Q^{(1-\alpha)w}dw  + O(1) + O \bigg(  \bigg( \frac{1}{ \log Q} + |\alpha| \bigg)  ( \Rep (\alpha) - 1/2)^{-3}    \bigg)   
\end{align*}
for $\Rep (\alpha) \geq 1/2 + 10/\log Q $, and
\begin{align*}
R_\pm ( \alpha, F) =  O( ( \log Q)^2 )
\end{align*}
for $ | \Rep ( \alpha) - 1/2| \leq 10/\log Q$.
                        \end{lemma}
\begin{proof}
  Since $ F : \R \to \R $, we have $ \overline{\hat{F}(w)} = \hat{F}( -w)$ for $ w \in \R$. Therefore it is enough to consider only the negative case.
  

  When $ | \Rep (\alpha) - 1/2 | \leq 10 / \log Q$,  by similar arguments to  \eqref{eqn:pr char sum}, we obtain that
 $$    \sum_p  \frac{ \log p }{ p^{\alpha }} \widehat{F} \bigg(  - \frac{ \log p }{ \log Q} \bigg) = F( i \mathcal{U} (\alpha-1)) \log Q +O( ( \log Q)^2 )  .$$

 Now we prove the first assertion. Assume that  $ \Rep( \alpha) \geq 1/2 + 10/\log Q $. By the prime number theorem of the form 
 $$ \vartheta(x) := \sum_{ p \leq  x } \log p = x + O( \sqrt{x} ( \log x)^2 )   $$
under RH, we have
 \begin{align*}
    \sum_p & \frac{ \log p }{ p^{\alpha }} \widehat{F} \bigg(  - \frac{ \log p }{ \log Q} \bigg) \\
     & = \int_1^\infty      \frac{ 1 }{ v^{\alpha }} \widehat{F} \bigg(  - \frac{ \log v }{ \log Q} \bigg) dv + \int_1^\infty      \frac{ 1 }{ v^{\alpha }} \widehat{F} \bigg(  - \frac{ \log v }{ \log Q} \bigg) d( \vartheta(v) - v) \\
    & = \log Q \int_{-\infty}^0 \hat{F}(w) Q^{(\alpha-1)w}dw  + O(1)+O \bigg(  \bigg( \frac{1}{ \log Q} + |\alpha| \bigg)  \int_1^{\infty} v^{- \alpha - 1/2 } ( \log v )^2 dv      \bigg) \\
    & = F( i \mathcal{U} (\alpha -1) ) \log Q -\log Q \int_{-\infty}^0   \hat{F}( - w) Q^{(1-\alpha)w}dw   \\
& \quad  + O(1) + O \bigg(  \bigg( \frac{1}{ \log Q} + |\alpha| \bigg)  ( \Rep (\alpha) - 1/2)^{-3}    \bigg)  .
 \end{align*}

\end{proof}

\begin{lemma}\label{lemma:Fizw1w2}
Let $ w_1 , w_2 $ be complex numbers with $\Rep (w_1) = \delta_1  < \Rep (w_2) = \delta_2$.  Let $F : \R \to \R$ be a smooth and rapidly decreasing function with compactly supported $\hat{F}$. Then 
\est{ 
  \frac{1}{2\pi i } &  \int_{ (\delta  )} F (i z  ) \bigg(  \frac{1}{z  - w_1        }  -    \frac{1}{  z - w_2     }  \bigg) dz \\ 
  & =  \left\{\begin{array}{ll}
  	\int_0^\infty   \hat{F}(-u) e^{-2 \pi u w_2} du - \int_0^\infty   \hat{F}(-u) e^{-2 \pi u w_1}  \> du  & \ \ \ \mathrm{if}~~ \delta  < \delta_1 \\
  	& \\
  	 \int_{-\infty}^0    \hat{F}(-u) e^{-2 \pi u w_1} du +  \int_0^\infty   \hat{F}(-u) e^{-2 \pi u w_2} \> du  & \ \ \  \mathrm{if}~~ \delta_1 < \delta  < \delta_2 \\
  	 &  \\
  	 \int_{-\infty}^0    \hat{F}(-u) e^{-2 \pi u w_1} du - \int_{-\infty}^0    \hat{F}(-u) e^{-2 \pi u w_2} \> du  & \ \ \ \mathrm{if}~~ \delta_2  < \delta.
  \end{array} \right.  
}
\end{lemma}
\begin{proof}
Applying the Fourier  inversion formula  
 and then changing the order of integrals, we see that
\est{ 
  \frac{1}{2\pi i } \int_{ (\delta  )} F (i z  ) &\bigg(  \frac{1}{z  - w_1        }  -    \frac{1}{  z - w_2     }  \bigg) dz \\
  =  &     \frac{1}{2\pi i } \int_{ (\delta  )}  \int_{-\infty}^\infty \hat{F}(-u) e^{- 2 \pi uz} du    \frac{w_1 - w_2 }{( z  - w_1 )(  z - w_2 )    }   dz   \\
  =&       \int_{-\infty}^\infty     \frac{1}{2\pi i } \int_{ (\delta  )} \hat{F}(-u) e^{- 2 \pi uz}      \frac{w_1 - w_2 }{( z  - w_1 )(  z - w_2 )    }     dz du       \\
  = &        \int_{-\infty}^0  \hat{F}(-u)  \frac{1}{2\pi i } \int_{ (\delta  )}    \bigg(  \frac{e^{- 2 \pi uz} }{z  - w_1        }  -    \frac{e^{- 2 \pi uz}}{  z - w_2     }  \bigg) dz   du \\ 
 & +        \int_0^\infty   \hat{F}(-u)  \frac{1}{2\pi i } \int_{ (\delta  )}    \bigg(  \frac{e^{- 2 \pi uz} }{z  - w_1        }  -    \frac{e^{- 2 \pi uz}}{  z - w_2     }  \bigg) dz   du .} 
For $u \leq 0$ we shift the $z$-integral to $-\infty$; otherwise, we shift the $z$-integral to $\infty$. By picking up residues properly, we can conclude the proof of the lemma. 
\end{proof}

\section{ Extracting the main contribution of  $C_{1, \underline  {G}}$} \label{sec:refineterms}

We recall from Equation (\ref{def:C1Fg}) that
 \begin{align*}
 C_{1,\underline{G}}  =
 &  \sum_{  S_1 + \cdots + S_4=  [ \nu ]  }   \bigg( \prod_{\ell \in S_3 } \widehat{F}_\ell (0) \bigg) \frac{  (-1)^{|S_1|+ |S_2|}}{(\log Q)^{|S_1|+ |S_2|}}  \\
 &     \times  \sum_q \frac{ \W(q/Q)}{\varphi(q)}  \sumstar_{\chi \smod q} \  \sumtwo_{\m,\n }  \frac{a_\m(S_1) \chi(\m) }{\sqrt{\m }}  \frac{b_\n(S_2) \bar{\chi}(\n) }{ \sqrt{\n }}  \int_\R   \bigg( \frac{ \n}{\m} \bigg)^{it} \prod_{\ell \in S_4 } E_\ell (t)    \emts dt .
\end{align*}
 In this section, we will extract the main contribution of $C_{1, \underline  {G}}$, and the first steps are to show that the main contribution comes from the following terms.
\begin{itemize}
	\item  $\m, \n \leq Q^2$; 
    \item $S_4 = \emptyset$ ;
    \item $\m, \n$ are squarefree. 
\end{itemize}

 This length of $\m$  and $\n$ is optimal for the large sieve inequality in Lemma \ref{lemma:lsi} as the size of the family of $L$-functions is $Q^2$. We will show how to truncate the sums over $\m, \n$ in Lemma \ref{lem:truncatetoQsq}. The integration over $t$ is an important ingredient to balance the size of $\m, \n$. Moreover, we will use the Cauchy-Schwarz inequality and the large sieve inequality to show that the contribution from other terms is small. 

 Next we extract the main ``diagonal" terms:
\begin{itemize}
	\item $S_1 = S_2 = \emptyset$.  (Equation \eqref{eqn:S1S2empty})
	\item $S_1$ and $S_2$ are not empty sets, and $\m = \n$. (Lemma \ref{lem:diagonalCD}).  These terms are easy to evaluate from the prime number theorem and partial summation. 
\end{itemize}
We will also describe how the main contribution from ``off-diagonal" terms looks like, and it will be calculated in the next section.

\begin{lemma} \label{lem:truncatetoQsq}Let all notations be defined as in Section \ref{sec:setup}. Then
\begin{align*}
 C_{1,\underline{G}}  =
 &   \sum_{  S_1 + \cdots + S_4=  [ \nu ]  }   \bigg( \prod_{\ell \in S_3 } \widehat{F}_\ell (0) \bigg) \frac{  (-1)^{|S_1|+ |S_2|}}{(\log Q)^{|S_1|+ |S_2|}}  \\
 &     \times  \sum_q \frac{ \W(q/Q)}{\varphi(q)}  \sumstar_{\chi \smod q} \  \sumtwo_{\m,\n \leq Q^2 }  \frac{a_\m(S_1) \chi(\m) }{\sqrt{\m }}  \frac{b_\n(S_2) \bar{\chi}(\n) }{ \sqrt{ \n }}  \int_\R   \bigg( \frac{ \n}{\m} \bigg)^{it} \prod_{\ell \in S_4 } E_\ell (t)    \emts dt   \\
 & + O( e^{- \tfrac{\varepsilon^2}{6} ( \log Q)^2 }),
\end{align*}
 where $ E_\ell(t) := E_{F_\ell} (t)$ is defined in \eqref{EFt}. 
\end{lemma}
\begin{proof}
As previously mentioned, each $\widehat{F}_\ell ( u_\ell )$ is supported in   $ |u_\ell | \leq \kappa_\ell :=\sum_{i \in G_\ell} \eta_i$. Thus,  $m_\ell  \leq Q^{\kappa_\ell}$  for $ \ell \in S_1 $,  $n_\ell  \leq Q^{\kappa_\ell}$  for $ \ell \in S_2 $ and
$$ \m =   \prod_{ \ell \in S_1} m_\ell  \leq Q^{\kappa(S_1) }, \quad   \n =   \prod_{ \ell \in S_2} n_\ell \leq Q^{\kappa(S_2) },$$
  where $ \kappa(S_1) := \sum_{ \ell \in S_1 } \kappa_\ell $ and $ \kappa(S_2) := \sum_{ \ell \in S_2 } \kappa_\ell $.  Note that $ \sum_{\ell=1}^\nu  \kappa_\ell = \sum_{i=1}^n \eta_i = \eta \leq 4-\varepsilon $. 

 The bound in \eqref{EFt} is not sufficient to prove the lemma, so we shall use the Fourier inversion formula.  The Fourier transform of  $F_\ell (\mathcal U( u -t)) (1+u^2) $ is 
\begin{align*}
\int_\R F_\ell \big(\, \mathcal U ( u -t)\big) (1+u^2)  e^{2 \pi i uv } du 
& =e^{ 2 \pi i tv}  \int_\R F_\ell \big(\, \mathcal U  u \big) (1+t^2 + 2tu + u^2 )  e^{2 \pi i uv }du  \\
& = e^{ 2 \pi i tv} \bigg(  \frac{1+t^2}{\mathcal U} \widehat{F}_\ell \left(\frac{v}{\mathcal U}\right) + \frac{  t }{ \pi i \, \mathcal U^2 }   \widehat{F}_\ell' \left(\frac{v}{\mathcal U}\right) - \frac{1}{4 \pi^2\, \mathcal U^3 }   \widehat{F}_\ell''\left(\frac{v}{\mathcal U}\right) \bigg),
\end{align*}
so for each $\ell \in S_4$, we have
  \es{ \label{Eellt}
E_\ell ( t) &  =  \frac{1}{ 2 \pi    }  \int_\R  F_\ell \big( \, \mathcal U( u -t)\big) (1+u^2)   \Rep \bigg[ \frac{\Gamma'}{\Gamma} \bigg( \frac12 \bigg( \frac12+iu +\kappa \bigg) \bigg)   \bigg] \frac{du}{1+u^2} \\
&  = \frac{1}{ 2 \pi    }  \int_\R \bigg( \int_\R   e^{ 2 \pi i tv} \bigg(  \frac{1+t^2}{\mathcal U} \widehat{F}_\ell \left(\frac{v}{\mathcal U}\right) + \frac{  t }{ \pi i \, \mathcal U^2 }   \widehat{F}_\ell' \left(\frac{v}{\mathcal U}\right) - \frac{1}{4 \pi^2\, \mathcal U^3 }   \widehat{F}_\ell''\left(\frac{v}{\mathcal U}\right) \bigg) e^{- 2 \pi i vu} dv  \bigg) \\
&\hskip 1in \times  \Rep \bigg[ \frac{\Gamma'}{\Gamma} \bigg( \frac12 \bigg( \frac12+iu +\kappa \bigg) \bigg)   \bigg] \frac{du}{1+u^2} \\
&  = \frac{1}{ 2 \pi    }  \int_\R \bigg( \int_{-\kappa_\ell}^{\kappa_\ell}    e^{ 2 \pi i tv_\ell  \, \mathcal U } \bigg(  (1+t^2 ) \widehat{F}_\ell (v_\ell  ) + \frac{  t }{ \pi i \, \mathcal U }   \widehat{F}_\ell' (v_\ell ) - \frac{1}{4 \pi^2 \, \mathcal U^2 }   \widehat{F}_\ell''(v_\ell  ) \bigg) e^{- 2 \pi i v_\ell u\, \mathcal U} dv_\ell   \bigg) \\
&\hskip 1 in \times \Rep \bigg[ \frac{\Gamma'}{\Gamma} \bigg( \frac12 \bigg( \frac12+iu +\kappa \bigg) \bigg)   \bigg] \frac{du}{1+u^2}.
}
Hence the $t$-integral 
$$ \int_\R   \bigg( \frac{ \n}{\m} \bigg)^{it} \prod_{\ell \in S_4 } E_\ell (t)    \emts dt$$ 
in $C_{1, \underline{G}} $ is a combination of 
$$ \int_\R \bigg( \frac{ \n}{\m} \bigg)^{it}  t^{A_1} e^{  i t  v(S_4 ) \log Q}   \emts dt $$
with a nonnegative integer $A_1 $, where 
$$ v(S_4 ) :=  \sum_{ \ell \in S_4} v_\ell $$
satisfying 
 $$ |v(S_4)| \leq \kappa(S_4 ) .$$ 
It is known that 
\es{\label{eqn:Fourierexsq} \int_{-\infty}^{\infty} e^{-x^2} e^{i\xi x}\> dx = \sqrt{\pi} e^{-\xi^2/4}.}
Taking $j^{th}$ derivative with respect to $\xi$ on both sides, we obtain that 
$$ \int_{-\infty}^{\infty} (ix)^j e^{-x^2} e^{i\xi x}\> dx =  e^{-\xi^2/4}  {P}_j(\xi),$$
where  ${P}_j(\xi)$ is an $j$-degree polynomial function. Therefore
$$  \int_\R \bigg( \frac{ \n}{\m} \bigg)^{it}  t^{A_1} e^{  i t  v(S_4 ) \log Q}   \emts dt = i^{-A_1} e^{ - \tfrac14 \big(   v(S_4 ) \log Q + \log \tfrac {\n}{\m}  \big)^2 } {P}_{A_1} \bigg( v(S_4 ) \log Q + \log \frac{\n}{\m} \bigg) .$$ 
If $ a_\m (S_1 ) \neq 0 $ for $ \m= \prod_{ \ell \in S_1 } m_\ell > Q^2 $, then $\kappa(S_1) > 2 $. Since $ \kappa(S_1 ) + \kappa (S_2 ) + \kappa (S_4) \leq 4 - \varepsilon$, it follows that
$$ \kappa (S_2 ) + \kappa (S_4) < 2 - \varepsilon, $$
and so
$$ \left| v(S_4) \log Q + \log \frac \n\m \right| =  \log \m - \log ( \n Q^{v(S_4)}) \geq \log Q^2 - \log Q^{\kappa(S_2 ) + \kappa(S_4 )} \geq \varepsilon \log Q . $$ 
Hence,
$$  \int_\R \bigg( \frac{ \n}{\m} \bigg)^{it}  t^{A_1} e^{  i t  v(S_4 ) \log Q}   \emts dt \ll (\log Q)^{ A_1 }  e^{ - \tfrac{ \varepsilon^2}{ 4} ( \log Q)^2 } \ll e^{ - \tfrac{ \varepsilon^2}{ 5} ( \log Q)^2 } .$$
Inserting the above bound in (\ref{Eellt}) and (\ref{def:C1Fg}), we obtain that the contribution from the terms $ \m > Q^2 $ is 
$$ \ll  Q^{A_2} e^{ - \tfrac{ \varepsilon^2}{ 5} ( \log Q)^2 } \ll e^{ - \tfrac{ \varepsilon^2}{ 6} ( \log Q)^2 }$$
for some constant $A_2 >0$.  The similar arguments can be applied to the terms $\n > Q^2 $, and this concludes the proof of the lemma.

\end{proof}

Next, we will show that the main contribution of $C_{1, \underline {G}} $ comes from terms with $S_4 = \emptyset$. 
\begin{lemma} \label{lem:S4empty} Let all notations be defined as in Section  \ref{sec:setup}.   Then
\begin{align*}
C_{1,\underline{G}} 
=& \sum_{  S_1 + S_2 + S_3=  [ \nu ]  }   \bigg( \prod_{\ell \in S_3 } \widehat{F}_\ell (0) \bigg) \frac{  (-1)^{|S_1|+ |S_2|}}{(\log Q)^{|S_1|+ |S_2|}}  \\
&       \sum_q \frac{ \W(q/Q)}{\varphi(q)}  \sumstar_{\chi \smod q} \  \sumtwo_{\m,\n \leq Q^2 }  \frac{a_\m(S_1) \chi(\m) }{\sqrt{\m }}  \frac{b_\n(S_2) \bar{\chi}(\n) }{ \sqrt{ \n }}  \int_\R   \bigg( \frac{ \n}{\m} \bigg)^{it}     \emts dt    + O \left(\frac Q{\log Q}\right).
\end{align*}
\end{lemma}
\begin{proof}
By the bound of $E_\ell (t)$ in Lemma \ref{lemma:explicit formula}, we obtain that each component of the main term of $C_{1, \underline G} $ in Lemma \ref{lem:truncatetoQsq} is bounded above by 
\begin{align*}
 &  \int_\mathbb R
        \sum_q \frac{ \W(q/Q)}{\varphi(q)}  \sumstar_{\chi \smod q} \bigg|  \sum_{\m  \leq Q^2 }  \frac{a_\m(S_1) \chi(\m) }{\m^{1/2 + it}}\bigg| \bigg| \sum_{\n   \leq Q^2 }  \frac{b_\n(S_2) \bar{\chi}(\n) }{\n^{1/2 - it}} \bigg|\frac{ \big(\log ( 2 + |t| ) \big)^{|S_4|} }{(\log Q)^{|S_1|+ |S_2| + |S_4|}}  e^{-t^2} \> dt
\end{align*}
Next, we apply the Cauchy-Schwarz inequality and the large sieve inequality (Lemma \ref{lemma:lsi}) and have that the above is bounded by
\begin{align*} 
        &   \int_\mathbb R  \frac{ \big(\log ( 2 + |t| ) \big)^{|S_4|} }{(\log Q)^{|S_1|+ |S_2| + |S_4|}}  e^{-t^2} \\
       & \times \bigg( \sum_q \frac{ \W(q/Q)}{\varphi(q)}  \sumstar_{\chi \smod q} \bigg|  \sum_{\m  \leq Q^2 }  \frac{a_\m(S_1) \chi(\m) }{\m^{1/2 + it}}\bigg|^2 \bigg)^{1/2}   \bigg(\sum_q \frac{ \W(q/Q)}{\varphi(q)}  \sumstar_{\chi \smod q}  \bigg| \sum_{\n   \leq Q^2 }  \frac{b_\n(S_2) \bar{\chi}(\n) }{\n^{1/2 - it}} \bigg|^2 \bigg)^{1/2}  \> dt \\
        &\ll  Q  \int_\mathbb R  \frac{ \big(\log ( 2 + |t| ) \big)^{|S_4|} }{(\log Q)^{|S_1|+ |S_2| + |S_4|}}    e^{-t^2} \bigg(  \sum_{ \m \leq Q^2 }  \frac{ |a_\m(S_1)|^2}{\m }   \sum_{ \n \leq Q^2 }  \frac{ |b_\n (S_2)|^2}{\n }  \bigg)^{1/2} \> dt \\
        &\ll  Q (\log Q)^{- |S_4|} . 
 \end{align*}
 Hence the contribution from $S_4 \neq \emptyset$ is at most $O(Q/\log Q) $.
 \end{proof}

Now we focus on the main term of Lemma \ref{lem:S4empty}. It is clear that the contribution of the case $ S_1 = S_2 = \emptyset$ is 
\begin{equation} \label{eqn:S1S2empty}
   D(\W, Q) \prod_{\ell=1 }^\nu  \widehat{F}_\ell (0).
   \end{equation}
If $S_2 = \emptyset$ but $S_1 \neq \emptyset$, then by (\ref{eqn:Fourierexsq}) the contribution from these terms is bounded by 
\es{ \label{bound:oneSempty} \ll \frac{Q}{\log Q}  \sum_{\m\leq Q^2 }  \frac{ |a_\m(S_1) | }{\sqrt{\m }} \left|  \int_\R   \bigg( \frac{ 1}{\m} \bigg)^{it}     \emts dt  \right| \ll \frac{Q}{\log Q}  \sum_{\m\leq Q^2 }  \frac{|a_\m(S_1)| e^{-(\log \m)^2/4}}{\sqrt{\m }} \ll \frac Q{\log Q}. }
The same holds for the case $ S_1 = \emptyset$ and $S_2 \neq \emptyset $ . Thus we can now consider the case $ S_1 , S_2 \neq \emptyset$. 

By repeated uses of the Cauchy-Schwarz inequality and Lemma \ref{lemma:lsi}, we can add the conditions such as $ \m, \n $ are squarefree with an error $O(Q/\log Q)$. For instance, the contribution of non-squarefree $\m$ is bounded by
\begin{align*} 
&  \frac{ 1}{(\log Q)^{|S_1|+ |S_2| }}  \int_\mathbb R
        \sum_q \frac{ \W(q/Q)}{\varphi(q)}  \sumstar_{\chi \smod q} \bigg|  \sum_{ \substack{ \m  \leq Q^2 \\  \mathrm{non-squarefree}}}  \frac{a_\m(S_1) \chi(\m) }{\m^{1/2 + it}}\bigg| \bigg| \sum_{\n   \leq Q^2 }  \frac{b_\n(S_2) \bar{\chi}(\n) }{\n^{1/2 - it}} \bigg|     e^{-t^2} \> dt\\ 
        &  \leq             \frac{1 }{(\log Q)^{|S_1|+ |S_2|  }}  \int_\mathbb R   e^{-t^2} \bigg( \sum_q \frac{ \W(q/Q)}{\varphi(q)}  \sumstar_{\chi \smod q} \bigg|  \sum_{ \substack{ \m  \leq Q^2 \\  \mathrm{non-squarefree}}}    \frac{a_\m(S_1) \chi(\m) }{\m^{1/2 + it}}\bigg|^2 \bigg)^{1/2}  \\
        & \hfill \cdot \bigg(\sum_q \frac{ \W(q/Q)}{\varphi(q)}  \sumstar_{\chi \smod q}  \bigg| \sum_{\n   \leq Q^2 }  \frac{b_\n(S_2) \bar{\chi}(\n) }{\n^{1/2 - it}} \bigg|^2 \bigg)^{1/2}  \> dt \\
        &\ll  Q   \frac{1 }{(\log Q)^{|S_1|+ |S_2| }} \int_\mathbb R    e^{-t^2} \bigg(  \sum_{ \substack{ \m  \leq Q^2 \\  \mathrm{non-squarefree}}}   \frac{ |a_\m(S_1)|^2}{\m }   \sum_{ \n \leq Q^2 }  \frac{ |b_\n (S_2)|^2}{\n }  \bigg)^{1/2} \> dt .
 \end{align*}
 Then we see that 
$$   \sum_{ \substack{ \m  \leq Q^2 \\  \mathrm{non-squarefree}}}   \frac{ |a_\m(S_1)|^2}{\m }  \leq \sum_{  \substack{ m_1 \cdots m_{|S_1|} \leq Q^2  \\   \mathrm{non-squarefree}} }  \frac{ \Lambda(m_1 )^2 \cdots \Lambda(m_{|S_1|})^2 }{ m_1 \cdots m_{|S_1|}} .$$
Each $m_\ell $ is supported in prime powers. If $m_1\cdots m_{|S_1|}$ is not squarefree, then at least one of the $m_\ell$'s is not squarefree, or all the $m_\ell$'s are squarefree but there are at least two of the $ m_\ell$'s having a common prime factor. Thus, we see that the above sum is 
\begin{align*}
  &  \leq   |S_1| \sum_{  \substack{ m_1 \cdots m_{|S_1|} \leq Q^2  \\  m_1  :   \mathrm{non-squarefree}} }  \frac{ \Lambda(m_1 )^2 \cdots \Lambda(m_{|S_1|})^2 }{ m_1 \cdots m_{|S_1|}} + \frac{ |S_1|(|S_1|-1)}{2}  \sum_{  \substack{ p_1 \cdots p_{|S_1|} \leq Q^2  \\  p_1 = p_2 } }  \frac{ ( \log p_1 )^2 \cdots (\log p_{|S_1|})^2 }{ p_1 \cdots p_{|S_1|}} \\ 
  & \ll ( \log Q)^{2 ( |S_1|-1)}.
  \end{align*} 
  We also see that 
  $$   \sum_{ \n \leq Q^2 }  \frac{ |b_\n (S_2)|^2}{\n } \ll ( \log Q)^{2|S_2|}. $$
  Hence, the contribution of the non-squarefree $\m$ is $O( Q / \log Q)$.
  
 Therefore,  $\m$  and $\n$ can be written as products of distinct primes as the following:
$$ \m = \prod_{ \ell \in S_1 } p_\ell , \qquad \n = \prod_{\ell \in S_2 } p_\ell . $$
  However, $\m, \n$ might have a common prime divisor. Let $(\m,\n) = \prod_{ \ell \in S_{11}} p_\ell = \prod_{\ell \in S_{21}} p_\ell $ for some $ S_{11} \subseteq S_1 $ and $ S_{21} \subseteq S_2 $. Then there is a unique bijection $\sigma : S_{11} \to S_{21}$ such that $  p_\ell = p_{\sigma(\ell)} $ for all $\ell \in S_{11}$. Moreover, since $ \widehat{F}_j$ is compactly supported, by the similar arguments to the proof of Lemma \ref{lem:truncatetoQsq},  we can remove the conditions $\m, \n \leq Q^2$  with error term of size $O\big(e^{-\tfrac {\varepsilon^2}{6}  (\log Q)^2}\big)$. Hence,
  \begin{equation}\label{eqn:C1GCG}
  	C_{1,\underline{G}} 
  	=   D(\mathcal W, Q)   \prod_{\ell =1 }^\nu  \widehat{F}_\ell (0) + \widetilde{C}_{ \underline{G}} + O\left(\frac{Q}{ \log Q }\right) , 
  	\end{equation}
  	where 
  	\est{\label{def:CtildeFg}
  	\widetilde{C}_{\underline{G}}&= \sum_{ \substack{  S_1 + S_2 + S_3=  [ \nu ] \\ S_1 , S_2 \neq \emptyset}  }   \bigg( \prod_{\ell \in S_3 } \widehat{F}_\ell (0) \bigg) \frac{  (-1)^{|S_1|+ |S_2|}}{(\log Q)^{|S_1|+ |S_2|}} \int_\R  \emts  \sum_{ \substack{S_{11}+S_{12} =  S_1 \\   S_{21}+ S_{22} = S_2  }} 
  	\sum_{ \substack{ \sigma : S_{11} \to S_{21}   \\   bijection} }   \\
  		& \times \sum_q \frac{ \W(q/Q)}{\varphi(q)}  \sumstar_{\chi \smod q}  \sum_{ \substack{P  } }\mu^2 (P)  \bigg( \prod_{\ell \in S_{11}}  \frac{   |\chi(p_\ell)|^2 ( \log p_\ell)^2 }{ p_\ell }  \widehat{F}_\ell  \bigg( - \frac{   \log p_\ell }{ \log Q} \bigg)  \widehat{F}_{ \sigma (\ell)}  \bigg(   \frac{   \log p_\ell }{ \log Q} \bigg) \bigg)      \\
  	&   \times          \sumtwo_{\substack{\m,\n  \\ (\m,\n)=1 \\(\m\n, P) = 1 } }  \frac{\mu^2 (\m) a_\m(S_{12}) \chi(\m) }{ \m^{1/2+it }}  \frac{\mu^2 (\n) b_\n(S_{22}) \bar{\chi}(\n) }{   \n^{1/2-it }} \> dt     ,
  }
and  $ P= \prod_{ \ell \in S_{11}}p_\ell.$ Note that the sum over $\m$ is 1 if $S_{12} = \emptyset$ and the sum over $\n$ is 1 if $S_{22} = \emptyset$ and the sum over $ \sigma $ is 1 if $ S_{11} = S_{21} = \emptyset$. When $S_{12} \neq \emptyset$ and $S_{22} = \emptyset$, 
  one can show that
$$ \int_{\mathbb R} e^{-t^2}\sum_{\substack{ \m  \\(\m, P) = 1 } }  \frac{\mu^2 (\m) a_m(S_{12}) \chi(\m) }{ \m^{1/2+it }}   \> dt  = O(1)  $$
by the same method as in (\ref{bound:oneSempty}) and its contribution to $\widetilde{C}_{\underline{G}}$ is $O( Q/ \log Q)$. The same holds for the case $S_{12} = \emptyset$ and $S_{22} \neq \emptyset$. Let $D_{\underline{G}}$ be the above sum with the additional conditions $S_{12}, S_{22} = \emptyset$ and $ N_{\underline{G}}$ be the above sum with the additional conditions $ S_{12}, S_{22} \neq \emptyset$. Then we see that
\begin{equation}\label{eqn:CGDGNG}
 \widetilde{C}_{ \underline{G}} = D_{\underline{G}} + N_{\underline{G}} + O(Q/\log Q). 
\end{equation}

The term $D_{\underline{G}}$ is so-called ``diagonal terms" and the term $N_{\underline{G}}$ is ``off-diagonal terms''.  $D_{\underline{G}}$ has a relatively simple representation as
  	\est{   	D_{\underline{G}}&= \sqrt{\pi} \sum_{ \substack{  S_1 + S_2 + S_3=  [ \nu ] \\ S_1 , S_2 \neq \emptyset}  }   \bigg( \prod_{\ell \in S_3 } \widehat{F}_\ell (0) \bigg) \frac{  (-1)^{|S_1|+ |S_2|}}{(\log Q)^{|S_1|+ |S_2|}} 
  	\sum_{ \substack{ \sigma : S_1 \to S_2   \\   bijection} }   \\
  		& \times \sum_q \frac{ \W(q/Q)}{\varphi(q)}  \sumstar_{\chi \smod q}  \sum_{ \substack{P  } }\mu^2 (P)  \bigg( \prod_{\ell \in S_1}  \frac{   |\chi(p_\ell)|^2 ( \log p_\ell)^2 }{ p_\ell }  \widehat{F}_\ell  \bigg( - \frac{   \log p_\ell }{ \log Q} \bigg)  \widehat{F}_{ \sigma (\ell)}  \bigg(   \frac{   \log p_\ell }{ \log Q} \bigg) \bigg)    ,
  }
where $ P = \prod_{ \ell \in S_1 } p_\ell$. Then we can obtain
\es{ \label{eqn:DG}  D_{\underline{G}} &= D(\mathcal W, Q) \sum_{ \substack{  S_1 + S_2 + S_3=  [ \nu ] \\ S_1 , S_2 \neq \emptyset \\ |S_1| = |S_2| } }   \bigg( \prod_{\ell \in S_3 } \widehat{F}_\ell (0) \bigg)  
	\sum_{ \substack{ \sigma : S_{1} \to S_{2}   \\   bijection} }  \bigg(\prod_{\ell \in S_1} \int_0^\infty v \widehat{F}_\ell (-v) \widehat{F}_{\sigma(\ell)} (v) dv \bigg)   + O\left(\frac{Q}{\log Q}\right) }
	by the following lemma. 
\begin{lemma}  \label{lem:diagonalCD}
Let $\chi$ be a primitive Dirichlet character mod $q \in [Q, 2Q]$ and $ P= \prod_{ \ell \in S_{1}}p_\ell$. Then 
\begin{align*}
	 & \sum_P \mu^2 (P)  \bigg( \prod_{\ell \in S_{1}}  \frac{   |\chi(p_\ell)|^2 ( \log p_\ell)^2 }{ p_\ell }  \widehat{F}_\ell  \bigg( - \frac{   \log p_\ell }{ \log Q} \bigg)  \widehat{F}_{ \sigma (\ell)}  \bigg(   \frac{   \log p_\ell }{ \log Q} \bigg) \bigg)\\
	 & =  ( \log Q)^{2 |S_{1}|} \prod_{\ell \in S_{1}}  \int_0^\infty v \widehat{F}_\ell (-v) \widehat{F}_{\sigma(\ell)} (v) dv  +  O(  ( \log Q)^{2 |S_{1}|-1} ).
	 \end{align*}
	 \end{lemma}

\begin{proof}
 By the inclusion-exclusion principle, the prime number theorem and the fact that 
	$ \sum_{p} (\log p)^r p^{-\alpha}$ is uniformly convergent  and bounded for $\alpha \geq 2$ and $r \leq 2|S_1|$, we have that 
	\begin{align*}
 &  \sum_{ \substack{P   \\ (P, q)=1} }\mu^2 (P)  \bigg( \prod_{\ell \in S_{1}}  \frac{   |\chi(p_\ell)|^2 ( \log p_\ell)^2 }{ p_\ell }  \widehat{F}_\ell  \bigg( - \frac{   \log p_\ell }{ \log Q} \bigg)  \widehat{F}_{ \sigma (\ell)}  \bigg(   \frac{   \log p_\ell }{ \log Q} \bigg) \bigg) \\
	& = \prod_{\ell \in S_{1}} \bigg(  \sum_{ (p, q)=1}  \frac{   ( \log p)^2 }{ p }  \widehat{F}_\ell  \bigg( - \frac{   \log p }{ \log Q} \bigg)  \widehat{F}_{ \sigma (\ell)}  \bigg(   \frac{   \log p }{ \log Q} \bigg)    \bigg) + O( (\log Q)^{ 2 |S_{1}| - 2}). 
	\end{align*}
	Since the number of primes diving $q$ is $O(\log q)$, the above is
	\begin{align*}
	& = \prod_{\ell \in S_{1}} \bigg(  \sum_p  \frac{   ( \log p)^2 }{ p }  \widehat{F}_\ell  \bigg( - \frac{   \log p }{ \log Q} \bigg)  \widehat{F}_{ \sigma (\ell)}  \bigg(   \frac{   \log p }{ \log Q} \bigg)  + O( \log Q )     \bigg) + O( (\log Q)^{ 2 |S_{1}| - 2}).
	\end{align*}
	By the prime number theorem and the partial summation, we obtain that 
	 $$ \sum_p  \frac{   ( \log p)^2 }{ p }  \widehat{F}_\ell  \bigg( - \frac{   \log p }{ \log Q} \bigg)  \widehat{F}_{ \sigma (\ell)}  \bigg(   \frac{   \log p }{ \log Q} \bigg)  = (\log Q)^2\int_0^\infty v \widehat{F}_\ell (-v) \widehat{F}_{\sigma(\ell)} (v) dv + O(1) . $$
Thus the lemma holds.	
	
\end{proof}
 
Therefore, by \eqref{eqn:C1GCG}, \eqref{eqn:CGDGNG} and \eqref{eqn:DG} we have
\es{ \label{eqn:C1GNG}	C_{1,\underline{G}}  =&    D(\mathcal W, Q) \sum_{    S_1 + S_2 + S_3=  [ \nu ]  }   \bigg( \prod_{\ell \in S_3 } \widehat{F}_\ell (0) \bigg)  
	\sum_{ \substack{ \sigma : S_{1} \to S_{2}   \\   bijection} } \bigg(\prod_{\ell \in S_1 }
 \int_0^\infty v \widehat{F}_\ell (-v) \widehat{F}_{\sigma(\ell)} (v) dv \bigg) \\
	& + N_{\underline{G}}     +  O\left(\frac{Q}{ \log Q }\right)} .

\section{Calculation of $N_{\underline{G}} $} \label{sec:ALS}
In this section we will calculate $N_{\underline{G}} $ defined in a line ahead of  \eqref{eqn:CGDGNG} using the asymptotic large sieve method. By the definition of $N_{\underline{G}} $ and switching summations, it can be written as
\es{ \label{eqn:CNg}
N_{\underline{G}}  &= \sum_{    S_1 + S_2 + S_3=  [ \nu ]    }   \bigg( \prod_{\ell \in S_3 } \widehat{F}_\ell (0) \bigg) \frac{  (-1)^{|S_1|+ |S_2|}}{(\log Q)^{|S_1|+ |S_2|}}  \sum_{ \substack{S_{11}+S_{12} =  S_1 \\   S_{21}+ S_{22} = S_2 \\ |S_{11}| = |S_{21}| \\ S_{12}, S_{22} \neq \emptyset }}
\sum_{ \substack{ \sigma : S_{11} \to S_{21}   \\   bijection} }   \\
& \times  \sum_{ \substack{P  } }\mu^2 (P)  \bigg( \prod_{\ell \in S_{11}}  \frac{   ( \log p_\ell)^2 }{ p_\ell }  \widehat{F}_\ell  \bigg( - \frac{   \log p_\ell }{ \log Q} \bigg)  \widehat{F}_{ \sigma (\ell)}  \bigg(   \frac{   \log p_\ell }{ \log Q} \bigg) \bigg) \Scal(P; S_{12}, S_{22}), 
}
where $ P = \prod_{ \ell \in S_{11}} p_\ell$ and $\Scal(P; S_{12}, S_{22})$ denotes
\es{ \label{Spartition}    \int_\R  \emts  \sum_{\substack{q \\ (q,P) = 1}} \frac{ \W(q/Q)}{\varphi(q)}  \sumstar_{\chi \smod q}        \sumtwo_{\substack{ \m,\n  \\ (\m,\n)=1 \\(\m\n, P) = 1 } }  \frac{\mu^2 (\m) a_\m(S_{12}) \chi(\m) }{ \m^{1/2+it }}  \frac{\mu^2 (\n) b_\n(S_{22}) \bar{\chi}(\n) }{   \n^{1/2-it }} \> dt     .
}
Note that $n$ is the positive integer introduced in Section \ref{sec:intro} and let $k$ and $r$ be positive integers with $k + r \leq n$. Due to the factor $\mu^2 (P)$, $P$ is supported in  squarefree positive integers and the number of prime divisors of $P$ is less than or equal to $n$. We start by estimating $\Scal(P; \{1,... , k\}, \{k+1,..., k+r\})$, which is a special case of $\Scal(P;S_{12}, S_{22} )$. It will be apparent that our treatment of $\Scal(P; \{1,... , k\}, \{k+1,..., k+r\})$ can be generalized to deal with $\Scal(P;S_{12}, S_{22} )$.

      \begin{proposition}\label{prop:als} 
Define
\begin{align*}
   \Scal  :=&  \ \Scal(P; \{1,... , k\}, \{k+1,..., k+r\})  \\
   = &  \int_{-\infty}^{\infty} e^{-t^2} \sum_{\substack{q \\ (q, P) = 1}} \FW \sumstar_{ \chi  \shortmod q }    \sumtwo_{ \substack{  \m, \n   \\ (\m,\n)=1 \\ (\m\n, P) = 1}    } \frac{a_\m b_\n \chi(\m) \bar{\chi}(\n)}{\sqrt{\m\n}} \pr{\frac \n\m}^{it} \> dt,
   \end{align*}
 where
\es{ \label{def:ambn} a_\m & =\mu^2 (\m)  \sum_{  p_1 \cdots p_k = \m  } \bigg(  \prod_{j=1}^k \log p_j   \widehat{F}_j \left( -\frac{\log p_j }{\log Q} \right) \bigg)   \\  
b_\n &= \mu^2(\n)  \sum_{  p_{k+1} \cdots p_{k+r} = \n  }    \bigg( \prod_{j=k+1}^{k+r}   \log p_j   \widehat{F}_j \left( \frac{\log p_j }{\log Q} \right)\bigg)    .} 
 Suppose that $ \widehat{F}_i  ( u)$ is supported in $|u| \leq \kappa_i$ for $ 1 \leq i \leq k + r$. Also for fixed $\varepsilon > 0,$ we assume that $ \kappa' + \kappa'' \leq 4 - \varepsilon $, where $ \kappa' = \sum_{i=1}^k \kappa_i $ and $ \kappa'' = \sum_{j=1}^r \kappa_{k+j} $. Then 
 \begin{align*}
& \Scal(P; \{1,... , k\}, \{k+1,..., k+r\}) \\
& \ \ \ \  =  Q (\log Q)^{k+r}    \sqrt{\pi}      \widetilde{\W}(1)  \prod_{p \, \nmid P} \bigg( 1 - \frac{1}{p^2}- \frac{1}{ p^3} \bigg)  \prod_{p | P} \pr{1 - \frac 1p} \mc {I}(k,r)  \\
 & \hskip 1in + O\left( Q(\log Q)^{k + r - 1}  \log \log Q  \right),
 \end{align*} 
  where 
  \est{   \mc {I}(k,r)   
 &:= \sum_{ \substack{ 1 \leq j_1 \leq k \\  k+1 \leq j_2 \leq k+r } }       \summany_{\substack{T_1,W_1, T_2, W_2, T_3, W_3  \subseteq [k+r] \\ T_1 + W_1 = \{1,. ..,j_1 -1\} \cup \{k + 1,...,  j_2 -1\} \\ T_2 + W_2 = \{ j_1 + 1,..., k\} \\ T_3 + W_3 = \{j_2  + 1, ..., k+r\} }} (-1)^{j_1  + r+ |W_2 |+| W_3|}    \\ &        \int_{ \substack{ \mc D_{k+r }(\vec{T}, \vec{W})  \\ u_{ j_1 } + u(\vec{T} )   > 1 }  } \left( \prod_{j=1}^{k+r}  \widehat F_j(-u_j) \right)       (1- u_{ j_1 } - u(\vec{T} )  ) \delta( u ([k+r]) )   \> d\vecu  ,}
  \est{ \mc D_{k+r }(\vec{T}, \vec{W}) &:= \mc D_{k+r}(T_1, T_2, T_3, W_1, W_2, W_3) \\
 	& := \bigg\{ \vecu \in \mathbb R^{k + r} : u_j < 0  \ \textrm{for} \ j \in T_1 \cup T_3 \cup W_3, \ \textrm{and} \  u_j > 0 \ \textrm{for} \ j \in T_2  \cup W_1  \cup W_2 \bigg\} ,}
$\delta(x)$ is the Dirac delta function, $ \vecu = ( u_1 , \dots, u_{k+r})$,  $ d\vecu = du_1 \cdots du_{k+r}$, $u(S) := \sum_{ j \in S } u_j $ for $S \subseteq [k+r]= \{ 1, \dots, k+r\}$ and $ u(\vec{T} ) := u( T_1 ) +  u(T_2) + u(T_3)  $.
\end{proposition} 
 In later application, $\kappa_i = \sum_{ j \in G_i } \eta_j $ when $ F_i (x) = \prod_{ j \in G_i } f_j ( x ) $, as in \eqref{def:Fell} and \eqref{kappa ell def}. 
We need new notations to extend Proposition \ref{prop:als} to general cases, so we will postpone it and complete the estimation of $N_{\underline{G}}$ in Section \ref{complete NG}.

\subsection*{Outline of the proof of Propotion \ref{prop:als}}
 Since there are various techical details in this section, we outline the key points of the proof here. 
 Roughly speaking, after orthogonality relation of Dirichlet characters, we shall study the sum of the form
$$ S \approx \sum_{q \asymp Q} \frac{1}{\varphi(q)} \sum_{\substack{d|q \\ d | \m - \n}} \varphi(d) \mu \left( \frac{q}{d}\right)\sumtwo_{  \m, \n \ll Q^{2-\epsilon}} \frac{a_\m b_\n}{\sqrt{\m\n}} ,$$
where $a_\m, b_\n$ are supported in squarefree numbers defined in \eqref{def:ambn}. Then we can divide the sum into two cases: small $d$ and big $d$. Thus $\mathcal S = \mathcal S_U + \mathcal S_L$.  

  For $\mathcal S_U$, the contribution from small $d$, we express the condition $d| \m-\n$ as the character sum $\sum_{\Psi \mod d } \Psi(\m) \overline{\Psi}(\n)$. The contribution from the principal character, say $\mathcal M_U$, is large, and the corresponding terms to non-pricipal characters are negligible. These will be proved in Lemma \ref{lem:SU}. Though $\mathcal M_U$ is large, it will cancel with one of the main terms from $\mathcal S_L$. 

 For $\mathcal S_L$, the key ingredient is the complementary divisor trick, which is $\m - \n = d g$ (ignoring gcd conditions). We then express the condition $d| \m-\n$ as the character sum $\sum_{\Psi \mod g } \Psi(\m) \overline{\Psi}(\n)$ instead. Note that since $d$ is large, and $\m, \n \ll Q^{2 - \epsilon}$, $g$ is small. In the critical range, which is $\m, \n \asymp Q^{2 - \epsilon}$ and $d \asymp Q$, we obtain that $g \asymp Q^{1 - \epsilon}$. The smaller conductor $g$ allows us to bound error terms. Note that without integration over $t$, the size of $\m, \n$ could go up to $Q^{4- \epsilon},$ and the size of $g$ would be too large to obtain negligible error terms.  

 Again the main term of $\mathcal S_{L}$ , say $\mathcal S_{L,0}$, is from the principal character (see Section \ref{proof of prop part 2}). Then we apply the Mellin inversion to write $S_{L, 0}$ in terms of contour integration of over $s$ and $z$ as in Lemma \ref{lemma 5.4}. The main terms come from the residue at $s = 0$  (from the factor $\zeta (1-s)$) and $s = 1$ (from the factors $\Gamma(1-s)$ and $B(-s)$). The contribution from the residue at $s = 0$ is cancelled out with $\mathcal M_U$. Then technical manipulation, e.g. inclusion-exclusion and Fourier transform, in Lemma \ref{lemma 5.6} - \ref{lemma 5.8}  is done to express the contribution of the residue at $s = 1$ in the form that will be easily compared with conjectures from random matrix in Section \ref{sec:RMT}.

\subsubsection*{ Proof of Proposition \ref{prop:als}:} We start from applying the orthogonality relation of Dirichlet characters and obtain that
\begin{equation*}\begin{split}
  \sum_{\substack{q \\ (q, P) = 1}} \FW \sumstar_{ \chi \shortmod q } \chi(\m) \overline{\chi} (\n)    
&=  \sum_{\substack{ q \\   (q, \m\n P ) = 1}} \FW       \sum_{\substack{ d| q \\ d| \m - \n }} \varphi(d) \mu\left( \frac{q}{d} \right)  \\
&  = \sumtwo_{ \substack{ c,d  \\ d| \m - \n \\ (\m\n P,cd ) = 1 } } \mu(c)  \frac{ \varphi(d) }{ \varphi(cd)}   \W \pr{\frac{cd}{Q}}   .
\end{split}\end{equation*}
Since $ \m$ is supported in products of $k$ distinct primes, $ \n $ is supported in products of $r $ distinct primes and  $(\m,\n)=1$, 
$$ \m \neq \n  .$$
   We have that 
\es{\label{eqn:splittoUL}
   \Scal &= \int_{-\infty}^{\infty} e^{-t^2}      \sumtwo_{\substack{\m, \n  \\ (\m,\n)=1 \\ (\m\n, P) = 1}} \frac{a_\m b_\n  }{\sqrt{\m\n}} \pr{\frac \n\m}^{it} \sumtwo_{ \substack{ c,d  \\ d| \m - \n \\ (\m\n P, cd) = 1 } } \mu(c)   \frac{ \varphi(d) }{ \varphi(cd)}    \W \pr{\frac{cd}{Q}}   \> dt\\
&=: \Scal_U  + \Scal_L ,
}
say, where $\Scal_U $ is the sum over $c > C$  (i.e. $d$ is small), and $\Scal_L $ is the sum over $c \leq C$  (i.e. $d$ is big)  with
$C = Q^{ \varepsilon_1 }$ for some $\varepsilon_1 > 0$ to be determined later. The remaining part of the proof will be given in Section \ref{proof of prop part 1} -- Section \ref{proof of prop part 4}

\subsection{Evaluating $\Scal_{U} $}\label{proof of prop part 1} 

In this section we will prove the following lemma. 

\begin{lemma} \label{lem:SU} Let all notations be as above.  Then for any $\epsilon  > 0$
	\est{\Scal_U = \mathcal {M}_{U}    +O \left( C Q^{(\kappa'+\kappa'')/2 -1+\epsilon }  + \frac{Q^{1 + \epsilon }}{C}\right), }
where
\es{\label{def:MSU}\mathcal {M}_U := -  \widetilde{\W}(0) &  B(0)    \int_{-\infty}^{\infty} e^{-t^2}      \sum_{\substack{\m, \n \\ (\m,\n)=1 \\ (\m\n, P) = 1}} \frac{a_\m b_\n  }{\sqrt{\m\n}} \pr{\frac \n\m}^{it}  \sum_{ \substack{  c\leq C   \\ ( c  ,\m\n P  )=1    } } \frac{ \mu(c) B_2(0 ,c) }{\varphi(c)}  B_1(0 ,\m\n P) \> dt}
and $B_2(0, c)$ and $B_1(0, \m\n P)$ are defined as in Lemma \ref{lemma:B sum}.
	
\end{lemma} 

\begin{proof}
 Let
$$ \Scal_U(\m,\n) := \sumtwo_{ \substack{ c > C ,\ d  \\ d| \m - \n \\ (\m\n P, cd) = 1 } } \mu(c)   \frac{ \varphi(d) }{ \varphi(cd)}    \W \pr{\frac{cd}{Q}},$$
then
\begin{equation*}\begin{split}
  \Scal_U &   =   \int_{-\infty}^{\infty} e^{-t^2}      \sum_{\substack{\m, \n \\ (\m,\n)=1  \\ (\m\n, P)=1 }} \frac{a_\m b_\n  }{\sqrt{\m\n}} \pr{\frac \n\m}^{it} \Scal_U (\m,\n)   \> dt .
\end{split}\end{equation*}
 Replacing the condition $ d|\m-\n$ by the orthogonality relation of a character sum, we have 
\begin{equation}\label{eqn:U}\begin{split}
\Scal_U (\m,\n) &=   \sumtwo_{ \substack{  c>C ,\ d   \\ ( cd ,\m\n P  )=1    } }   \frac{ \mu(c) }{ \varphi(cd)}    \W \pr{\frac{cd}{Q}} \sum_{\Psi \shortmod d} \Psi (\m) \overline{\Psi }(\n)  \\  
 &= \sumtwo_{ \substack{  c>C ,\ d   \\ ( cd ,\m\n P  )=1    } }  \frac{ \mu(c) }{ \varphi(cd)}     \W \pr{\frac{cd}{Q}}  + \sumtwo_{ \substack{  c>C ,\ d   \\ ( c ,\m\n P  )=1 \\ (d, P) = 1   } } \frac{ \mu(c) }{ \varphi(cd)}     \W \pr{\frac{cd}{Q}}\sum_{\Psi \neq \Psi_0 \shortmod d} \Psi (\m) \overline{\Psi }(\n)  \\
&=: \Scal_{U,0} (\m,\n)   + \Scal_{U,E}(\m,\n).
 \end{split}\end{equation}
We first evaluate $\Scal_{U, 0}(\m,\n) $.  Since  $\sum_{ c |q } \mu(c) = 0$ for $q > 1$, we see that 
$$ \sum_{ \substack{c|q \\ c >C}} \mu(c) = -  \sum_{ \substack{c|q \\ c \leq C}} \mu (c) $$
for $ q > 1 $. Hence, by  writing $\W$ in terms of its Mellin transform $\widetilde{\W}$, we have
\begin{equation*}\begin{split} 
 \Scal_{U,0}(\m,\n)  = &   -  \sumtwo_{ \substack{  c\leq C ,\ d   \\ ( cd ,\m\n P  )=1    } } \frac{ \mu(c) }{ \varphi(cd)}   \W \pr{\frac{cd}{Q}} \\
  =&  -  \sumtwo_{ \substack{  c\leq C ,\ d   \\ ( cd ,\m\n P  )=1    } } \frac{ \mu(c) }{ \varphi(cd)}   \frac{1}{2 \pi i } \int_{(2)} \widetilde{\W}(s)\left( \frac{ Q}{cd} \right)^s ds .
\end{split}\end{equation*}
Applying Lemma \ref{lemma:B sum} to the sum over $d$, we have
\est{
\Scal_{U, 0} (\m,\n) & = -   \frac{1}{2 \pi i } \int_{(2)} \widetilde{\W}(s) Q^s    \sum_{ \substack{  c\leq C   \\ ( c, \m\n P  )=1    } } \frac{ \mu(c) B_2(s ,c) }{\varphi(c) c^s} \zeta(1+s) B(s) B_1(s, \m\n P)    ds .
}
We move the $s$-contour to $(-1+\epsilon)$ and encounter a simple pole at $s = 0.$ Then for any small $\epsilon  > 0,$ 
\est{
 \Scal_{U, 0}(\m, \n)  
& =  -  \widetilde{\W}(0)   B(0)   \sum_{ \substack{  c\leq C   \\ ( c  ,\m\n  P)=1    } } \frac{ \mu(c) B_2(0 ,c) }{\varphi(c)}  B_1 (0 , \m\n P)   + O( C Q^{-1+\epsilon } )  
}
 by Lemma \ref{lemma:B sum}. Note that the $O$-term above depends on $ 2^{\omega(\m\n  P)} \leq 2^n $. 
Hence by the support of $\widehat F_\ell$ in Proposition \ref{prop:als}, we have
 \begin{equation*}
\begin{split}
  \Scal_{U,0} : = &    \int_{-\infty}^{\infty} e^{-t^2}      \sumtwo_{\substack{\m, \n \\ (\m,\n)=1 \\ (\m\n, P) = 1 }} \frac{a_\m b_\n  }{\sqrt{\m\n}} \pr{\frac \n\m}^{it} \Scal_{U,0}(\m,\n)    \> dt =    \mathcal {M}_U   + O( C Q^{(\kappa'+\kappa'')/2 -1+\epsilon } ).
\end{split}\end{equation*}
We next consider $\Scal_{U,E}(\m,\n)$. Define
\begin{equation*}\begin{split}
  \Scal_{U,E} &  : =  \int_{-\infty}^{\infty} e^{-t^2}   \sum_{\substack{\m,\n  \\ (\m,\n)=1 \\ (\m\n, P) = 1 }}     \frac{a_\m b_\n}{\sqrt{\m\n}} \pr{\frac \n\m}^{it}\Scal_{U,E}(\m,\n) \> dt \\
  &   =  \int_{-\infty}^{\infty} e^{-t^2}    \sumtwo_{  \substack{c>C ,\ d \\ (cd, P) = 1}   }  \frac{ \mu(c) }{ \varphi(cd)}   \W \pr{\frac{cd}{Q}} \sum_{\Psi \neq \Psi_0 \shortmod d} \sum_{\substack{\m,\n  \\ (\m,\n)=1 \\ (\m\n, cP)=1  }}     \frac{a_\m \Psi (\m) b_\n\overline{\Psi }(\n) }{\sqrt{\m\n}} \pr{\frac \n\m}^{it}     \> dt .
\end{split}\end{equation*}
By Lemma \ref{lemma:primecharsum1}, we obtain that
 \begin{equation*}\begin{split}
  \Scal_{U,E} &          \ll  (\log Q)^{2(k + r)}   \sumtwo_{  c>C ,d   }  \frac{ \varphi(d)  }{ \varphi(cd)}   \W \pr{\frac{cd}{Q}}  \ll  \frac{Q^{1+ \epsilon }}{C}
\end{split}\end{equation*}
for any $\epsilon >0$.   We derive the lemma from the fact that $ \Scal_U = \Scal_{U, 0} + \Scal_{U, E}\ .$

\end{proof}

\subsection{Evaluating $\Scal_{L} $}\label{proof of prop part 2} In this section, we will treat the terms with $c \leq C.$ We write
\begin{equation*}\begin{split}
  \Scal_L &   =   \int_{-\infty}^{\infty} e^{-t^2}      \sumtwo_{\substack{\m, \n \\ (\m,\n)=1 \\ (\m\n, P) = 1 }} \frac{a_\m b_\n  }{\sqrt{\m \n}} \pr{\frac \n\m}^{it}\Scal_L (\m,\n)   \> dt , 
\end{split}\end{equation*}
where
$$ \Scal_L (\m,\n):=  \sumtwo_{ \substack{ c \leq C ,d  \\ d| \m - \n \\ (\m\n P, cd) = 1 } } \mu(c)   \frac{ \varphi(d) }{ \varphi(cd)}    \W \pr{\frac{cd}{Q}}.$$
The conditions $(\m,\n)=1$ and $ d|\m-\n$ imply $ (\m\n,d)=1$, so that we can remove the condition $(\m\n,d)=1$ in the sum. By the identity
$$\frac{\varphi(d)}{\varphi(cd)} = \frac{1}{\varphi(c)} \sum_{ \substack{a|c \\ a | d } } \frac{\mu(a)}{a} ,$$
 we obtain that  
\begin{equation*}\begin{split} 
  \Scal_L (\m,\n) 
 =& \sumtwo_{ \substack{ c \leq C , \ d  \\ d| \m - \n \\ (\m\n P, c) = (P, d) = 1 } } \sum_{ \substack{a|c \\ a | d } } \frac{\mu(a)\mu(c)}{a\varphi(c)}    \W \pr{\frac{cd}{Q}} \\
 =& \sum_{ \substack{ c \leq C    \\ (\m\n P, c) = 1 } } \sum_{ a|c   } \frac{\mu(a)\mu(c)}{a\varphi(c)}  \sum_{ \substack{ d  \\ ad| \m - \n \\ (d, P) = 1}}  \W \pr{\frac{acd}{Q}}  \\
 =& \sum_{ \substack{ c \leq C    \\ (\m\n P, c) = 1 } } \sum_{ a|c   } \frac{\mu(a)\mu(c)}{a\varphi(c)} \sum_{\ell | P} \mu (\ell) \sum_{ \substack{ d  \\ ad\ell| \m - \n }}  \W \pr{\frac{acd\ell}{Q}}
\end{split}\end{equation*} 
 We substitute the sum over $d$ by the sum over $g$ through the condition $ad\ell g=|\m-\n|$ and  then write the condition $ a\ell g|\m-\n$ in term of Dirichlet characters. Hence
\begin{equation*}\begin{split} 
 \Scal_L (\m,\n) 
 &= \sum_{ \substack{ c \leq C    \\ (\m\n P, c) = 1 } } \sum_{ a|c   } \frac{\mu(a)\mu(c)}{a\varphi(c)}  \sum_{\ell | P} \mu (\ell) \sum_{ \substack{ g  \\ a\ell g| \m - \n}}  \W \pr{\frac{c|\m-\n|}{gQ}} \\
 &= \sum_{ \substack{ c \leq C    \\ (\m \n P, c) = 1 } } \sum_{ a|c   } \frac{\mu(a)\mu(c)}{a\varphi(c)} \sum_{\ell | P} \mu (\ell) \sum_{ g}  \W \pr{\frac{c|\m-\n|}{gQ}} \frac{1}{\varphi(a\ell g)} \sum_{  \Psi \shortmod {a\ell g}} \Psi(\m) \overline{\Psi}(\n)  \\
 &=:  \Scal_{L,0} (\m,\n) + \Scal_{L,E} (\m,\n) ,  
\end{split}\end{equation*}
where
$$ \Scal_{L,0} (\m,\n) = \sum_{ \substack{ c \leq C    \\ (\m\n P, c) = 1 } } \sumthree_{\substack{ a|c ,\ g , \ \ell |P\\ (a\ell g,\m\n)=1 }  } \frac{\mu(a)\mu(c)\mu(\ell) }{a\varphi(c) }    \W \pr{\frac{c|\m-\n|}{gQ}} \frac{1}{\varphi(a\ell g)}  $$
and 
$$ \Scal_{L,E} (\m,\n) = \sum_{ \substack{ c \leq C    \\ (\m\n P, c) = 1 } } \sum_{ a|c   } \frac{\mu(a)\mu(c)}{a\varphi(c)} \sum_{\ell | P} \mu(\ell ) \sum_{ g}  \W \pr{\frac{c|\m-\n|}{gQ}} \frac{1}{\varphi(a\ell g)} \sum_{  \substack{ \Psi \shortmod {a\ell g}  \\    \Psi \neq \Psi_0  }} \Psi(m) \overline{\Psi}(n) .$$
We remark that $(c,\m\n P) = 1$ and $a | c$ imply that $(a, \m\n P) = 1$. Define
\es{ \label{def:SL0} \Scal_{L,0} =  \int_{-\infty}^{\infty} e^{-t^2}      \sumtwo_{\substack{\m, \n \\ (\m,\n)=1 \\(\m\n, P)  = 1 }} \frac{a_\m b_\n  }{\sqrt{\m\n}} \pr{\frac \n\m}^{it} \Scal_{L,0} (\m,\n) \  dt }
and 
\es{\label{def:SLE} \Scal_{L,E} =  \int_{-\infty}^{\infty} e^{-t^2}      \sumtwo_{\substack{\m, \n \\ (\m,\n)=1 \\ (\m\n, P) = 1 }} \frac{a_\m b_\n  }{\sqrt{\m\n}} \pr{\frac \n\m}^{it} \Scal_{L,E} (\m,\n) \  dt , }
 so that
 $$ \Scal_L= \Scal_{L,0} + \Scal_{L,E} . $$

We first estimate $\Scal_{L,E}$.  
\begin{lemma} \label{lem:SLE} Let $\Scal_{L, E}$ be defined in (\ref{def:SLE}). Then for any $\epsilon > 0,$
$$  \Scal_{L, E} \ll C Q^{-1 + (\kappa' + \kappa'')/2 + \epsilon}, $$
where $\kappa' $ and $\kappa''$ are defined as in Proposition \ref{prop:als}.
\end{lemma} 
\begin{proof} We write out $S_{L, E}$ as
\begin{equation*}\begin{split}
 \Scal_{L,E} = &  \int_{-\infty}^{\infty} e^{-t^2}      \sumtwo_{\substack{\m, \n \\ (\m,\n)=1  \\ (\m\n, P) = 1}} \frac{a_\m b_\n  }{\sqrt{\m\n}} \pr{\frac \n\m}^{it} \sum_{ \substack{ c \leq C    \\ (\m\n P, c) = 1 } } \sum_{ a|c   } \frac{\mu(a)\mu(c)}{a\varphi(c)} \sum_{\ell | P} \mu(\ell ) \sum_{ g}  \W \pr{\frac{c|\m-\n|}{gQ}}  \\
 &\hskip 1 in \times \frac{1}{\varphi(a\ell g)}\sum_{  \substack{ \Psi \shortmod {a\ell g}  \\    \Psi \neq \Psi_0  }} \Psi(\m) \overline{\Psi}(\n)  \  dt .
 \end{split}\end{equation*}
If $ \m$ or $ \n$ is greater than $ Q^{ (\kappa'+\kappa'')/2 + \epsilon_1}$ for $ \epsilon_1 > 0$, then
$$ \min(\m,\n) \leq Q^{\min(\kappa' , \kappa'') } < Q^{ (\kappa'+\kappa'')/2 + \epsilon_1 } \leq \max (\m,\n),$$
and
$$   \frac{ \max(\m,\n) }{ \min(\m,\n) } \geq Q^{ (\kappa'+\kappa'')/2+\epsilon_1 - \min( \kappa' , \kappa'') } \geq Q^{\epsilon_1 } .$$
It then follows that
$$ \int_{-\infty}^\infty e^{-t^2 } \bigg( \frac{\n}{\m} \bigg)^{it} dt \ll e^{- ( \log \n/\m)^2 /4} \ll e^{ - ( \epsilon_1^2  /4)   ( \log Q)^2 } . $$
Hence, we can restrict the range of $\m, \n$ up to $ Q^{ (\kappa'+\kappa'')/2 + \epsilon_1} $ with an error of size $O(Q^{-A})$ for any positive integers $A.$ 
For $ \m,\n$ in this range,  we have
$$ |\m-\n| \leq 2Q^{ (\kappa'+\kappa'')/2 + \epsilon_1} .$$
Since $\W $ is supported in the interval $[1,2]$, if 
$$ \W  \pr{\frac{c|\m-\n|}{gQ}} \neq 0 ,$$
then
$$g \leq \frac{c|\m-\n|}{Q} \leq   2cQ^{ -1+ (\kappa'+\kappa'')/2 + \epsilon_1}.   $$
Therefore we add the condition $ g \leq \tilde{g}:= 2cQ^{ -1+ (\kappa'+\kappa'')/2 + \epsilon_1 }$ and then remove the restriction $ \m,\n \leq Q^{ (\kappa'+\kappa'')/2 + \epsilon_1} $ from the sum over $\m, \n$ with an additional error $O(Q^{-A})$.
Thus,
\begin{equation*}\begin{split}
 \Scal_{L,E} = &  \int_{-\infty}^{\infty} e^{-t^2}      \sumtwo_{\substack{\m, \n \\ (\m,\n)=1  \\ (\m\n, P) = 1}} \frac{a_\m b_\n  }{\sqrt{\m\n}} \pr{\frac \n\m}^{it} \sum_{ \substack{ c \leq C    \\ (\m\n P, c) = 1 } } \sum_{ a|c   } \frac{\mu(a)\mu(c)}{a\varphi(c)} \sum_{\ell | P} \mu(\ell ) \sum_{ g\leq \tilde{g}}  \W \pr{\frac{c|\m-\n|}{gQ}}  \\
 &\hskip 1 in \times \frac{1}{\varphi(a\ell g)}\sum_{  \substack{ \Psi \shortmod {a\ell g}  \\    \Psi \neq \Psi_0  }} \Psi(\m) \overline{\Psi}(\n)  \  dt   + O(Q^{-100}).
 \end{split}\end{equation*}
 By Mellin inversion of $\W$ and changing the order of sums and integrals we have for $\delta_1 > 0$
\begin{equation*}\begin{split}
 \Scal_{L,E}  
& =  \frac{1}{ 2\pi i } \int_{(-\delta_1)} \widetilde{\W}(s)Q^s  \int_{-\infty}^{\infty} e^{-t^2}       \sum_{  \substack{c \leq C \\ (c, P) = 1}     } \sum_{ a|c   } \frac{\mu(a)\mu(c)}{ac^s \varphi(c)} \sum_{\ell | P} \mu(\ell )  \sum_{ g\leq \tilde{g}}   \frac{1}{\varphi(a\ell g) g^{-s}}  \\
 & \hskip 1.5in \times \sum_{  \substack{ \Psi \shortmod {a\ell g}  \\   \Psi \neq \Psi_0 } } \sumtwo_{\substack{\m, \n \\ (\m,\n)=1 \\ (\m\n, cP) = 1 }} \frac{a_\m b_\n  }{\sqrt{\m\n}} \pr{\frac \n\m}^{it}  \Psi(\m) \overline{\Psi}(\n) |\m-\n|^{-s} \  dt \> ds + O(Q^{-100}),
 \end{split}\end{equation*}
  where $ \widetilde{\W} $ is the Mellin transform of $\W$. To separate $\m$ and $\n$, we apply the following identity
  \begin{equation}\label{eqn:mns}
   |\m-\n|^{-s}  = \frac{1}{2 \pi i } \int_{(\delta_2)} \frac{ \Gamma (1-s) \Gamma (z)}{ \Gamma (1-s+z)} ( \m^{z-s}\n^{-z} +\n^{z-s} \m^{-z} ) \> dz 
   \end{equation}
where  $ \delta_2 > 0$,  $ \Rep (s) < 0 $ and $ \m \neq \n $.
The integral is absolutely convergent due to the product of gamma factors decaying like $|z|^{-1 + \Rep (s) }$.  To prove \eqref{eqn:mns}, it is enough to show that
$$  (1-x)^{-s} = \frac{1}{2 \pi i } \int_{(\delta_2)} \frac{ \Gamma (1-s) \Gamma (z)}{ \Gamma (1-s+z)} ( x^{-z}  + x^{z-s}  ) \> dz   $$
for $ 0 < x < 1 $, $\delta_2 >0$ and $ \Rep(s)<0$ and it is readily seen by the following two identities. We have 
$$  \frac{1}{2 \pi i } \int_{(\delta_2)} \frac{ \Gamma (1-s) \Gamma (z)}{ \Gamma (1-s+z)}  x^{-z}    \> dz  = \sum_{ k = 0}^\infty   \frac{ (-1)^k}{k!} \frac{ \Gamma(1-s)}{\Gamma(1-s-k)} x^k    = (1-x)^{-s}   $$
by shifting the contour to the left and applying the binomial theorem, and 
$$  \frac{1}{2 \pi i } \int_{(\delta_2)} \frac{ \Gamma (1-s) \Gamma (z)}{ \Gamma (1-s+z)}   x^{z-s}    \> dz  =0  $$
by shifting the contour to the right. By \eqref{eqn:mns},  we write
\begin{equation*}\begin{split}
 \Scal_{L,E}  
 = & \frac{1}{ 2\pi i } \int_{(-\delta_1)} \widetilde{\W}(s)Q^s  \int_{-\infty}^{\infty} e^{-t^2}       \sum_{  \substack{c \leq C \\ (c, P) = 1}     } \sum_{ a|c   } \frac{\mu(a)\mu(c)}{ac^s \varphi(c)} \sum_{\ell | P} \mu(\ell )  \sum_{ g\leq \tilde{g}}   \frac{1}{\varphi(a\ell g) g^{-s}}  \\
 & \times \frac{1}{2 \pi i } \int_{(\delta_2)} \frac{ \Gamma (1-s) \Gamma (z)}{ \Gamma (1-s+z)} \sum_{  \substack{ \Psi \shortmod {a\ell g}  \\   \Psi \neq \Psi_0 } } ( \Scal_{L,E,1}+\Scal_{L,E,2}  )  \> dz \>   dt \>  ds + O(Q^{-100}), 
 \end{split}\end{equation*}
where
\begin{align*}
\Scal_{L,E,1} =&  \sum_{\substack{\m, \n \\ (\m,\n)=1 \\ (\m\n, cP)=1 }} \frac{a_\m b_\n  }{\sqrt{\m\n}} \pr{\frac \n\m}^{it}  \Psi(\m) \overline{\Psi}(\n)   \m^{z-s}\n^{-z}  ; \\
\Scal_{L,E,2} =&   \sum_{\substack{\m, \n \\ (\m,\n)=1 \\ (\m\n, cP)=1 }} \frac{a_\m b_\n  }{\sqrt{\m\n}} \pr{\frac \n\m}^{it}  \Psi(\m) \overline{\Psi}(\n) \n^{z-s} \m^{-z} . 
  \end{align*}  
We choose $  \delta_i = \frac{1}{\log Q}$ for $ i=1,2$. Applying  Lemma \ref{lemma:primecharsum1} to $\mathcal S_{L, E, 1}$ and $\mathcal S_{L, E, 2}$ and using the fact that $ cP \leq Q^4$ (due to support of $\widehat F$), we obtain that
$$ \Scal_{L,E }  \ll Q^{\epsilon} \sum_{c \leq C} \frac{1}{c} \sum_{a|c} \frac{1}{a} \sum_{\ell | P}  \sum_{g \leq \tilde g}  1 \ll C Q^{-1 + (\kappa' + \kappa'')/2 + \epsilon}  $$
for any $ \epsilon>0$, concluding the proof of the lemma.

\end{proof}

\subsection{Evaluating $\mc {M}_{U}  + \Scal_{L,0}  $}\label{proof of prop part 3}
Next, we compute $\Scal_{L,0}$  defined in \eqref{def:SL0}. Indeed, we will show that one of the main terms from $\Scal_{L, 0}$ will cancel out with the main term of $\Scal_U$, which is $\mathcal {M}_U$ defined in (\ref{def:MSU}). Let $\mc{I}(k,r)$ be defined as in Proposition \ref{prop:als}. In this section, we will show the following:  

\es{\label{lem:SUplusSL} \mc {M}_{U}  + \Scal_{L,0} = Q (\log Q)^{k+r}    \sqrt{\pi}      \widetilde{\W}(1)  &\prod_{p \, \nmid P} \bigg( 1 - \frac{1}{p^2}- \frac{1}{ p^3} \bigg)  \prod_{p|P} \bigg( 1 - \frac1p \bigg)  \mc {I}(k,r)\\
&   + O( Q(\log Q)^{k+r-1}\log \log Q  + Q^{1+\epsilon}/C ).}

By applying Mellin inversion, Lemma \ref{lemma:B sum} and \eqref{eqn:mns}, we obtain the following lemma.
\begin{lemma}\label{lemma 5.4}
We have that
\est{\Scal_{L, 0} &= \frac{1}{ (2\pi i)^2 } \int_{(-\delta_1)} \widetilde{\W}(s)Q^s  \zeta(1-s) B(-s)  B_4(-s, P) \int_{-\infty}^{\infty} e^{-t^2}     \sum_{ \substack{ c \leq C    \\ (c, P ) = 1 } } \frac{\mu(c)B_3(-s, c)}{c^s \varphi(c)}  \\
&\hskip .5in \times   \int_{(\delta_2)} \frac{ \Gamma (1-s) \Gamma (z)}{ \Gamma (1-s+z)} \bigg( \mathcal H_0 (s, s-z+it, z-it )  +  \mathcal H_0 ( s, z+it, s-z-it )  \bigg)        \> dt  \> dz \> ds }
and 
\begin{equation*} 
  \mc {M}_U  
   =       -  \widetilde{\W}(0)   B(0)  B_1(0, P)   \int_{-\infty}^{\infty} e^{-t^2}      \sum_{ \substack{ c\leq C \\ (c, P) = 1}    } \frac{ \mu(c) B_2(0 ,c) }{\varphi(c)}   \mc H_0 (0, it, -it )  \> dt,
  \end{equation*}
where each  $\delta_i $ is a small positive number for $ i = 1,2 $,
\es{\label{def:B3} B_3 (s ,  c) := \sum_{a|c} \frac{ \mu(a) B_2(  s , a) }{ a \varphi(a)} = \prod_{p | c} \left( 1 - \frac{1}{p(p-1)} \left(1 + \frac{1}{(p-1)p^{s + 1}} \right)^{-1}\right), }
\es{\label{def:B4}  B_4 (s ,  P ) := \sum_{\substack{\ell | P }} \frac{ \mu(\ell) B_2(  s , \ell) }{  \varphi(\ell)} = \prod_{p | P} \left(1 - \frac{1}{p - 1} \left(1 + \frac{1}{(p-1)p^{s + 1}} \right)^{-1} \right),}
 the functions $B$ and $B_2$ are  defined in Lemma \ref{lemma:B sum}  and 
$$  \mathcal H_0 (s,\alpha,\beta ) := \sumtwo_{\substack{\m, \n \\ (\m,\n)=1 \\
  ( \m\n, cP)=1 }} \frac{a_\m B_1( -s , \m) b_\n B_1(-s , \n)  }{ \m^{1/2+\alpha } \n^{1/2+\beta}} . $$ 
  \end{lemma}

\begin{proof}
 
First we write $\Scal_{L, 0}$ in (\ref{def:SL0}) in terms of the Mellin transform  of $\W$.  For small $\delta_1 > 0,$
\est{
 \Scal_{L,0}  
 = & \frac{1}{ 2\pi i } \int_{(-\delta_1)} \widetilde{\W}(s)Q^s  \int_{-\infty}^{\infty} e^{-t^2}      \sumtwo_{\substack{\m, \n \\ (\m,\n)=1 \\ (\m\n, P) = 1 }} \frac{a_\m b_\n  }{\sqrt{\m\n}} \pr{\frac \n\m}^{it} |\m-\n|^{-s}  \sum_{ \substack{ c \leq C    \\ (c, \m\n P ) = 1 } } \frac{\mu(c)}{c^s \varphi(c)}  \\
 & \hskip 2in \times \sum_{\substack{ a|c }  } \frac{\mu(a)}{a  } \sum_{\substack{ \ell | P \\ (\ell, \m\n) = 1} } \mu(\ell) \sum_{\substack{g \\ (g,\m\n) = 1}}  \frac{1}{\varphi(a\ell g)g^{-s}}    \>  dt  \> ds .  }
  By Lemma \ref{lemma:B sum}, the sum over $g$ is
$$ \sum_{ \substack{g \\ (g, \m\n)=1}} \frac{1}{ \varphi(a\ell g) g^{-s} } = \frac{1}{\varphi(a\ell)} \zeta(1-s) B(-s) B_1 ( -s , \m\n) B_2(-s , a\ell)  ,$$
where the functions $B$, $B_1 $ and $B_2$ are defined in the lemma. Since $a | c$, $(c, P) = 1$, $\ell | P$ and $B_2 (-s, \cdot)$ is multiplicative, it follows that $(a, \ell ) = 1$, $\varphi(a\ell) = \varphi(a)\varphi(\ell)$ and 
$$ B_2 (-s, a \ell) =   B_2 (-s, a  )B_2 (-s,   \ell) .$$ 
Moreover, we can drop the condition $(\ell,\m\n) = 1$ since $\ell | P$ and $(P, \m\n) = 1$.  Then we see that
\est{
 \Scal_{L,0}  
 = & \frac{1}{ 2\pi i } \int_{(-\delta_1)} \widetilde{\W}(s)Q^s  \int_{-\infty}^{\infty} e^{-t^2}      \sumtwo_{\substack{\m, \n \\ (\m,\n)=1 \\ (\m\n, P) = 1 }} \frac{a_\m b_\n  B_1 ( -s , \m\n)  }{\sqrt{\m\n}} \pr{\frac \n\m}^{it} |\m-\n|^{-s}  \sum_{ \substack{ c \leq C    \\ (c, \m\n P ) = 1 } } \frac{\mu(c)}{c^s \varphi(c)}  \\
 &   \times \sum_{\substack{ a|c }  } \frac{\mu(a) B_2(-s , a) }{a \varphi(a)  } \sum_{\substack{ \ell | P } } \frac{\mu(\ell)B_2 (-s,\ell)  }{\varphi(\ell)} \zeta(1-s) B(-s)  \>  dt  \> ds \\ 
= & \frac{1}{ 2\pi i } \int_{(-\delta_1)} \widetilde{\W}(s)Q^s  \int_{-\infty}^{\infty} e^{-t^2}      \sumtwo_{\substack{\m, \n \\ (\m,\n)=1 \\ (\m\n, P) = 1 }} \frac{a_\m b_\n  B_1 ( -s , \m\n)  }{\sqrt{\m\n}} \pr{\frac \n\m}^{it} |\m-\n|^{-s}  \sum_{ \substack{ c \leq C    \\ (c, \m\n P ) = 1 } } \frac{\mu(c)}{c^s \varphi(c)}  \\
 &   \times B_3 (-s, c) B_4 (-s, P)  \zeta(1-s) B(-s)  \>  dt  \> ds   .  } 
By changing the order of summations and applying  \eqref{eqn:mns} and the fact that $(\m, \n) = 1$, we can find the first identity in the lemma.

The second identity in the lemma follows easily by switching the order of summations in (\ref{def:MSU}) and using the coprime conditions.

 \end{proof}

 
 By estimating $\mathcal{H}_0 $ using Lemma \ref{lemma:cs}, we find the next lemma.
 
 \begin{lemma}\label{lemma 5.5}
 We have
 \begin{align*}
     \Scal_{L,0} & = \sum_{ \underline{K} \in \Pi_{k+r} } \mu_{k+r} ( \underline{O}, \underline{K})\ \mc S_{L,0}(\underline{K}) , \\
  \mc {M}_{U} &  = \sum_{ \underline{K} \in \Pi_{k+r} } \mu_{k+r} ( \underline{O}, \underline{K})\ \mc {M}_U(\underline{K}),
 \end{align*}
where for each $ \underline{K} = \{ K_1 , \dots, K_\tau \} \in \Pi_{k+r}$  
\es{\label{eqn:SL0step1}
 \Scal_{L,0}(\underline{K})   
 &:=   \frac{1}{ (2\pi i)^2 } \int_{(-\delta_1)} \widetilde{\W}(s)Q^s  \zeta(1-s) B(-s) B_4(-s, P) \int_{-\infty}^{\infty} e^{-t^2}        \sum_{  \substack{c \leq C   \\ (c, P) = 1}    } \frac{ \mu(c)B_3( -s ,  c ) }{ c^s   \varphi(c) }       \\ 
 &  \times  \int_{(\delta_2)} \frac{ \Gamma (1-s) \Gamma (z)}{ \Gamma (1-s+z)} \bigg( C_{\underline{K}} (cP; s,\alpha_1, \beta_1 )  +  C_{\underline{K}}  (cP;  s, \alpha_2, \beta_2 )  \bigg) \> dz   \> dt   \> ds   
}
with $ \alpha_1 =s-z+it$, $\beta_1 =  z-it$, $\alpha_2 =  z+it$ and $\beta_2 = s-z-it $, and
  \es{ \label{eqn:MSUstep1}
  \mc {M}_{U} (\underline{K})    & :=      -  \widetilde{\W}(0)   B(0)    \int_{-\infty}^{\infty} e^{-t^2}      \sum_{ \substack{ c\leq C  \\ (c, P) = 1}    } \frac{ \mu(c) B_2(0 ,c) }{\varphi(c)}  C_{\underline{K}}  (cP ; 0, it, -it )  \> dt,
  } 
   $$   C_{\underline{K}} (cP ; s,\alpha, \beta) = \sum_{p_1, \dots, p_\tau } \mc J_{cP ; s,\alpha,\beta} ( \iota_{\underline{K}} ( p_1 , \dots, p_\tau )) , $$
 and
  \begin{align*}
\mathcal J_{cP; s,\alpha,\beta} ( p_1 , \dots, p_{k+r}) = &
\prod_{j=1}^k    \frac{ \log p_j B_1 (-s, p_j)}{p_j^{1/2+\alpha}} \widehat{F}_j \bigg( -  \frac{\log p_j}{\log Q} \bigg)  \prod_{j=k+1}^{k+r}    \frac{ \log p_j B_1 (-s, p_j)}{p_j^{1/2+\beta }} \widehat{F}_j \bigg(    \frac{ \log p_j}{\log Q} \bigg)
\end{align*}
 if $ (p_1 \cdots p_{k+r}, cP ) =1 $, and equals to 0 otherwise. 
 \end{lemma}

 \begin{proof} 
    Define
\begin{align*}
 R_{\underline{K}} (cP ; s,\alpha, \beta) &= \sumsh_{p_1, \dots, p_\tau} \mathcal J_{c P ; s,\alpha,\beta} ( \iota_{\underline{K}} ( p_1, \dots, p_\tau )) 
  \end{align*}
for $ \underline{K} = \{ K_1 , \dots, K_\tau \} \in \Pi_{k+r}$, where $\sumsh$ is the sum over distinct primes. Then 
  $$ C_{\underline{H}} (cP ; s,\alpha, \beta) = \sum_{\underline{H}  \preceq \underline{K} } R_{\underline{K}} (cP ; s,\alpha, \beta)$$
  for any $\underline{H} \in \Pi_{k+r}$. By Lemma \ref{lemma:cs}, we have
\begin{equation}\label{eqn:h0sab}
 \mc H_0 (s,\alpha, \beta) = R_{\underline{O}} (cP ; s,\alpha, \beta)     = \sum_{ \underline{K} \in \Pi_{k+r} } \mu_{k+r} ( \underline{O}, \underline{K}) C_{\underline{K}} (cP ; s,\alpha, \beta). 
\end{equation}
By Lemma \ref{lemma 5.4} and \eqref{eqn:h0sab}, the lemma holds. 
\end{proof}


We now compute $C_{\underline{K}}$, which will yield the estimation of $\Scal_{L,0} $ and $ \mathcal{M}_U$, and obtain the following lemma.

\begin{lemma}\label{lemma 5.6}
Let notations be as defined above. Then we have
\es{\label{eqn:MUL0M1M2}  \mc {M }_U + \Scal_{L, 0} =  \sum_{ \substack{ \vecd \in \{ 0, 1, 2\}^{k+r} \\ d_j = 0 \mathrm{~for~some~} j } }      \mc M_{ \vecd }     + O(Q (\log Q)^{k+r-1}\log \log Q ), }
 where
 \es{ \label{def:Md}  \mc {M }_{ \vecd}  \  : = & \int_{-\infty}^{\infty} e^{-t^2}  \frac{1}{(2 \pi i)^2 } \int_{(1/2)}     \int_{(1-1/\mc U)} \widetilde{\W}(s)Q^s  \zeta(1-s) \zeta(2-s) B_5 (s)  B_4(-s, P)     \Gamma_0 (s,z)   \\
&  \times  \sum_{  \substack{c \leq C \\ (c, P) = 1}       } \frac{ \mu(c)B_3( -s ,  c ) }{ c^s   \varphi(c) }            \mc K_{ \vecd } (  s, s - z + it, z - it)       \> ds \> dz  \> dt  , }
\es{\label{def:Gamma0sz} \Gamma_0 (s,z) :=   \frac{ \Gamma (1-s) \Gamma (z)}{ \Gamma (1-s+z)} +   \frac{ \Gamma (1-s) \Gamma (s- z)}{ \Gamma (1-z)} }
       and
$$ \mc {K}_{\vecd} ( s_1, s_2, s_3  )  :=    \prod_{ j \leq k + r} \mc K_{ \{j\} , d_j } (  s_1, s_2, s_3 )   $$ 
for $ \vecd = ( d_1 , \dots, d_{k+r}) \in \{ 0, 1, 2 \}^{k+r} $ and each $ \mc K_{ \{j\} , d_j } $   is defined in \eqref{KK def 1} and \eqref{KK def 2}.   
\end{lemma}

\begin{proof}

 Define
 \begin{equation}\label{def KKj}
 \mathcal K_{K_j} ( cP ; s,\alpha, \beta) := \sum_{\substack{ p  \\ (p,cP)=1}} \left(  \prod_{\substack{\ell \leq k \\ \ell \in K_j }} \frac{ \log p \, B_1 ( -s , p)}{ p^{1/2+\alpha}} \widehat{F}_\ell \bigg( - \frac{ \log p }{\log Q} \bigg) \prod_{\substack{k< \ell \leq k+r \\ \ell \in K_j} } \frac{ \log p \, B_1 ( -s , p)}{ p^{1/2+\beta}} \widehat{F}_\ell \bigg( \frac{ \log p }{ \log Q} \bigg) \right)
 \end{equation}
 for each $j \leq \tau $, then we have 
 $$ C_{\underline{K}} ( cP ; s,\alpha,\beta) = \prod_{j \leq \tau} \mc K_{K_j} ( cP ; s,\alpha, \beta) . $$
 To estimate \eqref{eqn:SL0step1} and \eqref{eqn:MSUstep1}, $(\alpha, \beta)$ represents  $ ( it, -it), (s-z+it, z-it)$ or $(z+it, s-z-it)$.
 
 Let $\delta_1 = \delta_2 = 1/ \log Q$ in \eqref{eqn:SL0step1}, then $\Rep (s ) = - 1/ \log Q$, $ \Rep (z) = 1/ \log Q$ and $ |\alpha|, |\beta|  \leq 2 / \log Q $.  If $ |K_j | >1 $, then 
  $$ \mc K_{K_j } ( cP ; s, \alpha, \beta)  \ll   \sum_{ p \ll Q^4 }   \frac{ (\log p)^2 }{ p^{ 1- 4/ \log Q}} \ll  ( \log Q)^2 .$$
When $|K_j | =1$, there are two cases. We find that  
 \est{ \mc K_{K_j} ( cP ; s,\alpha, \beta) &= \sum_{\substack{ p  \\ (p,cP)=1}}  \frac{ \log p \, }{ p^{1/2+\alpha}}\left(1 - \frac{1}{p^{-s + 1}}\right)\left( 1 + \frac{1}{(p-1)p^{-s + 1}} \right)^{-1} \widehat{F}_\ell \bigg( - \frac{ \log p }{\log Q} \bigg)}
 for $ K_j = \{\ell \} $ with $ \ell \leq k $ and 
  \est{ \mc K_{K_j} ( cP ; s,\alpha, \beta) &= \sum_{\substack{ p  \\ (p,cP)=1}}  \frac{ \log p \, }{ p^{1/2+\beta}}\left(1 - \frac{1}{p^{-s + 1}}\right)\left( 1 + \frac{1}{(p-1)p^{-s + 1}} \right)^{-1} \widehat{F}_\ell \bigg(   \frac{ \log p }{\log Q} \bigg)  }
  for $ K_j = \{\ell \} $ with $ k< \ell \leq k+r $. Since the sum
   \est{   \sum_{\substack{ p  \\ (p,cP)=1}}  \frac{ \log p \, }{ p^{1/2+\alpha}}\left(1 - \frac{1}{p^{-s + 1}}\right)  \widehat{F}_\ell \bigg( \pm \frac{ \log p }{\log Q} \bigg)}
 is the main part of $ \mc K_{K_j} ( cP ; s,\alpha, \beta)$, motivated from Lemma \ref{lemma:prime sum} we define
  \es{ \label{KK def 1}
 \mc K_{K_j ,0 } ( cP; s, \alpha, \beta) & :=  \mc K_{K_j ,0 } (   s, \alpha, \beta) := F_\ell (- i \mathcal{U} (1/2-\alpha)) \log Q  ,\\
 \mc K_{K_j ,1 } ( cP; s, \alpha, \beta) & :=  \mc K_{K_j ,1 } (   s, \alpha, \beta) := - \log Q \int_{-\infty}^0 Q^{ v (1/2-\alpha)} \hat{F}_\ell (-v) dv ,   \\
 \mc K_{K_j ,2 } ( cP; s, \alpha, \beta) & :=  \mc K_{K_j ,2 } (   s, \alpha, \beta) := - \log Q  \int_0^\infty Q^{v (-1/2+s-\alpha)} \hat{F}_\ell (-v) dv  , \\
 \mc K_{K_j ,3 } ( cP; s, \alpha, \beta) & :=  \mc K_{K_j} ( cP ; s,\alpha, \beta) -   \sum_{ 0 \leq i \leq 2} \mc K_{K_j ,i } ( cP; s, \alpha, \beta) }
  for $ K_j = \{\ell \} $ with $ \ell \leq k $ and 
  \es{ \label{KK def 2}
 \mc K_{K_j ,0 } ( cP; s, \alpha, \beta) & :=  \mc K_{K_j ,0 } (   s, \alpha, \beta) := F_\ell ( i \mathcal{U} (1/2-\beta)) \log Q  ,\\
 \mc K_{K_j ,1 } ( cP; s, \alpha, \beta) & := \mc K_{K_j ,1 } (   s, \alpha, \beta) :=  - \log Q \int_{-\infty}^0 Q^{ v (1/2-\beta)} \hat{F}_\ell  (v) dv ,   \\
 \mc K_{K_j ,2 } ( cP; s, \alpha, \beta) & :=  \mc K_{K_j ,2 } (   s, \alpha, \beta) := - \log Q  \int_0^\infty Q^{v (-1/2+s-\beta)} \hat{F}_\ell (v) dv  , \\
 \mc K_{K_j ,3 } ( cP; s, \alpha, \beta) & :=  \mc K_{K_j} ( cP ; s,\alpha, \beta) -   \sum_{ 0 \leq i \leq 2} \mc K_{K_j ,i } ( cP; s, \alpha, \beta) }
   for $ K_j = \{\ell \} $ with $k<  \ell \leq k+r  $, then we see that
  \est{ \mc K_{K_j} ( cP ; s,\alpha, \beta)  = \sum_{ 0 \leq i \leq 3} \mc K_{K_j ,i } ( cP; s, \alpha, \beta)  }
 for $K_j = \{ \ell \} $. Since $\hat{F}_\ell$ is supported in $[-\kappa_\ell, \kappa_\ell]$, we have
 \begin{align*}
  \mc K_{K_j ,1 } ( cP; s, \alpha, \beta) &\ll  \log Q \int_{-\kappa_\ell }^0 Q^{ v /2 } | \hat{F}_\ell (-v) |  dv \ll 1 , \\
 \mc K_{K_j ,2 } ( cP; s, \alpha, \beta) & \ll   \log Q  \int_0^{\kappa_\ell}  Q^{ - v/2 } |\hat{F}_\ell (-v) |  dv  \ll 1  .
  \end{align*}
  To estimate $\mc K_{K_j ,3 }$, by Lemma \ref{lemma:prime sum} we first see that
  \begin{align*}
 \mc K_{K_j } ( cP; s, \alpha, \beta) &  = \sum_{p }  \frac{ \log p   }{ p^{1/2+\alpha}} \widehat{F}_\ell \bigg( - \frac{ \log p }{\log Q} \bigg)  - \sum_{p }  \frac{ \log p   }{ p^{3/2+\alpha -s }} \widehat{F}_\ell \bigg( - \frac{ \log p }{\log Q} \bigg) + O \bigg( \sum_{ p | c } 1 \bigg) \\
 & =  F_\ell (- i \mathcal{U} (1/2-\alpha)) \log Q + R_{-} ( 1/2+\alpha, F_\ell)    + O( \log Q ) \\
 & = \mc K_{K_j  , 0 } ( cP; s, \alpha, \beta) + O( ( \log Q)^2 ) 
\end{align*} 
for $K_j = \{ \ell \} $ with $ \ell \leq k $. Note  that the bound does not depend on $P$ because $\sum_{p | P} 1 \leq n$, where $n$ is defined in Section \ref{sec:intro}. Thus, by Lemma \ref{lemma:prime sum} we find that 
$$  \mc K_{K_j , 3  } ( cP; s, \alpha, \beta) = O( ( \log Q)^2 ). $$

If  $ \underline{K} = \{ K_1 , \dots, K_\tau \}  \in   \Pi':= \Pi'_{k+r}= \{ \underline{K}   \in \Pi_{k+r} :   \ |K_j| = 1 \ \   \textrm{ for some } \ \ j     \} $,
 then we see that
 \begin{align*}
 C_{\underline{K}} (cP; s, \alpha, \beta)  & =    \prod_{j \in I_{\underline{K}, 1 } }\bigg(  \sum_{d_j = 0,1,2,3}  \mc K_{K_j , d_j } (cP; s, \alpha, \beta) \bigg) \prod_{j  \in I_{\underline{K}, 2 } } \mc K_{K_j}(cP; s, \alpha, \beta) \\
  & = \bigg(  \sum_{   d_j = 0,1,2,3 \ \textrm{for ~}  j \in I_{\underline{K}, 1 }   }  \ \prod_{j \in I_{\underline{K}, 1 }} \mc K_{K_j , d_j } (cP; s, \alpha, \beta) \bigg) \prod_{j  \in I_{\underline{K}, 2 } }  \mc K_{K_j}(cP; s, \alpha, \beta)    , 
 \end{align*}
where 
$I_{\underline{K}, 1 } = \{ j \leq \tau : | K_j | = 1 \} $ and $ I_{ \underline{K}, 2 } = \{ j \leq \tau : |K_j | \geq 2 \} $. By the estimations in the previous paragraph, we have
 $$C_{\underline{K}} (cP; s, \alpha, \beta)  = C'_{\underline{K}} (cP; s, \alpha, \beta) + O( (\log Q)^{2(k+r)} )  $$
 for $ \underline{K} \in \Pi'$, where
 \begin{equation}\label{CKprime def}
 C'_{\underline{K}} (cP; s, \alpha, \beta) : = \bigg(  \sum_{  \substack{  d_j = 0,1,2,3 \ \textrm{for ~}  j \in I_{\underline{K}, 1 } \\    d_j = 0  \ \textrm{for~some~}  j \in I_{\underline{K}, 1 }          }   }  \ \prod_{j \in I_{\underline{K}, 1 }} \mc K_{K_j , d_j } (cP; s, \alpha, \beta) \bigg) \prod_{j  \in I_{\underline{K}, 2 } }  \mc K_{K_j}(cP; s, \alpha, \beta)   . 
 \end{equation}
If $\underline{K}  \in \Pi_{k+r} \setminus \Pi'  $, then we see that $ |K_j | \geq 2 $ for every $j$ and 
   $$ C_{\underline{K}} (cP; s, \alpha, \beta) = O( (\log Q)^{2(k+r)} ).$$
  Since the integral over $z$  in \eqref{eqn:SL0step1} is absolutely convergent when $\Rep(s) < 0$, we find that
 \begin{align*}
 \Scal_{L,0} & = \sum_{ \underline{K} \in \Pi' } \mu_{k+r} ( \underline{O}, \underline{K}) \, \Scal'_{L,0} (\underline{K}) + O((\log Q)^{2(k+r)} ) , \\
 \mc {M}_U & =\sum_{ \underline{K} \in \Pi' } \mu_{k+r} ( \underline{O}, \underline{K})   \, \mc {M'}_U ( \underline{K}) + O((\log Q)^{2(k+r)} ) ,
 \end{align*}
 where
\begin{equation*}\begin{split}
 \Scal'_{L,0}(\underline{K})   
 &: =    \frac{1}{ (2\pi i)^2 } \int_{(-\delta_1)} \widetilde{\W}(s)Q^s  \zeta(1-s) B(-s) B_4(-s, P) \int_{-\infty}^{\infty} e^{-t^2}        \sum_{  \substack{c \leq C \\ (c, P) = 1}       } \frac{ \mu(c)B_3( -s ,  c ) }{ c^s   \varphi(c) }\       \\ &  
  \times \int_{(\delta_2)} \frac{ \Gamma (1-s) \Gamma (z)}{ \Gamma (1-s+z)}  \bigg( C'_{\underline{K}} (cP;  s, s-z+it, z-it )  +  C'_{\underline{K}}  (cP ; s, z+it, s-z-it )  \bigg) \>dz  \> dt \>  ds  
  \end{split}
  \end{equation*} 
  and 
  \est{
  \mc {M}'_{U} (\underline{K})    := &      -  \widetilde{\W}(0)   B(0)   B_1(0, P) \int_{-\infty}^{\infty} e^{-t^2}      \sum_{ \substack{ c\leq C  \\ (c, P) = 1}   } \frac{ \mu(c) B_2(0 ,c) }{\varphi(c)}  C'_{\underline{K}}  (cP ; 0, it, -it )  \> dt  }
  for $ \underline{K} \in \Pi' $ with $ \delta_1 = \delta_2 = 1/\log Q$. 

Now we first shift the $z$-integral to $\Rep(z) = 1/2 $. Then for each $\underline{K}  \in \Pi'  $, each summand of $ C'_{\underline{K}} (cP;  s, s-z+it, z-it ) $ in \eqref{CKprime def}  has a  factor 
$$ \mc{K}_{K_j , 0}(cP;  s, s-z+it, z-it )  = F_\ell ( - i \mathcal{U} ( 1/2 - s+z-it)) \log Q   $$
 for some $j$ with $ K_j = \{ \ell \}$ and $ \ell \leq k $ or a factor
 $$  \mc{K}_{K_j , 0 }(cP;  s, s-z+it, z-it )  = F_\ell ( i \mathcal{U} ( 1/2 - z + it )) \log Q $$
 for some $j$ with $ K_j = \{ \ell \}$ and $ k <  \ell \leq k +r $. The other factors in \eqref{CKprime def} are bounded by a power of $Q$ uniformly for $s$ and $z$ in any given vertical strips. Thus, we see that  $ C'_{\underline{K}} (cP;  s, s-z+it, z-it ) $ is rapidly decreasing as $ \Imp (z) \to \infty$ by \eqref{eqn F bound}.  Similarly, $C'_{\underline{K}}  (cP ; s, z+it, s-z-it ) $ is also rapidly decreasing as $ \Imp (z) \to \infty$.  Thus we see that the multiple integrals in $\Scal'_{L,0}(\underline{K}) $ are absolutely convergent and we may change the order of integrals to have
\begin{equation*}\begin{split}
 \Scal'_{L,0}(\underline{K})   
 &  =   \int_{-\infty}^{\infty} e^{-t^2}  \frac{1}{ (2\pi i)^2 }  \int_{(1/2)} \int_{(- 1/\log Q )} \widetilde{\W}(s) Q^s  \zeta(1-s) B(-s) B_4(-s, P)         \sum_{  \substack{c \leq C \\ (c, P) = 1}       } \frac{ \mu(c)B_3( -s ,  c ) }{ c^s   \varphi(c) }\       \\ &  
  \times  \frac{ \Gamma (1-s) \Gamma (z)}{ \Gamma (1-s+z)}  \bigg( C'_{\underline{K}} (cP;  s, s-z+it, z-it )  +  C'_{\underline{K}}  (cP ; s, z+it, s-z-it )  \bigg)\>  ds \>dz  \> dt.    
  \end{split}
  \end{equation*} 
  We next shift the $s$-contour to $ \Rep (s)  = 1-1/\mathcal U$. Since $ \zeta(1-s) $ has a simple pole at $s=0$ with the residue $ -1$, we obtain
\begin{equation*}\begin{split}
 \Scal'_{L,0}(\underline{K})   
 &=  \int_{-\infty}^{\infty} e^{-t^2}  \frac{1}{(2 \pi i)^2 } \int_{(1/2)}   \int_{\big(1- \frac 1{\mathcal U} \big)} \widetilde{\W}(s)Q^s  \zeta(1-s) B(-s)        B_4(-s, P)      \sum_{  
  \substack{c \leq C  \\ (c, P) = 1 }      } \frac{ \mu(c)B_3( -s ,  c ) }{ c^s   \varphi(c) }       \\ 
 &  \hskip 0.2in \times  \frac{ \Gamma (1-s) \Gamma (z)}{ \Gamma (1-s+z)} \bigg( C'_{\underline{K}} (cP;  s, s-z+it, z-it )  +  C'_{\underline{K}}  (cP; s, z+it, s-z-it )  \bigg) \>ds \>dz  \> dt     \\
 & \hskip 0.2in +    \widetilde{\W}(0)     B(0)  B_4(0, P)      \sum_{ \substack{ c \leq C  \\ (c, P) = 1}       } \frac{ \mu(c)B_3( 0 ,  c ) }{     \varphi(c) }       \\ 
  &  
    \hskip 0.4in  \times \int_{-\infty}^{\infty} e^{-t^2}  \frac{1}{2 \pi i } \int_{(1/2)}   \big( C'_{\underline{K}} (cP;  0 ,  -z+it, z-it )  +  C'_{\underline{K}}  (cP ; 0, z+it,  -z-it )  \big)  \> \frac{dz}{z}   \> dt     .
  \end{split}\end{equation*}
  From the residue theorem 
  $$   \frac{1}{2 \pi i } \int_{(1/2)}        C'_{\underline{K}} (cP ; 0 ,  -z+it, z-it )     \frac{  dz}{z} =  \frac{1}{2 \pi i } \int_{(- 1/2)}        C'_{\underline{K}} (cP ; 0 ,  -z+it, z-it )     \frac{  dz}{z} + C'_{\underline{K}} (cP;  0,it,-it), $$
  and by the change of variable 
 $$ \frac{1}{2 \pi i } \int_{(- 1/2)}        C'_{\underline{K}} (cP ; 0 ,  -z+it, z-it )     \frac{  dz}{z}  =- \frac{1}{2 \pi i } \int_{( 1/2)}        C'_{\underline{K}} (cP ; 0 ,   z+it,- z-it )     \frac{  dz}{z}. $$
 Therefore
$$   \frac{1}{2 \pi i } \int_{(1/2)}       \big( C'_{\underline{K}} (cP ; 0 ,  -z+it, z-it )  +  C'_{\underline{K}}  (cP ; 0, z+it,  -z-it )  \big) \>   \frac{  dz}{z}  = C'_{\underline{K}} (cP;  0,it,-it). $$
Since $B_3 ( 0,c) = B_2 (0,c)$ and $B_4(0, P) = B_1(0, P) $, it is not difficult to see that the main term of $\mc {M}'_U  (\underline{K} ) $ cancels out the residue at $s = 0$ of $\Scal'_{L, 0}  (\underline{K} ) $. Therefore, we derive at
\begin{align*}
\mc {M}'_{U}& (\underline{K} )+ \Scal'_{L,0}(\underline{K} ) \\
  =&   \int_{-\infty}^{\infty} e^{-t^2}  \frac{1}{2 \pi i } \int_{(1/2)}   \frac{1}{ 2\pi i } \int_{\big(1-\frac 1 {\mc U}\big)} \widetilde{\W}(s)Q^s  \zeta(1-s) B(-s)  B_4(-s, P)      \sum_{  
  \substack{c \leq C  \\ (c, P) = 1 }      } \frac{ \mu(c)B_3( -s ,  c ) }{ c^s   \varphi(c) }       \\ 
 &   \times  \frac{ \Gamma (1-s) \Gamma (z)}{ \Gamma (1-s+z)} ( C'_{\underline{K}} (cP;  s, s-z+it, z-it )  +  C'_{\underline{K}}  (cP;  s, z+it, s-z-it )  ) \>ds \>dz  \> dt   
\end{align*}
for each $ \underline{K} \in \Pi'$. 

We split the integral into two such that
$$ \mc {M}'_{U}  (\underline{K} )+ \Scal'_{L,0}(\underline{K} )  = \mc{M}_1 (\underline{K} )  + \mc{M}_2 (\underline{K} )  ,$$
where
\begin{align*}
  \mc{M}_1  (\underline{K} )=&   \int_{-\infty}^{\infty} e^{-t^2}   \frac{1}{(2 \pi i)^2  }    \int_{(1/2)}   \int_{\big(1-\frac 1 {\mc U}\big)} \widetilde{\W}(s)Q^s  \zeta(1-s) B(-s)  B_4(-s, P)      \sum_{  
  \substack{c \leq C  \\ (c, P) = 1 }      } \frac{ \mu(c)B_3( -s ,  c ) }{ c^s   \varphi(c) }       \\ 
 &   \times  \frac{ \Gamma (1-s) \Gamma (z)}{ \Gamma (1-s+z)}  C'_{\underline{K}} (cP;  s, s-z+it, z-it )    \>ds \>dz  \> dt   
\end{align*}
and
\begin{align*}
  \mc{M}_2 (\underline{K} ) =&   \int_{-\infty}^{\infty} e^{-t^2}   \frac{1}{(2 \pi i)^2  }    \int_{(1/2)}     \int_{\big(1-\frac 1 {\mc U}\big)} \widetilde{\W}(s)Q^s  \zeta(1-s) B(-s)  B_4(-s, P)      \sum_{  
  \substack{c \leq C  \\ (c, P) = 1 }      } \frac{ \mu(c)B_3( -s ,  c ) }{ c^s   \varphi(c) }       \\ 
 &   \times  \frac{ \Gamma (1-s) \Gamma (z)}{ \Gamma (1-s+z)} C'_{\underline{K}}  (cP;  s, z+it, s-z-it )    \>ds \>dz  \> dt   .
\end{align*}
By changing the order of the $s$-integral and the $z$-integral in $\mc{M}_2$ and substituting $z$ by $s-z$, we see that
\begin{align*}
  \mc{M}_2  (\underline{K} ) =&   \int_{-\infty}^{\infty} e^{-t^2}  \frac{1}{(2 \pi i)^2  }     \int_{\big(1-\frac 1 {\mc U}\big)}  \int_{\big( \frac12 -\frac 1 {\mc U}\big)}    \widetilde{\W}(s)Q^s  \zeta(1-s) B(-s)  B_4(-s, P)      \sum_{  
  \substack{c \leq C  \\ (c, P) = 1 }      } \frac{ \mu(c)B_3( -s ,  c ) }{ c^s   \varphi(c) }       \\ 
 &   \times  \frac{ \Gamma (1-s) \Gamma (s-z)}{ \Gamma (1- z)} C'_{\underline{K}}  (cP;  s, s-z+it,  z-it )    \>dz \>ds  \> dt   .
\end{align*}
By shifting the $z$-contour to $\Rep(z)=1/2$ and changing the order of the $s$-integral and $z$-integral,  we have
\es{\label{M1M2K}
  \mc{M}_1 & (\underline{K} ) +  \mc{M}_2  (\underline{K} )  \\
  =&   \int_{-\infty}^{\infty} e^{-t^2}   \frac{1}{(2 \pi i)^2  }    \int_{(1/2)}   \int_{\big(1-\frac 1 {\mc U}\big)} \widetilde{\W}(s)Q^s  \zeta(1-s) B(-s)  B_4(-s, P)         \\ 
 &   \times  \Gamma_0 (s,z)     \sum_{  
  \substack{c \leq C  \\ (c, P) = 1 }      } \frac{ \mu(c)B_3( -s ,  c ) }{ c^s   \varphi(c) }   C'_{\underline{K}} (cP;  s, s-z+it, z-it )    \>ds \>dz  \> dt   ,
}
where $ \Gamma_0 (s,z) $ is defined in \eqref{def:Gamma0sz}.
 Since $\Rep (s) = 1- 1/\mathcal{U}$ and $\Rep (z) = 1/2 $, we have $ ( \Rep(s-z+it)  , \Rep( z-it) )  =  (1/2 - 1/\mathcal{U} , 1/2) $. By \eqref{def KKj} we find that  
 \begin{equation}\label{Kj bound 1}
\mc{K}_{K_j} ( cP ; s, s-z+it, z-it)  \ll \sum_{p \leq Q^4 } \frac{ ( \log p)^2}{p^2   }   \ll 1 
\end{equation}
for  $\Rep (s) = 1- 1/\mathcal{U}$, $\Rep (z) = 1/2 $ and $|K_j |>1$ and by \eqref{eqn F bound}, \eqref{KK def 1} and \eqref{KK def 2} we also find that
\begin{equation}\label{Kj bound 2}\begin{split}
\mc{K}_{K_j } ( cP ; s, s-z+it, z-it) &  \ll \sum_{p \leq Q^4}  \frac{ \log p }{p} \ll \log Q  , \\
\mc{K}_{K_j , 0} ( cP ; s, s-z+it, z-it) & = F_\ell  ( - i \mathcal{U} ( 1/2- s+z-it )) \log Q \\
&  \ll \frac{ \log Q }{  1+  \mathcal{U}^A | \Imp(s-z)+t |^A   }, \\
\mc{K}_{K_j , 1 } ( cP ; s, s-z+it, z-it) & \ll \log Q \int_{-\kappa_\ell}^0 Q^{v (1/2- \Rep(s-z+it)) } dv \ll \log Q , \\
\mc{K}_{K_j , 2 } ( cP ; s, s-z+it, z-it) &  \ll \log Q \int_0^{\kappa_\ell} Q^{ v(-1/2 + \Rep(s) - \Rep (s-z+it)) } dv \ll \log Q , \\
\mc{K}_{K_j , 3 } ( cP ; s, s-z+it, z-it) & \ll \log Q    \\
\end{split} \end{equation}
for any integer $A\geq 0$, $\Rep (s) = 1- 1/\mathcal{U}$, $\Rep (z) = 1/2 $ and  $ K_j = \{ \ell \}  $ with $ \ell \leq k$ and the similar inequalities hold for $ K_j = \{ \ell \}  $ with $k < \ell \leq k+r$. By Lemma \ref{lemma:prime sum}, we also have
 \begin{equation}\label{Kj bound 3}
   \mc{K}_{K_j , 3 } ( cP ; s, s-z+it, z-it)  = O\bigg( 1 + | \Imp(s-z) + t| + | \Imp( z)- t| + \sum_{ p | cP } \frac{ \log p }{p}  \bigg) 
   \end{equation}
 for  $\Rep (s) = 1- 1/\mathcal{U}$, $\Rep (z) = 1/2 $ and $ |K_j | = 1 $.
 We let 
\es{ \label{def:B5} B_5 (s) = \frac{B(-s)}{  \zeta(2-s)} = \prod_p \bigg( 1+ \frac{1}{ (p-1)p^{1-s}} \bigg) \bigg( 1- \frac{1}{ p^{2-s}} \bigg) . }
The function $B_5(s)$ is absolutely convergent when $\Rep(s) < 3/2$ and $B_5(1) = 1.$ 
 It is well-known that $ \frac{ \Gamma(x) \Gamma(y) }{ \Gamma(x+y)}$ is the beta function and
 $$ \bigg| \frac{ \Gamma(x) \Gamma(y) }{ \Gamma(x+y)}\bigg|  = \bigg| \int_0^1 t^{x-1} (1-t)^{y-1}dt \bigg| \leq \int_0^1 t^{\Rep(x)-1} (1-t)^{\Rep(y)-1}dt \leq \frac{ \Gamma(\Rep(x)) \Gamma(\Rep(y)) }{ \Gamma(\Rep(x+y))} $$
 holds for $ \Rep( x) >0$ and $ \Rep (y) > 0$. Thus, we see that
 \begin{equation}\label{gamma factors log Q bound}
 | \Gamma_0 (s,z)| = \bigg|  \frac{ \Gamma(1-s) \Gamma(z) }{ \Gamma(1-s+z)} +  \frac{ \Gamma(1-s) \Gamma(s-z) }{ \Gamma(1-z)}  \bigg|  \ll \log Q 
 \end{equation}
  holds for  $\Rep (s) = 1- 1/\mathcal{U}$ and $\Rep (z) = 1/2 $. If  $\underline{K} \in \Pi'$ contains a set $K_j$ with $ |K_j | \geq  2 $, then by \eqref{CKprime def}--\eqref{Kj bound 2}, \eqref{def:B5} and \eqref{gamma factors log Q bound}, we find that 
  \begin{align*}
  \mc{M}_1  (\underline{K} ) +  \mc{M}_2  (\underline{K} ) \ll  &  Q(\log Q)^{k+r -1 }  \int_{-\infty}^{\infty} e^{-t^2}         \int_{\big(1-\frac 1 {\mc U}\big)} | \widetilde{\W}(s)  \zeta(1-s)   \zeta(2-s)|    \\
  &  \int_{(1/2)} \bigg(   \frac{1}{1+    \mc{U}^2 ( \Imp(z-s)+t)^2 }      +   \frac{ 1}{1+    \mc{U}^2 ( \Imp(z) - t)^2}  \bigg)  \> |dz |  \> |ds | \> dt   \\
  \ll  &  Q(\log Q)^{k+r -2 }  \int_{-\infty}^{\infty} e^{-t^2}         \int_{\big(1-\frac 1 {\mc U}\big)} | \widetilde{\W}(s)  \zeta(1-s)   \zeta(2-s)|     \> |ds | \> dt \\
    \ll  &  Q(\log Q)^{k+r -2 }   \bigg( 1 + \int_{1-1/\mc{U} -i}^{1-1/\mc{U}+i}   \frac{ |ds|}{|1-s|}  \bigg)   \\
 \ll &  Q(\log Q)^{k + r - 2} \log \log Q.
  \end{align*} 
Therefore, it is enough to consider the case $\underline{K} = \underline{O} = \{ \{ 1 \} , \dots, \{ k+r \} \}   $ and we have
\est{   \mc {M }_U + \Scal_{L, 0}  =    \sum_{ \substack{ \vecd \in \{ 0, 1, 2, 3\}^{k+r} \\ d_j = 0 \mathrm{~for~some~} j } } \mc {M }_{  \vecd}  + O(  Q(\log Q)^{k + r - 2} \log \log Q),   }
where $\mc{M}_{\vecd}  $ is defined in \eqref{def:Md}.

If $\vecd \in \{ 0, 1, 2, 3\}^{k+r}$ satisfies $d_{j_1}=0$ and $d_{j_2}= 3 $ for some $j_1 $ and $j_2 $, then by  \eqref{Kj bound 2} -- \eqref{gamma factors log Q bound}, we find that 
  \begin{align*}
  \mc{M}_1  & (\underline{K} ) +  \mc{M}_2  (\underline{K} ) \\
  \ll  &  Q(\log Q)^{k+r   }  \int_{-\infty}^{\infty} e^{-t^2}         \int_{\big(1-\frac 1 {\mc U}\big)} | \widetilde{\W}(s)  \zeta(1-s)   \zeta(2-s)|    \\
  &  \int_{(1/2)} \bigg(   \frac{1 + | \Imp(z-s)-t| + | \Imp(z)-t | )}{1+    \mc{U}^4 ( \Imp(z-s)-t)^4 }      +      \frac{ 1 + | \Imp(z-s)-t| + | \Imp(z)-t | )}{1+    \mc{U}^4 ( \Imp(z) - t)^4}  \bigg)  \> |dz |  \> |ds | \> dt   \\
  \ll  &  Q(\log Q)^{k+r -1  }  \int_{-\infty}^{\infty} e^{-t^2}         \int_{\big(1-\frac 1 {\mc U}\big)} | \widetilde{\W}(s)  \zeta(1-s)   \zeta(2-s)|  (1+ | \Imp ( s) | )     \> |ds | \> dt \\
    \ll  &  Q(\log Q)^{k+r -1 }  \log \log Q.
  \end{align*} 
  Therefore, we prove the lemma.
\end{proof}


%
%
\begin{lemma}\label{lemma 5.7}
Let notations be as defined above. Then we have
 \est{ \mc {M }_U + \Scal_{L, 0}    \ = &  \sum_{ \substack{ 1 \leq j_1 \leq k \\  k+1 \leq j_2 \leq k+r } }    \int_{-\infty}^{\infty} e^{-t^2}  \frac{1}{(2 \pi i)^2 } \int_{(1/2)}     \int_{(1-1/\mc U)} \widetilde{\W}(s)Q^s  \zeta(1-s) \zeta(2-s) B_5 (s)  B_4(-s, P)      \\
&   \times  \Gamma_0 (s,z)    B_6 (s,P)       \mc{P}_{j_1, j_2}   ( s,z,t)       \> ds \> dz  \> dt  +O( Q(\log Q)^{k+r-1}\log \log Q  + Q^{1+\epsilon}/C )    }
for any $\epsilon>0$, where
$$  \mc{P}_{j_1, j_2}( s,z,t ): = \sum_{ \substack{ \vecd  \\ d_{j_1} =  d_{j_2}=0  \\  d_\ell \neq 0 ~ \mathrm{for}~ \ell <j_1, k < \ell < j_2   }} \mc{K}_{ \vecd}(  s, s-z+it, z-it)     $$
for $ j_1  \leq k $ and $ k < j_2 \leq k+r$. Moreover, we have
\es{ \label{Pj1j2szt}
 \mc{P}_{j_1, j_2}& (  s,z,t ) \\
= & \ (\log Q)^{k + r } (-1)^{j_1 + r }  \int_{-\infty}^{\infty} Q^{u_{j_1} (1/2 -s+z-it )} \widehat F_{j_1} (-u_{j_1} ) \> du_{j_1}   \int_{-\infty}^{\infty} Q^{-u_{j_2}(1/2-z+it)} \widehat F_{j_2}( -u_{j_2}) \> du_{j_2}    \\
& \times \prod_{\substack{ 1 \leq j < j_1  \\ \mathrm{or} \\ k + 1 \leq j < j_2 } } \left( \int_{-\infty}^0 Q^{u_j(1/2 - s+z-it )} \widehat F_j (-u_j) \> du_j + \int_0^\infty Q^{ u_j(-1/2 +z-it ) } \widehat F_j ( -u_j) \> du_j     \right) \\ 
& \times  \prod_{ j_1 < j \leq k} \left(  \int_0^{\infty} Q^{u_j(1/2 -s+z-it )} \widehat F_j (-u_j) \> du_j - \int_0^\infty Q^{u_j(-1/2 +z-it ) } \widehat F_j (-u_j) \> du_j   \right)  \\
& \times  \prod_{j_2  < j \leq k + r} \left(  \int_{-\infty}^0 Q^{u_j(1/2 -s+ z-it)} \widehat F_j (- u_j) \> du_j  -  \int_{-\infty}^0  Q^{u_j(-1/2 + z-it ) } \widehat F_j (- u_j) \> du_j    \right) .  }
\end{lemma}

\begin{proof}
 
If $\vecd \in \{ 0,1,2\}^{k+r} $ satisfies the property that $d_j = 0 $ for some $ j > k$ and $ d_j \neq 0 $ for all $ j \leq k $, then we estimate $\mc{M}_{\vecd}$ in Lemma \ref{lemma 5.6} by shifting the $s$-contour to $ \Rep(s)=1-\epsilon$ for $ \epsilon>0$. The $z$-contour remains on the line $ \Rep(z)=1/2$. Then by \eqref{eqn F bound}, \eqref{KK def 1} and \eqref{KK def 2}, we see that 
$$ \mc{K}_{\vecd} ( s, s-z+it, z-it) = O\bigg(   \frac{  (\log Q)^{k+r}  }{ 1+ \mc{U}^2 ( \Imp(z)-t )^2 } \bigg)$$
and
$$ \mc{M}_{  \vecd} = O( Q^{1-\epsilon} ( \log Q)^{k+r-1} ) . $$
If $\vecd \in \{ 0,1,2\}^{k+r} $ satisfies the property that $d_j = 0 $ for some $ j \leq k $ and $ d_j \neq 0 $ for all $ k < j \leq k+r$, then we estimate $\mc{M}_{\vecd}$ by shifting the $s$-contour to $ \Rep(s)=1- \epsilon -1/\log Q  $ and the $z$-contour to $ \Rep(z)=1/2- \epsilon$ for $ \epsilon>0$. Then by \eqref{eqn F bound}, \eqref{KK def 1} and \eqref{KK def 2}, we see that 
$$ \mc{K}_{\vecd} (  s, s-z+it, z-it) = O\bigg(   \frac{  (\log Q)^{k+r}  }{ 1+ \mc{U}^2 ( \Imp(s-z)+t )^2 } \bigg)$$
and
$$ \mc{M}_{  \vecd} = O( Q^{1-\epsilon} ( \log Q)^{k+r -1} ) . $$
The other $\vecd$'s satisfy that $d_{j_1} = d_{j_2} = 0 $ for some $j_1$ and $j_2$ such that $ j_1 \leq k < j_2 $. Hence,  
$$   \mc {M }_U + \Scal_{L, 0}   = \sum_{ \substack{ \vecd \in \{0,1,2\}^{k+r}  \\ d_{j_1} = d_{j_2} =0 \\  \mathrm{for~some}~ j_1 \leq k < j_2   }} \mc{M}_{\vecd} + O( Q^{1-\epsilon} ( \log Q)^{k+r} ) . $$
 
Now we consider the sum
\est{ \sum_{ \substack{ \vecd  \in \{0,1,2\}^{k+r} \\ d_{j_1} = d_{j_2} =0 \\  \mathrm{for~some}~ j_1 \leq k < j_2   }} \mc{K}_{\vecd} (  s, s-z+it, z-it) 
&= \sum_{ \substack{ 1 \leq j_1 \leq k \\  k+1 \leq j_2 \leq k+r } }  \sum_{ \substack{ \vecd  \\ d_{j_1} =  d_{j_2}=0  \\  d_\ell \neq 0 ~ \mathrm{for}~ \ell <j_1, k < \ell < j_2   }} \mc{K}_{ \vecd}(  s, s-z+it, z-it)  . \\
& =  \sum_{ \substack{ 1 \leq j_1 \leq k \\  k+1 \leq j_2 \leq k+r } }  \mc{P}_{j_1, j_2}( s,z,t ).}
We can write $\mc{P}_{j_1, j_2}(  s,z,t ) $  as the following product
\est{       
  \mc K_{\{j_1 \} , 0} \mc K_{\{j_2 \}, 0 } \prod_{\substack{ 1 \leq j < j_1 \\ \mathrm{or} \\ k < j < j_2  }  } \big( \mc K_{\{j \}, 1} + \mc K_{\{j \} , 2}   \big) \prod_{\substack{ j_1 < j \leq k  \\ \mathrm{or}\\   j_2 < j \leq k+r  }} \big(\mc K_{\{j \} , 0} + \mc K_{\{j \}, 1} + \mc K_{\{j \}, 2}  \big) ,}
  where 
  $$\mc K_{\{ j \}, \ell} := \mc K_{\{ j \}, \ell} ( s,  s-z+it, z-it).$$
 We see that
$$ \mc{K}_{ \{j\} , 0 } + \mc{K}_{ \{j\} , 1 } =  \log Q \int_0^\infty Q^{v(1/2-s+z-it)} \hat{F}_j (-v) dv$$
for $j \leq k$ and 
  $$ \mc{K}_{ \{j\} , 0 } + \mc{K}_{ \{j\} , 1 } =  \log Q \int_0^\infty Q^{v(1/2-z+it)} \hat{F}_j (v) dv  =  \log Q \int_{-\infty}^0 Q^{-v(1/2-z+it)} \hat{F}_j (-v) dv $$
for $j > k$. Hence we have \eqref{Pj1j2szt} and  
\est{  \mc {M }_U + \Scal_{L, 0}    \ = &  \sum_{ \substack{ 1 \leq j_1 \leq k \\  k+1 \leq j_2 \leq k+r } }    \int_{-\infty}^{\infty} e^{-t^2}  \frac{1}{(2 \pi i)^2 } \int_{(1/2)}     \int_{(1-1/\mc U)} \widetilde{\W}(s)Q^s  \zeta(1-s) \zeta(2-s) B_5 (s)  B_4(-s, P)      \\
&  \times \Gamma_0 (s,z)   \sum_{  \substack{c \leq C \\ (c, P) = 1}       } \frac{ \mu(c)B_3( -s ,  c ) }{ c^s   \varphi(c) }         \mc{P}_{j_1, j_2}   ( s,z,t)       \> ds \> dz  \> dt  +O( Q(\log Q)^{k+r-1}\log \log Q)  . }
The sum over $ c $ is asymptotic to  
$$ B_6 (s,P) := \sum_{  \substack{c=1 \\ (c, P) = 1}       }^\infty  \frac{ \mu(c)B_3( -s ,  c ) }{ c^s   \varphi(c) }  $$
with an error $ O( C^{-1 })$ for $ \Rep(s )= 1-1/\mc{U} $. Hence, 
 \est{ \mc {M }_U + \Scal_{L, 0}    \ = &  \sum_{ \substack{ 1 \leq j_1 \leq k \\  k+1 \leq j_2 \leq k+r } }    \int_{-\infty}^{\infty} e^{-t^2}  \frac{1}{(2 \pi i)^2 } \int_{(1/2)}     \int_{(1-1/\mc U)} \widetilde{\W}(s)Q^s  \zeta(1-s) \zeta(2-s) B_5 (s)  B_4(-s, P)      \\
&   \times  \Gamma_0 (s,z)    B_6 (s,P)       \mc{P}_{j_1, j_2}   ( s,z,t)       \> ds \> dz  \> dt  +O( Q(\log Q)^{k+r-1}\log \log Q  + Q^{1+\epsilon}/C )    }
for any $\epsilon>0$.

\end{proof}

By expanding the products in \eqref{Pj1j2szt} and changing the order of integrals, we have
 \est{
 \mc{P}_{j_1, j_2}  (  s,z,t )  
=& \  (\log Q)^{k + r } (-1)^{j_1  + r}   \summany_{\substack{T_1,W_1, T_2, W_2, T_3, W_3 \\ T_1 + W_1 = \{1,..,j_1 -1\} \cup \{k + 1,...,  j_2 -1\} \\ T_2 + W_2 = \{ j_1 + 1,.., k\} \\ T_3 + W_3 = \{j_2  + 1, ..., k+r\} }}  (-1)^{|W_2 |+| W_3|}  \\
& \times \int_{\mc{D}_{k+r  }(\vec{T}, \vec{W})}\left( \prod_{j = 1}^{k + r} \widehat F_j(-u_j) \right)   Q^{(1-s)  (u_{ j_1 } + u(\vec{T} ))    + (-1/2 + z - it) u([k+r])}    \> d\vecu ,
}   
where $ \mc{ D}_{k+r }(\vec{T}, \vec{W})$, $u(\vec{T})$ and $ u([k+r])$ are defined in Proposition \ref{prop:als}. Hence,
 \est{ \mc {M }_U + \Scal_{L, 0}   =&   (\log Q)^{k + r }  \sum_{ \substack{ 1 \leq j_1 \leq k \\  k+1 \leq j_2 \leq k+r } }       \summany_{\substack{T_1,W_1, T_2, W_2, T_3, W_3 \\ T_1 + W_1 = \{1,..,j_1 -1\} \cup \{k + 1,...,  j_2 -1\} \\ T_2 + W_2 = \{ j_1 + 1,.., k\} \\ T_3 + W_3 = \{j_2  + 1, ..., k+r\} }}  (-1)^{j_1  + r+ |W_2 |+| W_3|}     \mc{M}  ( j_1, j_2, \vec{T}, \vec{W} ) \\
& + O( Q(\log Q)^{k+r-1}\log \log Q  + Q^{1+\epsilon}/C )  ,  }
where 
\es{ \label{Mj1j2TW}   \mc{M}    &  ( j_1, j_2, \vec{T}, \vec{W} ) \\
  := &      \int_{-\infty}^{\infty} e^{-t^2}  \frac{1}{(2 \pi i)^2 } \int_{(1/2-\epsilon_1)}     \int_{(1-\epsilon_2 )} \widetilde{\W}(s)   \zeta(1-s) \zeta(2-s) B_5 (s)  B_4(-s, P)  B_6 (s,P)    \Gamma_0 (s,z)    \\
& \ \ \ \    \times   \int_{ \mc{D}_{k+r }(\vec{T}, \vec{W})   }\left( \prod_{j = 1}^{k + r} \widehat F_j(-u_j) \right)   Q^{1+ (1-s) (  u_{ j_1 } + u(\vec{T} ) -1)  + (-1/2 + z - it) u([k+r])}    \> d\vecu       \> ds \> dz  \> dt     }
for $0< \epsilon_1 <  \epsilon_2 < 1/100$. We next estimate $ \mc{M}      ( j_1, j_2, \vec{T}, \vec{W} ) $ to derive the following lemma.
\begin{lemma} \label{lemma 5.8} Let notations be as defined above. Then we have
\est{  \mc{M}   ( j_1, j_2, \vec{T}, \vec{W}  )  
= &      \frac{	Q \sqrt{\pi}}{2}   \widetilde{\W}(1)  \prod_{p \, \nmid P} \bigg(1 -  \frac{1}{p^2} - \frac{1}{ p^3} \bigg)   \prod_{p |P } \bigg( 1 - \frac1p \bigg)  \\ &  \times \int_{ \substack{ \mc D_{k + r }(\vec{T}, \vec{W})  \\ u_{ j_1 } + u(\vec{T} )   > 1 }  } \left( \prod_{j=1}^{k+r}  \widehat F_j(-u_j) \right)       (1- u_{ j_1 } - u(\vec{T} )  ) \delta( u ([k+r]) )   \> d\vecu         \\
& \ \ \ \ + O( Q/\log Q)  .  }
\end{lemma}
\begin{proof}
 To make the multiple integral in \eqref{Mj1j2TW} absolutely convergent for $ \Rep (s) $ near $1$, we integrate the $u_{j_2}$-integral by parts twice. 
\es{\label{MjTW}  \mc{M}  & ( j_1, j_2, \vec{T}, \vec{W}   )  \\
 = &    \frac{1}{ (\log Q)^2 }  \int_{-\infty}^{\infty}  \frac{1}{(2 \pi i)^2 } \int_{(1/2-\epsilon_1)}     \int_{(1-\epsilon_2 )} \widetilde{\W}(s)   \zeta(1-s) \zeta(2-s) B_5 (s)  B_4(-s, P)  B_6 (s,P)     \Gamma_0 (s,z)     \\
&    \times   \int_{\mc{D}_{k+r }(\vec{T}, \vec{W})  }\left( \prod_{j\neq j_2 } \widehat F_j(-u_j) \right) \widehat F_{j_2}'' (-u_{j_2} )   Q^{1+ (1-s) (u_{ j_1 } + u(\vec{T} ) -1)  + (-1/2 + z - it) u([k+r])}   \frac{ e^{-t^2}  \> d\vecu       \> ds \> dz  \>  dt  }{   (1/2 - z + it)^2}    }
  Based on the exponent of $Q$ in the integrand, we split the domain $ \mc{D}_{k+r }(\vec{T}, \vec{W})$ into the following three subsets : 
\begin{align*}
  \mc D_{1} &=  \{ \vecu \in  \mc D_{k+r}(\vec{T}, \vec{W})  : u_{ j_1 } + u(\vec{T} ) -1 > 0 , \ u([k+r]) < 0 \} ,  \\ 
   \mc D_{2} &= \{ \vecu \in \mc  D_{k+r }(\vec{T}, \vec{W})  : u_{ j_1 } + u(\vec{T} ) -1 > 0 , \ u([k+r])  \geq 0 \} ,  \\  
   \mc D_{3} &= \{ \vecu \in \mc  D_{k+r }(\vec{T}, \vec{W}) : u_{ j_1 } + u(\vec{T} ) -1 \leq 0  \}  .
 \end{align*}
Clearly,
$$  \mc{M}    ( j_1, j_2, \vec{T}, \vec{W} )  = \sum_{i=1}^3  \mc{M}    ( j_1, j_2, \vec{T}, \vec{W} ; \mc{D}_i ) , $$
where each $\mc{M}    ( j_1, j_2, \vec{T}, \vec{W} ; \mc{D}_i )$ is defined analogously to $ \mc{M}    ( j_1, j_2, \vec{T}, \vec{W} )$ with $\mc{D}_i $ in place of $\mc{D}_{k + r }(\vec{T}, \vec{W})$. We now compute each $\mc{M}    ( j_1, j_2, \vec{T}, \vec{W} ; \mc{D}_i )$ as follows, expecting that the main contribution comes from the region $\mc D_1 $ (Case 1). In each case, we will shift the $s$ and $z$ contours in a way that the real part of the exponent of $Q$ in \eqref{MjTW} is $\leq 1 $.
\\
\\
{\bf Case 1: $\mc{M}    ( j_1, j_2, \vec{T}, \vec{W} ; \mc{D}_1 ) $.}
 The integrand has a double pole at $s=1$.  
  By shifting the $s$-integral to $1+\epsilon$, we pick up the residue at $s=1$.
\est{   \mc{M}  & ( j_1, j_2, \vec{T}, \vec{W}  ;\mc{D}_1 )  \\
 = &    \frac{1}{ (\log Q)^2 }  \int_{-\infty}^{\infty}  \frac{1}{(2 \pi i)^2 } \int_{(1/2-\epsilon_1)}     \int_{(1+ \epsilon  )} \widetilde{\W}(s)   \zeta(1-s) \zeta(2-s)  B_5 (s)  B_4(-s, P)  B_6 (s,P)     \Gamma_0 (s,z)    \\
&   \times     \int_{\mc{D}_1 }\left( \prod_{j\neq j_2 } \widehat F_j(-u_j) \right) \widehat F_{j_2}'' (-u_{j_2} )   Q^{1+ (1-s) (u_{ j_1 } + u(\vec{T} ) -1)  + (-1/2 + z - it) u([k+r])}   \frac{ e^{-t^2} \> d\vecu       \> ds \> dz  \>   dt }{   (1/2 - z + it)^2}    \\
& + \mathrm{Res}_{s=1}.  }
Since each $\widehat F_j(-u_j) $ and its derivatives are compactly supported, by shifting the $z$-integral to $1/2 - 1/\log Q $ the above integral is 
\es{ \label{boundhalfminus}  \ll & \frac{Q }{ (\log Q)^2 }  \int_{-\infty}^{\infty}    \int_{(1/2 -1/{\log Q})}     \int_{(1+ \epsilon  )} | \widetilde{\W}(s)   \zeta(1-s) \zeta(2-s) |     \\
&   \times     \int_{\mc{D}_1 }\left( \prod_{j\neq j_2 }|  \widehat F_j(-u_j) | \right) |  \widehat F_{j_2}'' (-u_{j_2} )  |     \frac{ e^{-t^2} \>  d\vecu       \>  |ds|  \> | dz | \>   dt }{  |1/2 - z + it |^2}    \\
	\ll & \frac{Q }{ (\log Q)^2 }  \int_{-\infty}^{\infty}    \int_{-\infty}^{\infty}         \frac{ 1}{  \frac{1}{(\log Q)^2 }+ (w-t)^2 }     \>  dw  \>   e^{-t^2}dt   \\
\ll & \frac{Q }{  \log Q  } . }
Hence, we see that
$$\mc{M}   ( j_1, j_2, \vec{T}, \vec{W}  ;\mc{D}_1 )  =  \mathrm{Res}_{s=1} + O \bigg(  \frac{Q }{  \log Q  } \bigg). $$
Next we compute the residue at $s = 1$. Since 
 $$ \zeta(2-s) \Gamma(1-s) = \frac{1}{ (s-1)^2 } + O(1)   $$
 as $ s \to 1$, we have 
\es{ \label{ress1}  &  \mathrm{Res}_{s=1}  \\
 = & -  \frac{1}{ (\log Q)^2 }  \int_{-\infty}^{\infty}  \frac{1}{ 2 \pi i  } \int_{(1/2-\epsilon_1)}  \frac{ \partial}{\partial s}  \bigg[ \widetilde{\W}(s)   \zeta(1-s)   B_5 (s)  B_4(-s, P)  B_6 (s,P)    \bigg(  \frac{\Gamma(z)}{\Gamma(1-s+z)} +  \frac{ \Gamma(s-z)}{\Gamma(1-z)}  \bigg) \bigg]_{s=1}   \\
&     \times   \int_{\mc{D}_1 }\left( \prod_{j\neq j_2 } \widehat F_j(-u_j) \right) \widehat F_{j_2}'' (-u_{j_2} )   Q^{1 + (-1/2 + z - it) u([k+r])}   \frac{ e^{-t^2}}{   (1/2 - z + it)^2}  \> d\vecu        \> dz  \> dt  \\   
&+    \frac{2}{ (\log Q)^2 }  \int_{-\infty}^{\infty}   \frac{1}{ 2 \pi i  } \int_{(1/2-\epsilon_1)}   \widetilde{\W}(1)   \zeta(0)   B_5 (1)  B_4(-1 , P)  B_6 (1,P)       \\
&    \times   \int_{\mc{D}_1 }\left( \prod_{j\neq j_2 } \widehat F_j(-u_j) \right) \widehat F_{j_2}'' (-u_{j_2} )   Q^{1  + (-1/2 + z - it) u([k+r])} (u_{ j_1 } + u(\vec{T} ) -1) \log Q  \frac{  e^{-t^2} \> d\vecu        \> dz  \> dt   }{   (1/2 - z + it)^2}  .  }
The first integral can be estimated by shifting the $z$-integral to $ 1/2-1/\log Q$ as in \eqref{boundhalfminus} and it is bounded by $\ll Q/\log Q$.  For the second integral, we shift the $z$-integral to $1/2+\epsilon$ and pick  up the residue at $z=1/2+it$. The shifted integral may be estimated similarly to \eqref{boundhalfminus}. Therefore we find that
\est{ & \mc{M}  ( j_1, j_2, \vec{T}, \vec{W}  ;\mc{D}_1 )  
   \\
 &= -  2 Q   \sqrt{\pi}  \widetilde{\W}(1)   \zeta(0)     B_4(-1 , P)  B_6 (1,P)         \int_{\mc{D}_1 }\left( \prod_{j\neq j_2 } \widehat F_j(-u_j) \right) \widehat F_{j_2}'' (-u_{j_2} )     u([k+r])  (u_{ j_1 } + u(\vec{T} ) -1)   \> d\vecu         \\
 &+    \frac{2}{  \log Q  }  \int_{-\infty}^{\infty}   \frac{1}{ 2 \pi i  } \int_{(1/2 + \epsilon )}   \widetilde{\W}(1)   \zeta(0)   B_5 (1)  B_4(-1 , P)  B_6 (1,P)       \\
&    \times   \int_{\mc{D}_1 }\left( \prod_{j\neq j_2 } \widehat F_j(-u_j) \right) \widehat F_{j_2}'' (-u_{j_2} )   Q^{1  + (-1/2 + z - it) u([k+r])} (u_{ j_1 } + u(\vec{T} ) -1)   \frac{  e^{-t^2} \> d\vecu        \> dz  \> dt   }{   (1/2 - z + it)^2}  + O\bigg(  \frac{Q}{\log Q} \bigg) \\  
 &= -   2Q   \sqrt{\pi}  \widetilde{\W}(1)   \zeta(0)     B_4(-1 , P)  B_6 (1,P)         \int_{\mc{D}_1 }\left( \prod_{j\neq j_2 } \widehat F_j(-u_j) \right) \widehat F_{j_2}'' (-u_{j_2} )     u([k+r])  (u_{ j_1 } + u(\vec{T} ) -1)   \> d\vecu         \\
 &+  O \bigg(     \frac{Q}{  \log Q  }  \int_{-\infty}^{\infty} \int_{-\infty}^{\infty} e^{-t^2}     \frac{  1  }{   \epsilon^2 + (w-t)^2 } dw dt \bigg) +O\bigg(  \frac{Q}{\log Q} \bigg)  \\
   &= -   2Q   \sqrt{\pi}  \widetilde{\W}(1)   \zeta(0)     B_4(-1 , P)  B_6 (1,P)         \int_{\mc{D}_1 }\left( \prod_{j\neq j_2 } \widehat F_j(-u_j) \right) \widehat F_{j_2}'' (-u_{j_2} )     u([k+r])  (u_{ j_1 } + u(\vec{T} ) -1)   \> d\vecu         \\
 &+    O\bigg(  \frac{Q}{\log Q} \bigg) .  }
 We observe that for $ \tilde{u}_{j_2 } := u([k+r])-u_{j_2} $, the $u_{j_2} $-integral is
 \est{ \int_{ -\infty}^{ -\tilde{u}_{j_2} } \widehat{F}''_{j_2} (-u_{j_2}) u([k+r]) du_{j_2} &=  \int_{ -\infty}^{ -\tilde{u}_{j_2} } \widehat{F}'_{j_2} (-u_{j_2})   du_{j_2} = - \widehat{F}_{j_2} (\tilde{u}_{j_2} )\\
 &  = - \int_\R \widehat{F}_{j_2} (-u_{j_2} )\, \delta( u ([k+r]) )\> du_{j_2},   }
 where $\delta$ is the Dirac delta function. Hence,
\est{ & \mc{M} ( j_1, j_2, \vec{T}, \vec{W}  ;\mc{D}_1 )  \\
& =   2Q   \sqrt{\pi}  \widetilde{\W}(1)   \zeta(0)     B_4(-1 , P)  B_6 (1,P)             \int_{ \substack{ \mc D_{k+r}(\vec{T}, \vec{W})  \\ u_{ j_1 } + u(\vec{T} )   > 1 }  } \left( \prod_{j=1}^{k+r}  \widehat F_j(-u_j) \right)       (u_{ j_1 } + u(\vec{T} ) -1) \delta( u ([k+r]) )   \> d\vecu         \\
& \hskip 4in + O( Q/\log Q)  .  }
\\
\\
{\bf Case 2: $\mc{M}   ( j_1, j_2, \vec{T}, \vec{W} ; \mc{D}_2 ) $.}
    We shift the $s$-contour to the line $\Rep(s) = 1 + \epsilon$ as in the first case and pick up the residue at $s=1$. Then we obtain that
  $$\mc{M}   ( j_1, j_2, \vec{T}, \vec{W}  ;\mc{D}_2 )  =  \mathrm{Res}_{s=1} + O \bigg(  \frac{Q }{  \log Q  } \bigg),$$
  where $   \mathrm{Res}_{s=1}$ is in \eqref{ress1}.  To estimate $\mathrm{Res}_{s=1}$ we bound the integrals in \eqref{ress1} trivially and obtain that
$$ \mc{M}   ( j_1, j_2, \vec{T}, \vec{W} ; \mc{D}_2 ) =O \bigg(  \frac{Q }{  \log Q  } \bigg) .$$
\\
\\
{\bf Case 3: $\mc{M}   ( j_1, j_2, \vec{T}, \vec{W} ; \mc{D}_3 ) $.}
 For this case, we shift the $z$-integral to $ 1/2-1/\log Q$ and bound the integral trivially.  Then we see that
 \est{   \mc{M}  & ( j_1, j_2, \vec{T}, \vec{W}  ; \mc{D}_3  )   \\
 \ll  &    \frac{Q}{ (\log Q)^2 }  \int_{-\infty}^{\infty}  \int_{(1/2- 1/\log Q )}     \int_{(1-\epsilon_2 )} | \widetilde{\W}(s)   \zeta(1-s)   |     \\
&    \times   \int_{\mc{D}_3   }\left( \prod_{j\neq j_2 } | \widehat F_j(-u_j) | \right) | \widehat F_{j_2}'' (-u_{j_2} )  |     \frac{ e^{-t^2}  \> d\vecu       \> | ds|  \> |dz|   \>  dt  }{   |1/2 - z + it|^2} \\
\ll  &    \frac{Q}{ (\log Q)^2 }  \int_{-\infty}^{\infty}  \int_{-\infty}^{\infty}       \frac{ e^{-t^2}   \> dw  \>  dt  }{   1/(\log Q)^2 + ( w-t)^2 } \\
 \ll  &   \frac{Q}{\log Q }. }

Combining Cases 1-3 and the facts that $\zeta(0) = -1/2$, $ B_4 (-1,P) = \prod_{p \, \nmid P} (1 -  p^{-2}-   p^{-3} )    $ and $ B_6 (1,P)  = \prod_{p |P } ( 1 - p^{-1} )  $, we derive that
\est{  \mc{M}   ( j_1, j_2, \vec{T}, \vec{W}  )  
= &    	Q \sqrt{\pi}    \widetilde{\W}(1)  \prod_{p \, \nmid P} \bigg(1 -  \frac{1}{p^2} - \frac{1}{ p^3} \bigg)   \prod_{p |P } \bigg( 1 - \frac1p \bigg)  \\ &  \times \int_{ \substack{ \mc D_{k + r }(\vec{T}, \vec{W})  \\ u_{ j_1 } + u(\vec{T} )   > 1 }  } \left( \prod_{j=1}^{k+r}  \widehat F_j(-u_j) \right)       (1- u_{ j_1 } - u(\vec{T} )  ) \delta( u ([k+r]) )   \> d\vecu         \\
& \ \ \ \ + O( Q/\log Q)  .  }
\end{proof}

Therefore, we have
\est{\mc {M }_U & + \Scal_{L, 0} \\
  =&  Q  (\log Q)^{k + r }  \sqrt{\pi}    \widetilde{\W}(1)     \prod_{p \, \nmid P} \bigg(1 -  \frac{1}{p^2} - \frac{1}{ p^3} \bigg)   \prod_{p |P } \bigg( 1 - \frac1p \bigg)    \sum_{ \substack{ 1 \leq j_1 \leq k \\  k+1 \leq j_2 \leq k+r } }       \summany_{\substack{T_1,W_1, T_2, W_2, T_3, W_3 \\ T_1 + W_1 = \{1,..,j_1 -1\} \cup \{k + 1,...,  j_2 -1\} \\ T_2 + W_2 = \{ j_1 + 1,.., k\} \\ T_3 + W_3 = \{j_2  + 1, ..., k+r\} }}\\ 
	& \times  (-1)^{j_1  + r+ |W_2 |+| W_3|}         \int_{ \substack{ \mc D_{k+r }(\vec{T}, \vec{W})  \\ u_{ j_1 } + u(\vec{T} )   > 1 }  } \left( \prod_{j=1}^{k+r}  \widehat F_j(-u_j) \right)       (1- u_{ j_1 } - u(\vec{T} )  ) \delta( u ([k+r]) )   \> d\vecu       \\
& \ \ \ \ \ + O( Q(\log Q)^{k+r-1}\log \log Q  + Q^{1+\epsilon}/C ) ,  }
which proves \eqref{lem:SUplusSL}.


\subsection{Conclusion of the proof of Proposition \ref{prop:als}}\label{proof of prop part 4}
By Equations \eqref{eqn:splittoUL} and \eqref{lem:SUplusSL}, and  Lemmas \ref{lem:SU} and \ref{lem:SLE}, we have 
\es{  \Scal = &  Q (\log Q)^{k+r}    \sqrt{\pi}      \widetilde{\W}(1)  \prod_{p \, \nmid P} \bigg( 1 - \frac{1}{p^2}- \frac{1}{ p^3} \bigg)  \prod_{p|P} \bigg( 1 - \frac1p \bigg)  \mc {I}(k,r) \\
&   + O \left( Q(\log Q)^{k+r-1}\log \log Q  + C Q^{(\kappa'+\kappa'')/2 -1+\epsilon}  + \frac{Q^{1 + \epsilon}}{C}\right) }
for any $\epsilon>0$. Since $\kappa' + \kappa'' \leq 4 -\varepsilon,$ by letting $C = Q^{\varepsilon/3}$, the $O$-terms above are $O(Q(\log Q)^{k+r-1}\log \log Q  )$. Therefore, we finally have
\est{ \Scal = &  Q (\log Q)^{k+r}    \sqrt{\pi}      \widetilde{\W}(1)  \prod_{p \, \nmid P} \bigg( 1 - \frac{1}{p^2}- \frac{1}{ p^3} \bigg)  \prod_{p|P} \bigg( 1 - \frac1p \bigg)  \mc {I}(k,r)    + O(Q(\log Q)^{k + r -1} \log \log Q ).}

\subsection{The estimation of $N_{\underline{G}}$ }\label{complete NG} 
To complete the calculation of $N_{\underline{G}}$, we first evaluate $\Scal(P; S_{12}, S_{22})$ defined in \eqref{Spartition}. We list sets $S_{12}$ and $S_{22}$ in increasing order as $S_{12} = \{ \alpha_1 , \dots, \alpha_{ |S_{12}|} \}$ and $ S_{22} = \{ \beta_1 , \dots, \beta_{ |S_{22}|} \}$. By modifying arguments of Proposition \ref{prop:als} we find that
 \begin{align*}
& \Scal(P; S_{12}, S_{22}) \\
& \ \ \ \  =  Q (\log Q)^{|S_{12}|+|S_{22}|}    \sqrt{\pi}      \widetilde{\W}(1)  \prod_{p \, \nmid P} \bigg( 1 - \frac{1}{p^2}- \frac{1}{ p^3} \bigg)  \prod_{p | P} \pr{1 - \frac 1p} \mc {I}\big(S_{12}, S_{22} \big)   \\
 & \hskip 1in + O\left( Q(\log Q)^{|S_{12}|+|S_{22}| - 1}   \log \log Q  \right),
 \end{align*} 
  where 
  \es{ \label{def:I}   \mc {I}\big(S_{12}, S_{22}\big)     
 &:= \sum_{ \substack{ 1 \leq j_1 \leq |S_{12}| \\   1 \leq j_2 \leq |S_{22}| } }       \summany_{\substack{T_1,W_1, T_2, W_2, T_3, W_3 \\ T_1 + W_1 = \{ \alpha_1,.., \alpha_{j_1 -1} \} \cup \{\beta_1,..., \beta_{  j_2 -1} \} \\ T_2 + W_2 = \{\alpha_{ j_1 + 1},.., \alpha_{|S_{12}|} \} \\ T_3 + W_3 = \{\beta_{j_2  + 1} , ..., \beta_{|S_{22}|}  \} }} (-1)^{j_1  + |S_{22}|+ |W_2 |+| W_3|}    \\ 
 &    \times    \int_{ \substack{ \mc D_{|S_{12}|+|S_{22}| }(\vec{T}, \vec{W})  \\ u_{ \alpha_{j_1 }} + u(\vec{T} )   > 1 }  } \left( \prod_{j \in S_{12}\cup S_{22}}  \widehat F_j(-u_j) \right)       (1- u_{ \alpha_{j_1 }} - u(\vec{T} )  ) \delta( u (S_{12}) + u(S_{22}) )   \> d\vecu   }
 and
  \est{ \mc D_{|S_{12}|+|S_{22}| }(\vec{T}, \vec{W})   =   \bigg\{ \vecu & = ( u_{\alpha_1} , \dots, u_{\alpha_{|S_{12}|}}, u_{ \beta_1 } , \dots, u_{\beta_{|S_{22}|}}   ) \in \R^{|S_{12}|+|S_{22}|}  \\
  &  : u_j < 0  \ \textrm{for} \ j \in T_1 \cup T_3 \cup W_3, \ \textrm{and} \  u_j > 0 \ \textrm{for} \ j \in T_2  \cup W_1  \cup W_2 \bigg\} }
with   $ d\vecu = du_{\alpha_1}  \cdots du_{\alpha_{|S_{12}|}} du_{ \beta_1 }  \cdots du_{\beta_{|S_{22}|}}   $. By Equation \eqref{eqn:CNg} and Lemma \ref{lemma:D asymp}, we have
\est{
N_{\underline{G}}  = & D(\W,Q)  \sum_{    S_1 + S_2 + S_3=  [ \nu ]    }   \bigg( \prod_{\ell \in S_3 } \widehat{F}_\ell (0) \bigg)   (-1)^{|S_1|+ |S_2|} \sum_{ \substack{S_{11}+S_{12} =  S_1 \\   S_{21}+ S_{22} = S_2 \\ |S_{11}| = |S_{21}| \\ S_{12}, S_{22} \neq \emptyset }} \mc {I}\big(S_{12}, S_{22} \big) 
\sum_{ \substack{ \sigma : S_{11} \to S_{21}   \\   bijection} }  \frac{1}{(\log Q)^{2|S_{11}|}}   \\
&  \times  \sum_{ \substack{P  } }\mu^2 (P)  \bigg( \prod_{\ell \in S_{11}}  \frac{   ( \log p_\ell)^2 }{ p_\ell }  \widehat{F}_\ell  \bigg( - \frac{   \log p_\ell }{ \log Q} \bigg)  \widehat{F}_{ \sigma (\ell)}  \bigg(   \frac{   \log p_\ell }{ \log Q} \bigg) \bigg) \prod_{p | P} \bigg( 1 - \frac{1}{p^2}- \frac{1}{ p^3} \bigg)^{-1} \pr{1 - \frac 1p} \\
& \ \ \ + O( Q  \log \log Q   / \log Q)     .
} 
Modifying the proof of Lemma \ref{lem:diagonalCD}, we can show that 
\es{ \label{eqn:NGasymp}
N_{\underline{G}}  = &   D(\W,Q)  \sum_{    S_1 + S_2 + S_3=  [ \nu ]    }   \bigg( \prod_{\ell \in S_3 } \widehat{F}_\ell (0) \bigg)   (-1)^{|S_1|+ |S_2|}   \\
& \times \sum_{ \substack{S_{11}+S_{12} =  S_1 \\   S_{21}+ S_{22} = S_2 \\ |S_{11}| = |S_{21}| \\ S_{12}, S_{22} \neq \emptyset }} \mc {I}\big(S_{12}, S_{22} \big) 
\sum_{ \substack{ \sigma : S_{11} \to S_{21}   \\   bijection} }      \bigg( \prod_{\ell \in S_{11}}  \int_0^\infty v \widehat{F}_\ell (-v) \widehat{F}_{\sigma(\ell)} (v) dv \bigg) \\
& \ \ \ \ + O( Q \log \log Q  / \log Q)      .
} 
 Therefore, by Equations \eqref{eqn:rc rel}, \eqref{eqn:C1GNG} and \eqref{eqn:NGasymp} we conclude that 
 \es{ \label{eqn:lhs estimation complete}   \frac{\mathcal{L}_1 (f, \W, Q)}{  D(\W,Q) }       =&     \sum_{   \underline{ G} \in \Pi_n  } \mu_n ( \underline{O}, \underline{G})    \sum_{   S_1 + S_2 + S_3=  [ \nu ]   }   \bigg( \prod_{\ell \in S_3 } \widehat{F}_\ell (0) \bigg)  
	\sum_{ \substack{ \sigma : S_{1} \to S_{2}   \\   bijection} } \bigg(\prod_{\ell \in S_1 }
 \int_0^\infty v \widehat{F}_\ell (-v) \widehat{F}_{\sigma(\ell)} (v) dv \bigg) \\
	& +  \sum_{   \underline{ G} \in \Pi_n  } \mu_n ( \underline{O}, \underline{G})   \sum_{    S_1 + S_2 + S_3=  [ \nu ]    }   \bigg( \prod_{\ell \in S_3 } \widehat{F}_\ell (0) \bigg)   (-1)^{|S_1|+ |S_2|}    \sum_{ \substack{S_{11}+S_{12} =  S_1 \\   S_{21}+ S_{22} = S_2 \\ |S_{11}| = |S_{21}| \\ S_{12}, S_{22} \neq \emptyset }} \mc {I}\big(S_{12}, S_{22} \big) \\
	& \times \quad \sum_{ \substack{ \sigma : S_{11} \to S_{21}   \\   bijection} }      \bigg( \prod_{\ell \in S_{11}}  \int_0^\infty v \widehat{F}_\ell (-v) \widehat{F}_{\sigma(\ell)} (v) dv \bigg)   \\
	&   +   O\left(\frac{ \log \log Q}{ \log Q }\right) , }   
	where
	 $ \mu_n ( \underline{O}, \underline{G}) $ is defined in Lemma \ref{lemma:cs}  
 and $	\mc {I}\big(S_{12}, S_{22} \big)$ is defined in \eqref{def:I}.

 \section{Comparison with Random Matrix Theory} \label{sec:RMT}
 In this section, we will complete the proof of Theorem \ref{main thrm} by comparing \eqref{eqn:lhs estimation complete} with the integral 
 $$   \int_{ \R^n }  f  (  \vecx ) W^{(n)} ( \vecx )   d\vecx $$
  in \eqref{eq of main thm}.  As mentioned earlier, though this integration is in a nice form, we may need to go through complicated combinatorial arguments to match \eqref{eqn:lhs estimation complete} with the integral above. Instead we shall use a new formula from Conrey and Snaith's work \cite{ConreySnaith}. In particular, we will work with Equation \eqref{eqn:initialConSn} for $\mathcal R$. This $\mathcal R$ composes of two components, say $\mathcal R_0$ and $\mathcal R_1$.  
  
 For $i = 0, 1$, $\mathcal R_i$ corresponds to the sum over $S, T$ in $J^*(A; B)$ (defined in \eqref{def:JAB}), where  $|S| = |T| = i$. With careful analysis for residues of contour integrations in $\mathcal R_i$, it turns out that $\mathcal R_0$ matches with the diagonal terms of $C_{1,\underline{G}}$ in Equation \eqref{eqn:C1GCG} (see Lemma \ref{lem:J0}), and $\mathcal R_i$ corresponds to the off-diagonal terms $N_{\underline G}$ of $C_{1,\underline{G}}$ (see Section \ref{sec:calcJ1}). Moreover, each component $\mathcal I(S_{12}, S_{22})$ in the off-diagonal terms $N_{\underline G}$ corresponds to each component $\mathcal J(S_{12}, S_{22})$ of $\mathcal R_1$ defined in Equation (\ref{def:JS12S22}) (see Lemma \ref{lem:J1}).

 To do this, first we need the following lemma, which expresses the integral as the limit of $n$-correlation of eigenvalues of random unitary matrices of size $N \to \infty$.
 
 \begin{lemma} Let $ f : \R^n \to \R$ be smooth and rapidly decreasing. For an $N \times N$ unitary matrix $X_N$, write its eigenvalues as $ e^{i \theta_j}$ with $-\pi \leq \theta_1 \leq \cdots \leq \theta_N <  \pi$. Then 
 	\begin{equation*}
 	\lim_{N \to \infty}  \int_{U(N)} \sumstar_{1 \leq j_1 , \dots, j_n \leq N} f\bigg( \frac{N \theta_{j_1}}{2 \pi} , \dots , \frac{N \theta_{ j_n}}{2 \pi}  \bigg)  dX_N    =   \int_{ \R^n }  f  (  \vecx ) W^{(n)} ( \vecx )   d\vecx ,
 	\end{equation*}
 where $dX_N$ is the Haar measure on the group of $N \times N$ unitary matrices $U(N)$ and $W^{(n)}(\vecx)$ is defined in (\ref{def:Wn}).
 \end{lemma} 
  Note that the condition $ - \pi \leq \theta_1 \leq \cdots \leq \theta_N < \pi$  in the above lemma is also required for Theorem 3.4 of \cite{ConreySnaith}, which will be used in the proof of Proposition \ref{sn:eq1}. 
  \begin{proof}
  By Theorem 3.1 of \cite{ConreySnaith}, we have
  \begin{equation*}
  \begin{split}
   \int_{U(N)} & \> \sumstar_{1 \leq j_1 , \dots, j_n \leq N} f\bigg( \frac{N \theta_{j_1}}{2 \pi} , \dots , \frac{N \theta_{ j_n}}{2 \pi}  \bigg) dX_N  \\
  & = \frac{1}{(2 \pi )^n} \int_{ [-\pi ,  \pi]^n }  f \bigg(  \frac{Nx_1}{2 \pi} , \dots, \frac{ N x_n}{2 \pi} \bigg) \det_{n \times n } S_N ( x_k - x_j ) d\vecx  \\
  & =    \int_{ [-N/2 , N/2 ]^n }  f  (   x_1  , \dots,  x_n ) \frac{1}{N^n}  \det_{n \times n } S_N \bigg( \frac{ 2 \pi}{N}( x_k -x_j ) \bigg)  d\vecx  ,
  \end{split}
  \end{equation*}
  where 
  $$ S_N ( x ) = \frac{  \sin( N x /2)}{ \sin ( x/2)} . $$
 It is easy to see that 
$$  \lim_{N \to \infty}   \frac{1}{N^n}  \det_{n \times n } S_N \bigg( \frac{ 2 \pi}{N}( x_k - x_j ) \bigg)   = W^{(n)} ( \vecx ) .$$ 
  Since $f$ has a rapid decay, we have 
  \begin{equation*}
  \begin{split}
  & \lim_{N \to \infty} \int_{U(N)} \sumstar_{1 \leq j_1 , \dots, j_n \leq N} f\bigg( \frac{N \theta_{j_1}}{2 \pi} , \dots , \frac{N \theta_{ j_n}}{2 \pi}  \bigg) dX_N \\ 
  & =    \int_{ \R^n }  f  (   x_1  , \dots,  x_n )  \lim_{N \to \infty} \frac{1}{N^n}  \det_{n \times n } S_N \bigg( \frac{ 2 \pi}{N}( x_k -x_j ) \bigg)  d\vecx    =    \int_{ \R^n }  f  ( \vecx )  W^{(n)} ( \vecx )   d\vecx . 
    \end{split}
  \end{equation*}
 \end{proof}
 By the above lemma, Theorem \ref{main thrm} is equivalent to   
 \est{ \lim_{ Q \to \infty} \frac{ \mathcal{L}_1 ( f, \W, Q)}{D(\W, Q)} = \lim_{N \rightarrow \infty} \int_{U(N)} \sumstar_{1 \leq j_1 , \dots, j_n \leq N} f\bigg( \frac{N \theta_{j_1}}{2 \pi} , \dots , \frac{N \theta_{ j_n}}{2 \pi}  \bigg) dX_N. }
 Let  $ f(x_1 , \dots, x_n ) = \prod_{i=1}^n f_i (x_i) $ have $C4$-property and let $ \underline{G}= \{G_ 1 , \dots,G_{\nu } \} \in \Pi_n$ be a partition of $ [n]= \{1, 2,..., n\}$. Define
 $$ F_\ell (x) = \prod_{i \in G_\ell} f_i (x)   $$
  as in \eqref{def:Fell}.
 By combinatorial sieving in Lemma \ref{lemma:cs},  we have
 \begin{equation}\label{eqn:combsieveRandom}\begin{split}
 &   \lim_{N \to \infty}  \int_{U(N)} \sumstar_{1 \leq j_1 , \dots, j_n \leq N} f\bigg( \frac{N \theta_{j_1}}{2 \pi} , \dots , \frac{N \theta_{ j_n}}{2 \pi}  \bigg)  dX_N  \\
 & = \lim_{N \to \infty}  \int_{U(N)}    \sum_{   \underline{ G} \in \Pi_n  } \mu_n ( \underline{O}, \underline{G})    \sum_{1 \leq j_1, ..., j_\nu \leq N} \prod_{\ell = 1}^{\nu} F_{\ell}\left(\frac{N\theta_{j_\ell}}{2\pi}\right) \> dX_N  .
 \end{split}\end{equation}
Then Theorem \ref{main thrm} can be deduced from Equations \eqref{eqn:rc rel}, \eqref{eqn:combsieveRandom} and the following proposition.

\begin{proposition} \label{sn:eq1} 
Let $D(\W, Q)$ and $C_{1, \underline{G}} $   be defined in Equations \eqref{def:DWQ} and \eqref{eqn:C1Fginitial}, respectively. Then
 \begin{equation*}
 \lim_{Q \rightarrow \infty} \frac{C_{1, \underline{G}} }{D(\W, Q)}  = \lim_{N \to \infty}  \int_{U(N)}   \sum_{1 \leq j_1, ..., j_\nu \leq N} \prod_{\ell = 1}^{\nu} F_{\ell}\left(\frac{N\theta_{j_\ell}}{2\pi}\right) \> dX_N.
 \end{equation*}
	
\end{proposition}

\subsubsection*{Proof of Proposition \ref{sn:eq1}} 
 We want to apply Theorems 3.3 and 3.4 in \cite{ConreySnaith} to 
$$   \mc{R} :=      \lim_{N \to \infty}  \int_{U(N)}   \sum_{1 \leq j_1, ..., j_\nu \leq N} \prod_{\ell = 1}^{\nu} F_{\ell}\left(\frac{N\theta_{j_\ell}}{2\pi}\right) \> dX_N  .$$
Theorem 3.4 in \cite{ConreySnaith} requires the periodicity of each function $F_{\ell}$ to cancel the integrals on the horizontal segments in the proof \cite[p.499]{ConreySnaith}. In our case, since we are taking the limit $N \to \infty$, we see that the limits of the integrals on the horizontal segments converge to zero. Thus, we can still apply Theorems 3.3 and 3.4 to find that
    \begin{equation} \begin{split} \label{eqn:initialConSn}
   \mc{R}   = &   \lim_{N \to \infty }  \frac{1}{(2 \pi i )^\nu}   \sum_{  S_1+S_2+S_3   = [ \nu ] }  (-1)^{|S_1|+|S_3|}   N^{|S_3|} \\
& \times  \int_{ \mathcal{C}_+^{|S_2|}}  \int_{ \mathcal{C}_-^{|S_1|+|S_3|} } J^* (z_{S_2} ; -z_{S_1} )  \prod_{\ell = 1}^{\nu} F_{\ell}\left(\frac{ iN }{2\pi} z_\ell \right) dz_{S_3} dz_{S_1} dz_{S_2}  ,
  \end{split} \end{equation}
where 
$ \mathcal{C}_+ $ denotes the path from $ \delta_1 - \pi i $ up to $ \delta_1 + \pi i $, $ \mathcal{C}_- $ denotes the path from $ - \delta_1 + \pi i $ down to $ - \delta_1 - \pi i $ for some $ \delta_1 >0$, $z_{S_i} = \{ z_\ell : \ell \in S_i \} $, $ -z_{S_i } = \{ -z_\ell : \ell \in S_i \} $ and $dz_{S_i } = \prod_{\ell \in S_i } dz_\ell $. 
 	\es{ \label{def:JAB} J^*  (A ; B) := \sum_{\substack{S \subset A, T \subset B \\ |S| = |T|  }} e^{-N(\sum_{\widehat{\alpha} \in S} \widehat{\alpha} + \sum_{\widehat{\beta} \in T} \widehat{\beta})} \frac{Z(S,T)Z(S^-, T^-)}{Z^\dagger(S, S^-) Z^\dagger(T, T^-)} \sum_{\substack{(A-S) + (B - T) \\ = U_1 + ... + U_Y \\ |U_y| \leq 2}} \prod_{y = 1}^Y H_{S, T}(U_y),}
where $S^- = \{ - \hat{\alpha} : \hat{\alpha} \in S \}$, 
$$Z (A, B) = \prod_{\substack{\alpha \in A \\ \beta \in B }} z(\alpha + \beta), \qquad  Z^\dagger(A, B) = \prod_{\substack{\alpha \in A \\ \beta \in B \\ \alpha + \beta \neq 0}} z(\alpha + \beta) $$
with $z(x) = ( 1- e^{-x})^{-1} $,  $A-S = \{ \alpha \in A : \alpha \not \in S \}$, $(A-S) + (B-T)$ is the disjoint union of two sets $A-S$ and $B-T$, and 
 	\est{H_{S, T}(W) = \left\{ \begin{array}{ll} \sum_{\widehat{\alpha} \in S} \frac{z'}{z}(\alpha - \widehat{\alpha}) - \sum_{\widehat{\beta} \in T} \frac{z'}{z}(\alpha + \widehat{\beta}) & {\rm if} \ W = \{\alpha\} \subset A-S ,\\
 			\sum_{\widehat{\beta} \in T} \frac{z'}{z}(\beta - \widehat{\beta}) - \sum_{\widehat{\alpha} \in S} \frac{z'}{z}(\beta + \widehat{\alpha}) & {\rm if} \ W = \{\beta\} \subset B-T ,\\
 			\pr{\frac{z'}{z}}'(\alpha + \beta) & {\rm if} \ W = \{\alpha, \beta\} \ \textrm{with} \ \alpha \in A-S, \beta \in B - T ,\\
 			0 & {\rm otherwise.}\end{array}\right.}
The innermost sum of $J^*(A; B)$ is the sum over all partitions of $(A-S) + (B - T)$ into singletons or doubletons $U_1, ..., U_Y $.

 We change the orientation of the $z_\ell$-integral in \eqref{eqn:initialConSn} for each $ \ell \in S_1 \cup S_3 $ and this removes the factor $ (-1)^{|S_1|+|S_3|} $. Since $ F_\ell$ is rapidly decreasing, we can extend each vertical integrals. Thus 
$$  \mc{R}=   \lim_{N \to \infty }  \frac{1}{(2 \pi i )^\nu}   \sum_{  S_1+S_2+S_3   = [ \nu ] }    N^{|S_3|} \int_{ (\delta_1)^{|S_2|}}  \int_{ (-\delta_1)^{|S_1|+|S_3|} } J^* (z_{S_2} ; -z_{S_1} )  \prod_{\ell = 1}^{\nu} F_{\ell}\left(\frac{ iN }{2\pi} z_\ell \right)  dz_{S_3} dz_{S_1} dz_{S_2}  .
$$
 Since
$$\frac{N}{2 \pi i } \int_{(- \delta_1)} F_\ell \bigg( \frac{iN}{2 \pi} z_\ell \bigg)  dz_\ell  =  \frac{1}{i}\int_{(0)} F_{\ell}(i z) dz = \widehat{F}_{\ell}(0) $$
for each $\ell \in S_3$, 
\est{   \mc{R}=  &    \sum_{  S_1+S_2+S_3   = [ \nu ] } \bigg( \prod_{\ell \in S_3} \widehat{F}_\ell (0) \bigg)  \mc{R}(S_1, S_2), }
where 
\est{
\mc{R} (S_1, S_2 ) :=&   \lim_{N \to \infty } \frac{1}{(2 \pi i )^{|S_1|+|S_2|}}   \int_{ (\delta_1)^{|S_2|}}  \int_{ (-\delta_1)^{|S_1| } } J^* (z_{S_2} ; -z_{S_1} )  \prod_{\ell \in S_1 \cup S_2 } F_{\ell}\left(\frac{ iN }{2\pi} z_\ell \right)  dz_{S_1}   dz_{S_2}  . }

We now consider $J^* (z_{S_2} ; -z_{S_1} )$. When $ |S|=|T|\geq 2 $,
$$ | e^{-N(\sum_{\widehat{\alpha} \in S} \widehat{\alpha} + \sum_{\widehat{\beta} \in T} \widehat{\beta})} | \leq e^{ -4N\delta_1}.$$
Combining above with  \eqref{def:Fell} and the C4-Property in Section \ref{sec:intro},  we have
$$ \prod_{\ell = 1}^{\nu} F_{\ell}\left(\frac{ iN }{2\pi} z_\ell \right)  \ll  \prod_{\ell = 1}^{\nu}  \int_{-\infty}^\infty | \widehat{F}_{\ell}(v)   e^{Nz_\ell v } | dv  \ll e^{N \delta_1 (4-\varepsilon) } $$
 and  the contribution to $\mc{R}$  is  
$$ \ll   N^\nu   e^{-\varepsilon N \delta_1} \to 0 $$
as $ N \to \infty$.  Hence, the main contribution of $\mc{R}$ comes solely from the cases $|S|=|T|=0,1$.  
Let $\mathcal J_i$ be the contribution from the case $|S| = |T| = i$ for $i = 0, 1$. Then 
\es{ \label{eqn:twocaseJ}\mc{R} (S_1 , S_2 ) 
&=:  \mathcal J_0(S_1 , S_2) + \mathcal J_1(S_1 , S_2) .}
Define
\es{\label{def:Ri}   \mc{R}_i =     \sum_{  S_1+S_2+S_3   = [ \nu ] } \bigg( \prod_{\ell \in S_3} \widehat{F}_\ell (0) \bigg)  \mathcal J_i (S_1, S_2)  }
for each $i = 0,1 $, so that
\es{\label{def:R} \mc{R} = \mc{R}_0 + \mc{R}_1 . }

 \subsection{ Calculation of $\mathcal J_0 (S_1, S_2)$} In this section, we will show the following lemma.
 
 \begin{lemma} \label{lem:J0} Let notations be as defined above. Then
 	\begin{equation*} \begin{split}
 	\mathcal{J}_0(S_1,  S_2) = &  \sum_{ \substack{ \sigma: S_1 \to S_2 \\ bijection} }         \prod_{\ell \in S_1} \bigg(        \int_0^\infty    v \widehat{F}_\ell ( -v  )  \widehat{F}_{\sigma(\ell)} (v)  \> dv      \bigg), 
 	\end{split} \end{equation*}
and so it can be easily deduced that 
 	\est{   \mc{R}_0 =     \sum_{  S_1+S_2+S_3   = [ \nu ] } \bigg( \prod_{\ell \in S_3} \widehat{F}_\ell (0) \bigg)    \sum_{ \substack{ \sigma: S_1 \to S_2 \\ bijection} }         \prod_{\ell \in S_1} \bigg(        \int_0^\infty    v \widehat{F}_\ell ( -v  )  \widehat{F}_{\sigma(\ell)} (v)  \> dv      \bigg) .  }
\end{lemma}

\begin{proof}
  For this case, $S = T = \emptyset$, and so $\mathcal J_0 (S_1, S_2)$ equals
$$  \lim_{N \to \infty } \frac{1}{(2 \pi i )^{|S_1|+|S_2|}}   \int_{ (\delta_1)^{|S_2|}}  \int_{ (-\delta_1)^{|S_1| } } \sum_{\substack{ z_{S_2}   + (-z_{S_1})     \\ = U_1 + \cdots + U_Y \\ |U_y| \leq 2}} \prod_{y = 1}^Y H_{ \emptyset, \emptyset} (U_y)    \prod_{\ell \in S_1 \cup S_2 } F_{\ell}\left(\frac{ iN }{2\pi} z_\ell \right)  dz_{S_1}   dz_{S_2}  
$$
for $\delta_1 > 0$, where
\est{H_{ \emptyset, \emptyset} (W) = \left\{ \begin{array}{ll} 
		\pr{\frac{z'}{z}}'(   \alpha + \beta ) & {\rm if} \ W = \{\alpha, \beta\} \ \textrm{with} \ \alpha \in   z_{S_2} , \beta \in    - z_{S_1}  ,  \\
		0 & {\rm otherwise.}\end{array}\right.}
Notice that $ \prod_{y = 1}^Y H_{ \emptyset, \emptyset} (U_y)$ is non-zero if and only if every $U_j$ contains one element from $-z_{S_1}$ and the other element from $z_{S_2}$. Thus, for each partition $U_1 + \cdots + U_Y$, there is a natural bijection $\sigma : S_1 \to S_2 $, defined by 
$$\sigma(\ell)= \ell'$$ 
when $ -z_\ell  , z_{\ell'} \in U_j $ for some $j$. Hence $ \mathcal{J}_0(S_1,  S_2)$ equals
    \begin{equation*} \begin{split}
 &    \sum_{ \substack{ \sigma: S_1 \to S_2 \\ bijection} }  \lim_{N \to \infty }       \prod_{\ell \in S_1} \bigg(  \frac{1}{(2 \pi i )^2 }    \int_{ (\delta_1) }  \int_{ (-\delta_1)  }  \bigg( \frac{ z'}{z} \bigg)'  ( z_{\sigma(\ell)} - z_{\ell})     F_{\ell}\left(\frac{ iN }{2\pi} z_\ell \right) F_{\sigma(\ell)}\left(\frac{ iN }{2\pi} z_{\sigma(\ell)}  \right)  \> dz_{\ell}  \> dz_{\sigma(\ell)}   \bigg)  \\
 & = \sum_{ \substack{ \sigma: S_1 \to S_2 \\ bijection} }  \lim_{N \to \infty }       \prod_{\ell \in S_1} \bigg(  \frac{1}{(iN)^2 }    \int_{ (\delta_1) }  \int_{ (-\delta_1)  }  \bigg( \frac{ z'}{z} \bigg)' \bigg( \frac{  2\pi }{N} ( z_{\sigma(\ell)} - z_{\ell}) \bigg)  F_{\ell} (i z_{ \ell})   F_{\sigma(\ell)} ( i z_{\sigma(\ell)}  )  \> dz_{\ell}  \> dz_{\sigma(\ell)}   \bigg)  .
\end{split} \end{equation*}
Since
 $$  \lim_{N \to \infty} \frac{1}{(iN)^2}    \pr{\frac{z'}{z}}' \bigg( \frac{ 2 \pi }{N}( z_{\sigma(\ell)} - z_\ell )\bigg) = \frac{ 1 }{ ( 2 \pi i)^2  ( z_{\sigma(\ell)} - z_\ell  )^2  }  ,$$
we have
    \begin{equation} \label{double integral 1}
\mathcal{J}_0(S_1,  S_2) =   \sum_{ \substack{ \sigma: S_1 \to S_2 \\ bijection} }         \prod_{\ell \in S_1} \bigg(  \frac{1}{(2 \pi i )^2 }    \int_{ (\delta_1) }  \int_{ (-\delta_1)  }   F_{\ell} (i z_{ \ell})   F_{\sigma(\ell)} ( i z_{\sigma(\ell)}  )  \frac{1}{ (z_\ell - z_{\sigma(\ell)} )^2}\> dz_{\ell}  \> dz_{\sigma(\ell)}   \bigg)  .
\end{equation}
The double integrals above is
 \es{ \label{double integral 2}     & =    \frac{1}{( 2 \pi i )^2 } \int_{(\delta_1)}   \int_{ ( - \delta_1 ) }   \frac{1 }{( z_{\sigma(\ell)} - z_\ell )^2 }  F_{\sigma(\ell)} ( i z_{\sigma(\ell)}  ) \int_\R  e^{ - 2 \pi  v z_\ell}    \widehat{F}_\ell(-v) \> dv \> dz_\ell \> dz_{\sigma(\ell)} \\
    & =    \frac{1}{( 2 \pi i )^2 } \int_{(\delta_1)} \int_{\R} \int_{ ( - \delta_1 ) }   \frac{ e^{ - 2 \pi  v z_\ell}  }{( z_{\sigma(\ell)} - z_\ell )^2 }  \> dz_\ell  F_{\sigma(\ell)} ( i z_{\sigma(\ell)}  )  \widehat{F}_\ell(-v) \> dv \> dz_{\sigma(\ell)}. }
  We shift the $z_\ell$ contour right when $ v \geq 0 $ with the residue at the pole $z_\ell = z_{\sigma(\ell)}$ and left when $ v < 0 $, then the above is
    \es{\label{double integral 3} 
   &=   \frac{1}{  2 \pi i  } \int_{(\delta_1)} \int_{0}^\infty        2 \pi v    e^{ - 2 \pi  v z_{\sigma(\ell)} }   F_{\sigma(\ell)} ( i z_{\sigma(\ell)}  )  \widehat{F}_\ell(-v) \> dv \> dz_{\sigma(\ell)}  \\
   & =           \int_0^\infty    v \widehat{F}_\ell ( -v  )  \widehat{F}_{\sigma(\ell)} (v)  \> dv  ,   }
and this completes the proof of the lemma.

\end{proof}
 \subsection{ Calculation of $\mathcal J_1 (S_1, S_2)$}  \label{sec:calcJ1}
 This is the case $|S| = |T |=1 $ in $\mc{R}(S_1 , S_2 )  $. There exist $ \alpha \in S_1 $ and $ \beta \in S_2 $ such that $ S = \{ z_\beta \} $ and $ T = \{ - z_\alpha \} $. For $\delta_1 > 0$, we then have
 \est{
 \mc{J}_1 (S_1, S_2 ) :=&   \lim_{N \to \infty } \frac{1}{(2 \pi i )^{|S_1|+|S_2|}}   \int_{ (\delta_1)^{|S_2|}}  \int_{ (-\delta_1)^{|S_1| } }   \prod_{\ell \in S_1 \cup S_2 } F_{\ell}\left(\frac{ iN }{2\pi} z_\ell \right)  \\    
 & \times \sum_{\substack{ \alpha \in S_1 \\ \beta  \in S_2 }} \frac{ e^{ N( z_\alpha - z_\beta) }           }{   (1-e^{ z_\alpha - z_\beta} ) ( 1- e^{-z_\alpha + z_\beta}) } \sum_{\substack{ ( - z_{ S_1 \setminus \{ \alpha \} } ) + z_{ S_2 \setminus \{ \beta \} }   \\ = U_1 + ... + U_Y \\ |U_y| \leq 2}} \prod_{y = 1}^Y H_{\{ z_\beta \} , \{ - z_\alpha \} }(U_y)    dz_{S_1}   dz_{S_2}  , }
  where  for $ \ell_1 \in S_1  \setminus \{  \alpha\}  $ and $  \ell_2  \in S_2 \setminus \{  \beta \} $, 
 \begin{equation*}\begin{split}
 H_{\{ z_\beta \} , \{ - z_\alpha \} } ( \{ z_{\ell_2}  \} ) &  =  \frac{1}{ 1- e^{ z_{\ell_2} - z_\beta }}  -   \frac{1}{  1- e^{ z_{\ell_2} - z_\alpha }} ,  \\
H_{\{ z_\beta \} , \{ - z_\alpha \} } ( \{ -z_{\ell_1} \}) & =     \frac{1}{ 1- e^{ z_\alpha - z_{\ell_1} } }  -    \frac{1}{ 1- e^{  z_\beta - z_{\ell_1}  } } , \\
H_{\{ z_\beta \} , \{ - z_\alpha \} } ( \{ z_{\ell_2} ,  -z_{\ell_1} \}  ) & =    \frac{e^{  z_{\ell_2}  - z_{\ell_1} }}{ (1- e^{  z_{\ell_2} - z_{\ell_1} }  )^2 }, 
\end{split}\end{equation*}
 and otherwise, $ H_{\{ z_\beta \} , \{ - z_\alpha \} } ( W) = 0 $. As indicated in Remark 3.2 of \cite{ConreySnaith}, even though each term $H_{\{ z_\beta \} , \{ - z_\alpha \} } ( \{ z_{\ell_2} ,  -z_{\ell_1} \}  )$ has a singularity on the contour, the integrand has no poles because they cancel.  We would like to shift contours in such a way that we avoid singularities of the integrand. If $ S_1 = \{ \alpha_1 , \dots, \alpha_{|S_1|} \} $  with $ \alpha_1 < \cdots < \alpha_{|S_1|}$ and $ S_2 = \{ \beta_1 , \dots, \beta_{|S_2|}\} $ with $ \beta_1 < \cdots < \beta_{|S_2|}$, then define
\est{ \int_{ S_1^-} & := \int_{ (- \delta_1)} \cdots \int_{ ( - \delta_{  |S_1| })}, \qquad  \int_{ S_2^+}   := \int_{ ( \delta_{ 1})} \cdots \int_{ ( \delta_{  |S_2| })} }
for some $ 0 < \delta_1 < \cdots < \delta_{\max( |S_1| , |S_2|)} $. After we replace $ \int_{(-\delta_1)^{|S_1|}}$ and $ \int_{( \delta_1)^{|S_2|}}$ by $ \int_{ S_1^-}$ and $\int_{S_2^+}$, respectively,  changing the order of integrals and summations is legitimate.   

We next estimate the integrand of $ \mc{J}_1 (S_1, S_2 )$. 
Define 
\est{ S_{11} &= S_1 \setminus S_{12}, \qquad S_{12}  = \{ \alpha \} \cup \{   \ell \in S_1 :    \{- z_\ell \} = U_y   ~~ \mathrm{for~some}~~ y   \} , \\
S_{21} &= S_2 \setminus S_{22} , \qquad S_{22}   = \{ \beta  \} \cup \{ \ell \in S_2 :    \{z_\ell \} = U_y ~~ \mathrm{for~some}~~ y   \} .  }
Further let a bijection $ \sigma : S_{11} \to S_{21} $  be defined such that for any $ \ell \in S_{11}$, $ \{ -z_\ell , z_{\sigma(\ell )} \} = U_y $ for some $y$. Hence, 
\est{ & \sum_{\substack{ \alpha \in S_1 \\ \beta  \in S_2 }} \frac{ e^{ N( z_\alpha - z_\beta) }           }{   (1-e^{ z_\alpha - z_\beta} ) ( 1- e^{-z_\alpha + z_\beta}) } \sum_{\substack{ ( - z_{ S_1 \setminus \{ \alpha \} } ) + z_{ S_2 \setminus \{ \beta \} }   \\ = U_1 + ... + U_Y \\ |U_y| \leq 2}} \prod_{y = 1}^Y H_{\{ z_\beta \} , \{ - z_\alpha \} }(U_y)    \\   
& = \sum_{ \substack{ S_1 = S_{11}+S_{12} \\ S_2 = S_{21}+ S_{22} \\ S_{12} , S_{22} \neq \emptyset }}    \sum_{ \substack{ \alpha \in S_{12} \\ \beta \in S_{22}}}  \frac{ e^{ N( z_\alpha - z_\beta) }           }{    ( 1- e^{z_\alpha - z_\beta})(1-e^{-z_\alpha + z_\beta} )  }  \prod_{ \ell \in S_{12} \setminus \{ \alpha \} }
H_{\{ z_\beta \} , \{ - z_\alpha \} } ( \{ -z_\ell \}) \prod_{ \ell \in S_{22} \setminus \{ \beta \} } H_{\{ z_\beta \} , \{ - z_\alpha \} } ( \{ z_\ell  \} )  \\
&  \hskip 1in \times \sum_{  \substack{ \sigma:S_{11} \to S_{21} \\ bijection}}  \prod_{ \ell \in S_{11}}
H_{\{ z_\beta \} , \{ - z_\alpha \} } ( \{ z_{\sigma(\ell)} ,  -z_\ell \}  ).
 } 
 We then apply the above identity to $  \mc{J}_1 (S_1, S_2 ) $, substitute $ z_\ell$ by $ 2 \pi z_\ell / N$ for all $ \ell \in S_1 \cup S_2 $ and take the limit $N \to \infty$. Since
 $$  \lim_{N \to \infty} iN ( 1- e^{2 \pi x/N}) = - 2 \pi i x   , $$ 
 we find that 
 \est{
 \mc{J}_1 (S_1, S_2 )  =& \sum_{ \substack{ S_1 = S_{11}+S_{12} \\ S_2 = S_{21}+ S_{22} \\ S_{12} , S_{22} \neq \emptyset  }}  \frac{1}{(  2 \pi i  )^{|S_1|+|S_2|}}   \int_{ S_2^+}  \int_{ S_1^- }   \prod_{\ell \in S_1 \cup S_2 } F_{\ell}(i z_\ell  )   \\   
&  \times     \sum_{ \substack{ \alpha \in S_{12} \\ \beta \in S_{22}}}  \frac{ - e^{  2 \pi  ( z_\alpha - z_\beta) }           }{   ( z_\alpha - z_\beta  )^2 }   \prod_{ \ell \in S_{12} \setminus \{ \alpha \} } \bigg(  \frac{1}{ - z_\alpha + z_\ell    }  -    \frac{1}{ -  z_\beta + z_\ell   }  \bigg)  \prod_{ \ell  \in S_{22} \setminus \{ \beta \} } \bigg(  \frac{1}{ - z_\ell + z_\beta  }  -   \frac{1}{   - z_\ell + z_\alpha }  \bigg)  \\
&  \times  \sum_{  \substack{ \sigma:S_{11} \to S_{21} \\ bijection}}  \prod_{ \ell \in S_{11}}  \frac{1}{ (  z_{\sigma(\ell)} - z_\ell    )^2 }          dz_{S_1}   dz_{S_2} . }
By \eqref{double integral 1}--\eqref{double integral 3} and combining the products on $S_{12} \setminus \{\alpha \}$ and $ S_{22} \setminus \{ \beta \} $, we have
 \es{\label{partitionJ1}
 \mc{J}_1 (S_1, S_2 )  =& \sum_{ \substack{ S_1 = S_{11}+S_{12} \\ S_2 = S_{21}+ S_{22} \\ S_{12} , S_{22} \neq \emptyset  }}  \bigg(  \sum_{  \substack{ \sigma:S_{11} \to S_{21} \\ bijection}} \bigg( \prod_{ \ell \in S_{11}}  \int_0^\infty v \widehat{F}_\ell (-v) \widehat{F}_{ \sigma(\ell)} (v) dv   \bigg) \bigg) \mc{J} ( S_{12}, S_{22}) , }
where
\es{ \label{def:JS12S22} 
 \mc{J} ( S_{12}, S_{22})   :=  &    \sum_{ \substack{ \alpha \in S_{12} \\ \beta \in S_{22}}}   \frac{ -1}{(  2 \pi i  )^{|S_{12}|+|S_{22}|}}   \int_{  S_{22}^+ }  \int_{  S_{12}^-  }   \bigg( \prod_{\ell \in S_{12} \cup S_{22} } F_{\ell}(i z_\ell  )\bigg)        \frac{  e^{  2 \pi  ( z_\alpha - z_\beta) }           }{   ( z_\alpha - z_\beta  )^2 }    \\
&  \times \prod_{ \ell \in S_{12} \cup S_{22} \setminus \{ \alpha, \beta \} } \bigg(  \frac{1}{z_\ell - z_\alpha      }  -    \frac{1}{  z_\ell  -  z_\beta    }  \bigg)   dz_{S_{12}}   dz_{S_{22}} . }

The evaluation of $\mathcal J(S_{1}, S_2)$ follows from the calculation of $\mathcal J(S_{12}, S_{22})$ below.

\begin{lemma} \label{lem:J1} Let $\mc{I} ( S_{12}, S_{22})$ be defined as Equation \eqref{def:I}. Then
	$$ \mc{J} (S_{12}, S_{22}) = (-1)^{|S_{12}|+|S_{22}|} \ \mc{I} ( S_{12}, S_{22}) . $$
\end{lemma}

\begin{proof}
Write $S_{12} = \{ \alpha_1 , \dots, \alpha_{|S_{12}|} \} $ with  $ \alpha_1 < \cdots < \alpha_{|S_{12}|}$ and $S_{22} = \{ \beta_1 , \dots, \beta_{|S_{22}|} \} $ with  $ \beta_1 < \cdots < \beta_{|S_{22}|}$. Let $ \alpha = \alpha_{j_1}$ and $\beta = \beta_{j_2}$ in the above sum. Then
\est{ 
 \mc{J} ( S_{12}, S_{22})   =  &    \sum_{ \substack{ j_1 \leq | S_{12}| \\ j_2 \leq | S_{22}| }}   \frac{  1}{(  2 \pi    )^2 }   \int_{   (-\delta_{j_1}) }  \int_{ (\delta_{j_2} )  }     F_{\alpha_{j_1}}(i z_{\alpha_{j_1}}  )  F_{\beta_{j_2}}(i z_{\beta_{j_2}}   )        \frac{  e^{  2 \pi  ( z_{\alpha_{j_1}} - z_{\beta_{j_2}} ) }           }{   ( z_{\alpha_{j_1}}  - z_{\beta_{j_2}}   )^2 }    \\
&  \times \prod_{ \alpha_\ell \in S_{12}   \setminus \{ \alpha_{j_1}  \} } \frac{1}{2\pi i } \int_{ (-\delta_\ell )} F_{\alpha_{\ell}}(i z_{\alpha_{\ell}}  ) \bigg(  \frac{1}{z_{\alpha_\ell } - z_{\alpha_{j_1}}       }  -    \frac{1}{  z_{\alpha_\ell}  -  z_{\beta_{j_2}}     }  \bigg) dz_{\alpha_\ell}     \\
&  \times  \prod_{ \beta_\ell \in   S_{22} \setminus \{   \beta_{j_2} \} } \frac{1}{2\pi i } \int_{ ( \delta_\ell )}  F_{\beta_{\ell}}(i z_{\beta_{\ell}}   )\bigg(  \frac{1}{z_{\beta_\ell} - z_{\alpha_{j_1}}       }  -    \frac{1}{  z_{\beta_\ell } -  z_{\beta_{j_2}}     }  \bigg)  dz_{\beta_\ell} dz_{\beta_{j_2} }   dz_{\alpha_{j_1}} . }
By Lemma \ref{lemma:Fizw1w2}, $ \mc{J} ( S_{12}, S_{22}) $ becomes
\begin{align*}
      \sum_{ \substack{ j_1 \leq | S_{12}| \\ j_2 \leq | S_{22}| }}   \frac{  1}{(  2 \pi    )^2 } &  \int_{   (-\delta_{j_1}) }  \int_{ (\delta_{j_2} )  }        \frac{  e^{  2 \pi  ( z_{\alpha_{j_1}} - z_{\beta_{j_2}} ) }           }{   ( z_{\alpha_{j_1}}  - z_{\beta_{j_2}}   )^2 }   (-1)^{|S_{12}| - j_1 }  F_{\beta_{j_2}}(i z_{\beta_{j_2}}   )    \int_\R   \hat{F}_{\alpha_{j_1}}(-u_{\alpha_{j_1}}  )     e^{ - 2  \pi u_{\alpha_{j_1}}  z_{\alpha_{j_1}} }    du_{\alpha_{j_1}}    \\
&\times \prod_{\ell<j_1 }\bigg(    \int_{-\infty}^0    \hat{F}_{\alpha_{\ell}}(-u_{\alpha_{\ell}} ) e^{-2 \pi u_{\alpha_{\ell}} z_{\alpha_{j_1}}} du_{\alpha_{\ell}} +  \int_0^\infty   \hat{F}_{\alpha_\ell} (-u_{\alpha_{\ell}}) e^{-2 \pi u_{\alpha_{\ell}} z_{\beta_{j_2}}} du_{\alpha_{\ell}}  \bigg)  \\
& \times \prod_{\ell < j_2 }\bigg(    \int_{-\infty}^0    \hat{F}_{\beta_\ell} (-u_{\beta_\ell}) e^{-2 \pi u_{\beta_\ell} z_{\alpha_{j_1}}} du_{\beta_\ell} +  \int_0^\infty   \hat{F}_{\beta_\ell}(-u_{\beta_\ell}) e^{-2 \pi u_{\beta_\ell} z_{\beta_{j_2}}} du_{\beta_\ell}  \bigg)  \\
& \times \prod_{ j_1 < \ell \leq |S_{12}|  } \bigg( \int_0^\infty   \hat{F}_{\alpha_\ell} (-u_{\alpha_{\ell}}) e^{-2 \pi u_{\alpha_{\ell}} z_{\alpha_{j_1}}} du_{\alpha_{\ell}}-   \int_0^\infty   \hat{F}_{\alpha_\ell} (-u_{\alpha_{\ell}}) e^{-2 \pi u_{\alpha_{\ell}} z_{\beta_{j_2}}} du_{\alpha_{\ell}}    \bigg)  \\
&\times \prod_{j_2 < \ell \leq |S_{22}| } \bigg(    \int_{-\infty}^0    \hat{F}_{\beta_\ell} (-u_{\beta_\ell}) e^{-2 \pi u_{\beta_\ell} z_{\alpha_{j_1}}} du_{\beta_\ell} - \int_{-\infty}^0    \hat{F}_{\beta_\ell} (-u_{\beta_\ell}) e^{-2 \pi u_{\beta_\ell} z_{\beta_{j_2}}} du_{\beta_\ell} \bigg)  dz_{\beta_{j_2} }   dz_{\alpha_{j_1}} . 
\end{align*}
Note that there is no $u_{\beta_{j_2}}$-integral above. After we expand the products and combine the $u_j$-integrals together, the above equals
\begin{align*}
  &    \sum_{ \substack{ j_1 \leq | S_{12}| \\ j_2 \leq | S_{22}| }}   \frac{  1}{(  2 \pi    )^2 }   \int_{   (-\delta_{j_1}) }  \int_{ (\delta_{j_2} )  }        \frac{  e^{  2 \pi  ( z_{\alpha_{j_1}} - z_{\beta_{j_2}} ) }           }{   ( z_{\alpha_{j_1}}  - z_{\beta_{j_2}}   )^2 }      F_{\beta_{j_2}}(i z_{\beta_{j_2}}   )  \summany_{\substack{T_1,W_1, T_2, W_2, T_3, W_3 \\ T_1 + W_1 = \{ \alpha_1,.., \alpha_{j_1 -1} \} \cup \{\beta_1,..., \beta_{  j_2 -1} \} \\ T_2 + W_2 = \{\alpha_{ j_1 + 1},.., \alpha_{|S_{12}|} \} \\ T_3 + W_3 = \{\beta_{j_2  + 1} , ..., \beta_{|S_{22}|}  \} }} (-1)^{ j_1 + |S_{12}|   +|W_2 |+| W_3|}    \\ 
&      \times  \int_{   \mc D_{|S_{12}|+|S_{22}| }(\vec{T}, \vec{W} : \beta_{j_2})    } \left( \prod_{j \in S_{12}\cup S_{22} \setminus \{   \beta_{j_2}     \} }  \widehat F_j(-u_j) \right)   e^{ - 2 \pi ( u_{\alpha_{j_1}}  + u(\vec{T}))z_{\alpha_{j_1}} - 2 \pi     u(\vec{W}) z_{\beta_{j_2}}} d \tilde{\vecu}  dz_{\beta_{j_2} } dz_{\alpha_{j_1}}    , 
\end{align*}
where $ \mc D_{|S_{12}|+|S_{22}| }(\vec{T}, \vec{W} : \beta_{j_2})$ is defined analogously to $ \mc D_{|S_{12}|+|S_{22}| }(\vec{T}, \vec{W}  )$ but without the $u_{\beta_{j_2}}$-coordinate. By changing the order of integrals, we see that
\begin{align*}
 \mc{J} ( S_{12}, S_{22})   =  &    \sum_{ \substack{ j_1 \leq | S_{12}| \\ j_2 \leq | S_{22}| }}  \frac{  1}{   2 \pi  i } \int_{ (\delta_{j_2} )  }          F_{\beta_{j_2}}(i z_{\beta_{j_2}}   )  \summany_{\substack{T_1,W_1, T_2, W_2, T_3, W_3 \\ T_1 + W_1 = \{ \alpha_1,.., \alpha_{j_1 -1} \} \cup \{\beta_1,..., \beta_{  j_2 -1} \} \\ T_2 + W_2 = \{\alpha_{ j_1 + 1},.., \alpha_{|S_{12}|} \} \\ T_3 + W_3 = \{\beta_{j_2  + 1} , ..., \beta_{|S_{22}|}  \} }} (-1)^{ j_1 + |S_{12}|   +|W_2 |+| W_3|}    \\ 
&    \times     \int_{   \mc D_{|S_{12}|+|S_{22}| }(\vec{T}, \vec{W} : \beta_{j_2})    } \left( \prod_{j \in S_{12}\cup S_{22} \setminus \{   \beta_{j_2}     \} }  \widehat F_j(-u_j) \right)    \\
& \times  \frac{  -1}{   2 \pi   i }   \int_{   (-\delta_{j_1}) }         \frac{    e^{ - 2 \pi ( u_{\alpha_{j_1}}  + u(\vec{T})-1 )z_{\alpha_{j_1}} - 2 \pi     ( u(\vec{W})+1)  z_{\beta_{j_2}}}        }{   ( z_{\alpha_{j_1}}  - z_{\beta_{j_2}}   )^2 }  dz_{\alpha_{j_1}} d \tilde{\vecu}  dz_{\beta_{j_2} }     . 
\end{align*}
The last integral is nonzero only when $   u_{\alpha_{j_1}}  + u(\vec{T})>1   $. In such a case, we shift the $z_{\alpha_{j_1}}$-integral to $\infty $ and obtain that the above equals 
\begin{align*}
   &    \sum_{ \substack{ j_1 \leq | S_{12}| \\ j_2 \leq | S_{22}| }}  \frac{  1}{     i } \int_{ (\delta_{j_2} )  }          F_{\beta_{j_2}}(i z_{\beta_{j_2}}   )  \summany_{\substack{T_1,W_1, T_2, W_2, T_3, W_3 \\ T_1 + W_1 = \{ \alpha_1,.., \alpha_{j_1 -1} \} \cup \{\beta_1,..., \beta_{  j_2 -1} \} \\ T_2 + W_2 = \{\alpha_{ j_1 + 1},.., \alpha_{|S_{12}|} \} \\ T_3 + W_3 = \{\beta_{j_2  + 1} , ..., \beta_{|S_{22}|}  \} }} (-1)^{ j_1 + |S_{12}|   +|W_2 |+| W_3|}    \\ 
&    \times     \int_{ \substack{  \mc D_{|S_{12}|+|S_{22}| }(\vec{T}, \vec{W} : \beta_{j_2}) \\  u_{\alpha_{j_1}}  + u(\vec{T})>1  }     } \left( \prod_{j \in S_{12}\cup S_{22} \setminus \{   \beta_{j_2}     \} }  \widehat F_j(-u_j) \right)      ( 1-  u_{\alpha_{j_1}}  - u(\vec{T})  )    e^{ - 2 \pi ( u_{\alpha_{j_1}}  + u(\vec{T}) + u(\vec{W}) )z_{\beta_{j_2}}  }        d \tilde{\vecu}  dz_{\beta_{j_2} }     . 
\end{align*}
We then interchange the order of integrals again and replace the $z_{\beta_{j_2}}$-integral by $\hat{F}_{\beta_{j_2}}$. Thus $ \mc{J} ( S_{12}, S_{22}) $ equals  
\begin{align*}
   &    \sum_{ \substack{ j_1 \leq | S_{12}| \\ j_2 \leq | S_{22}| }}         \summany_{\substack{T_1,W_1, T_2, W_2, T_3, W_3 \\ T_1 + W_1 = \{ \alpha_1,.., \alpha_{j_1 -1} \} \cup \{\beta_1,..., \beta_{  j_2 -1} \} \\ T_2 + W_2 = \{\alpha_{ j_1 + 1},.., \alpha_{|S_{12}|} \} \\ T_3 + W_3 = \{\beta_{j_2  + 1} , ..., \beta_{|S_{22}|}  \} }} (-1)^{ j_1 + |S_{12}|   +|W_2 |+| W_3|}    \\ 
&   \times      \int_{ \substack{  \mc D_{|S_{12}|+|S_{22}| }(\vec{T}, \vec{W} : \beta_{j_2}) \\  u_{\alpha_{j_1}}  + u(\vec{T})>1  }     } \left( \prod_{j \in S_{12}\cup S_{22} \setminus \{   \beta_{j_2}     \} }  \widehat F_j(-u_j) \right)         ( 1-  u_{\alpha_{j_1}}  - u(\vec{T})  )  \hat{ F}_{\beta_{j_2}} (   u_{\alpha_{j_1}}  + u(\vec{T}) + u(\vec{W}) )        d \tilde{\vecu}           \\ 
 =  &    \sum_{ \substack{ j_1 \leq | S_{12}| \\ j_2 \leq | S_{22}| }}         \summany_{\substack{T_1,W_1, T_2, W_2, T_3, W_3 \\ T_1 + W_1 = \{ \alpha_1,.., \alpha_{j_1 -1} \} \cup \{\beta_1,..., \beta_{  j_2 -1} \} \\ T_2 + W_2 = \{\alpha_{ j_1 + 1},.., \alpha_{|S_{12}|} \} \\ T_3 + W_3 = \{\beta_{j_2  + 1} , ..., \beta_{|S_{22}|}  \} }} (-1)^{ j_1 + |S_{12}|   +|W_2 |+| W_3|}    \\ 
&   \times     \int_{ \substack{  \mc D_{|S_{12}|+|S_{22}| }(\vec{T}, \vec{W}  ) \\  u_{\alpha_{j_1}}  + u(\vec{T})>1  }     } \left( \prod_{j \in S_{12}\cup S_{22}   }  \widehat F_j(-u_j) \right)         ( 1-  u_{\alpha_{j_1}}  - u(\vec{T})  )  \delta (    u(S_{12}) + u(S_ {22}) )        d  \vecu,        
\end{align*}
which is the same as  $(-1)^{|S_{12}|+|S_{22}|} \ \mc{I} ( S_{12}, S_{22}) $, finishing the lemma. 

\end{proof}

\subsection{Conclusion of the proof of Proposition \ref{sn:eq1} }
From Equation (\ref{eqn:C1GNG}) and Lemma \ref{lem:J0}, 
\est{ \lim_{Q \rightarrow \infty}	\frac{C_{1,\underline{G}}}{D(\W, Q)}  =   \mc R_0  +  \lim_{Q \rightarrow \infty} \frac{ N_{\underline{G}} }{D(\W, Q)} .}
Then from Equations \eqref{eqn:NGasymp}, (\ref{partitionJ1}) and Lemma \ref{lem:J1}, we obtain that
\est{\lim_{Q \rightarrow \infty} \frac{ N_{\underline{G}} }{D(\W, Q)} = \mc R_1. }
Thus by Equation (\ref{def:R}), we derive at
\est{ \lim_{Q \rightarrow \infty}	\frac{C_{1,\underline{G}}}{D(\W, Q)} = \mc R_0 + \mc R_1 = \mc R}
as desired. 

\section{Acknowledgement}
This work was initially suggested by Brian Conrey during the Arithmetic Statistics MRC program at Snowbird. We would like to thank him for his guidance and provide us useful materials. Also we would also like to thank Xiannan Li for helpful editorial comments, and Sheng-Chi Liu and Maksym Raziwi\l\l \ for discussion during this project. Part of this work was done while the first author was in residence at the Mathematical Sciences Research Institute (MSRI) in Berkeley, California, during the Spring semester of year 2017, supported in part by the National Science Foundation (NSF) under Grant No. DMS-1440140. She also would like to acknowledge support from AMS-Simons Travel grant. The second author has been supported by the National Research Foundation of Korea (NRF) grant funded by the Korea government(MSIP)
(No. 2016R1C1B1008405).  Lastly we are grateful to the anonymous referees, whose comments and suggestions were most helpful in improving the paper.

\end{document}